\documentclass[a4paper,english,10pt,oneside]{article}
\usepackage[english]{babel}
\usepackage{amsmath}
\usepackage{amssymb}
\usepackage{wasysym}

\usepackage{tikz}
\usepackage{pgfplots}

\usepackage{graphicx}
\usepackage{mathtools}
\usepackage{amsthm}
\usepackage{booktabs}
\usepackage{xspace}
\usepackage[hidelinks]{hyperref}
\usepackage[font=small]{caption}

\usepackage{subcaption}
\usepackage{todonotes}

\usepackage{thmtools}
\usepackage{thm-restate}

\usepackage{chngcntr}
\usepackage{apptools}
\AtAppendix{\numberwithin{lemma}{section}}
\AtAppendix{\numberwithin{proposition}{section}}
\usepackage{cleveref}

\graphicspath{{pictures/}}

\definecolor{darkblue}{rgb}{0,0,.5}
\DeclareMathOperator{\proj}{proj} % projection
\DeclareMathOperator{\dist}{dist} % distance
\DeclareMathOperator{\rec}{rec} % recession cone
\DeclareMathOperator{\conv}{conv} % convex hull
\DeclareMathOperator{\cl}{cl} % closure
\DeclareMathOperator{\inte}{int} % interior
\DeclareMathOperator{\relinte}{ri} % relative interior
\newcommand{\RR}{\mathbb{R}} % set of reals
\newcommand{\cQ}{\mathcal{Q}} % a generic quadratic set
\newcommand{\slp}{\bar{s}} % lp solution/point to separate
\newcommand{\xlp}{\bar{x}} % lp solution/point to separate
\newcommand{\ylp}{\bar{y}} % lp solution/point to separate
\newcommand{\stcomp}[1]{\left(#1\right)^c} %complement

\newcommand{\st}{\,:\,}
\newcommand{\phifun}{\phi_{\lambda}} % phi function that gives maximal for the bad case
\newcommand{\Hzero}{H_0} % hyperplane a^T x + d^T y = 0
\newcommand{\HH}{H} % hyperplane a^T x + d^T y = -1
\newcommand{\xbeta}{{x_\beta}} % point exposing inequalities for bad case
\newcommand{\xbetaO}{x_{\beta_0}} % point exposing inequality of beta_0 for bad case
\newcommand{\maxhomobg}{C_{G(\lambda)}} % maximal S-free set for homogeneous with bound good case
\newcommand{\maxhomo}{C_{\lambda}} % maximal S-free set for homogeneous
\newcommand{\maxnonhomob}{C} % maximal S-free non homo bad
\newcommand{\shomob}{S_{\leq 0}} % S homo with bound
\newcommand{\maxhomobb}{C_{\phifun}} % maximal S-free set for homogeneous with bound bad case
\newcommand{\shomo}{S^h} % S homo
\newcommand{\snonhomo}{S} % S non homo

\newcommand{\shomorunning}{S^2_{\leq 0}} % S homo in running example
\usepackage{xifthen}
\ifthenelse{\isundefined{\T}}{%
  \newcommand{\T}{\mathsf{T}}
  }{%
  \renewcommand{\T}{\mathsf{T}}
}

\theoremstyle{plain}
\newtheorem{theorem}{Theorem}
\newtheorem{proposition}{Proposition}
\newtheorem{lemma}{Lemma}
\newtheorem{corollary}{Corollary}
\theoremstyle{definition}
\newtheorem{definition}{Definition}
\declaretheoremstyle[
  qed=\qedsymbol
]{examplestyle}
\declaretheorem[style=examplestyle]{example}
\declaretheorem[style=examplestyle]{remark}

\usepackage{zibtitlepage}

\ZTPAuthor{\ZTPHasOrcid{Gonzalo Mu\~noz}{0000-0002-9003-441X},\ZTPHasOrcid{Felipe Serrano}{0000-0002-7892-3951}}
\ZTPTitle{\bf Maximal quadratic-free sets}

\ZTPInfo{The described research activities are funded by the German Federal Ministry for
Economic Affairs and Energy within the project EnBA-M (ID: 03ET1549D).
The work for this article has been (partly) conducted within the Research Campus
MODAL funded by the German Federal Ministry of Education and Research (BMBF
grant number 05M14ZAM).
}

\ZTPNumber{19-56}
\ZTPMonth{11}
\ZTPYear{2019}

\title{\bf Maximal quadratic-free sets}
\author{
  Gonzalo Mu\~noz\thanks{Universidad de O'Higgins, Rancagua, Chile, \texttt{gonzalo.munoz@uoh.cl}}~,
  Felipe Serrano\thanks{Zuse Institute Berlin, Takustr.~7, 14195~Berlin,
  Germany, \texttt{serrano@zib.de}}
}

\begin{document}
\zibtitlepage
\maketitle

\begin{abstract}
  The intersection cut paradigm is a powerful framework that facilitates the generation of valid linear inequalities, or cutting planes,
  for a potentially complex set $S$. The key ingredients in this construction are a simplicial conic relaxation of $S$ and an $S$-free set: a convex zone whose interior does not
  intersect $S$. Ideally, such $S$-free set would be maximal inclusion-wise, as it would generate a deeper cutting plane. However, maximality can be a challenging goal in general. In this work, we show how to construct maximal $S$-free sets when $S$ is defined as a general quadratic inequality. Our maximal
  $S$-free sets are such that efficient separation of a vertex in LP-based approaches to quadratically constrained problems is guaranteed. To the best of our knowledge, this work is the first to provide maximal quadratic-free sets. 
\end{abstract}

\section{Introduction}

Cutting planes have been at the core of the development of
tractable computational techniques for integer-programming for decades. Their
rich theory and remarkable empirical performance have constantly caught the
attention of the optimization community, and has recently seen renewed efforts
on their extensions to the non-linear setting. 

Consider a generic optimization problem, which we assume to have linear objective without loss of generality:
% \begin{subequations} \label{eq:generalqp}
% \begin{align}
%   \min \quad & c^\T x \\
%   \text{s.t.} \quad & x^\T Q_i x + d^Tx + b \leq 0 \quad i=1,\ldots, m\\
%    & x\in \mathbb{R}^n.
% \end{align}
% \end{subequations}
\begin{subequations} \label{eq:generaloptimization}
  \begin{align}
    \min \quad & c^\T x \\
    \text{s.t.} \quad & x\in S \subseteq \mathbb{R}^n.
  \end{align}
  \end{subequations}
A common framework for finding strong approximations to this problem is to first
find $\xlp$, an extreme point optimal solution of an LP relaxation of
\eqref{eq:generaloptimization}, and check if $\xlp \in S$.
If so, then \eqref{eq:generaloptimization} is solved. Otherwise, we try
to find a cutting plane: an inequality separating $\xlp$ from $S$.
Such inequality can be used to refine the LP relaxation of
\eqref{eq:generaloptimization}.

One way of finding such a cutting plane is through the \emph{intersection
cut}~\cite{Tuy1964,Balas1971,Glover1973} framework.
We refer the reader to~\cite{ConfortiCornuejolsZambelli2011a} for the necessary
background on this procedure.
For the purposes of this article, it suffices to know that for an intersection
cut to be computed we need $\xlp \not\in S$ as above, a simplicial conic relaxation of $S$ with
apex $\xlp$, and an $S$-free set $C$ ---a convex set satisfying \( \inte(C) \cap
S = \emptyset \)--- such that $\xlp \in \inte(C)$.
In this work we assume that $\xlp$ and the simplicial cone are given and focus
on the construction of the $S$-free sets. 

A particularly important case is obtained when \eqref{eq:generaloptimization} is
a quadratic problem, that is, 
\[
  S = \{ x\in \mathbb{R}^n \st x^\T Q_i x + b_i^\T x + c_i \leq 0,\, i=1,\ldots, m \}
\]
for certain $n\times n$ matrices $Q_i$, not necessarily positive semi-definite.
Note that if $\xlp \not\in S$, there exists $i\in \{1,\ldots, m\}$ such that
\[
  \xlp \not\in S_i \coloneqq \{ x\in \mathbb{R}^n \st x^\T Q_i x + b_i^\T
  x + c_i \leq 0\},
\]
and constructing an $S_i$-free set containing $\xlp$ would suffice to ensure
separation.
Thus, slightly abusing notation, given $\xlp$ we focus on a systematic way of
constructing $S$-free sets containing $\xlp$, where $S$ is defined using
a single quadratic inequality:
\[
  S = \{ x\in \mathbb{R}^n \st x^\T Q x + b^\T x + c \leq 0\}.
\]
As a final note, if we consider the simplest form of intersection cuts, where the cuts are computed using the intersection points of the $S$-free set and the extreme rays of the simplicial conic relaxation of $S$ (i.e., using the gauge), then the largest the $S$-free set the better. In other words, if two $S$-free sets $C_1,C_2$ are such that $C_1 \subsetneq C_2$, the intersection cut derived from $C_2$ is stronger than the one derived from $C_1$~\cite{ConfortiCornuejolsDaniilidisLemarechalMalick2015}.
Therefore, we aim at computing \emph{maximal} $S$-free sets.

\subsection{Literature review}
The history of intersection cuts and $S$-free sets dates back to the 60's.
%
%Here we review some landmarks without attempting to being complete.
%
The most basic form of an intersection cut is given by the unique hyperplane that goes
through the intersection points between the rays of the simplicial cone and the
boundary of the $S$-free set.
They were originally introduced in the nonlinear setting by Tuy~\cite{Tuy1964}
for the problem of minimizing a concave function over a polytope.
Later on, they were introduced in integer programming by Balas~\cite{Balas1971} and have been largely studied since.
The more modern form of intersection cuts deduced from an arbitrary convex
$S$-free set is due to Glover~\cite{Glover1973}, although the term $S$-free was
coined by Dey and Wolsey~\cite{DeyWolsey2010}.

Even though the origin of intersection cuts was in nonlinear optimization, most of the developments have been in the mixed-integer linear programming literature. See e.g. \cite{ConfortiCornuejolsDaniilidisLemarechalMalick2015,cornuejols2013sufficiency,basu2010minimal} for in-depth analyses of the relation of intersection cuts using maximal $\mathbb{Z}^n$-free sets and the generation of facets of $\conv(S)$. We also refer the reader to \cite{Andersen2007,Basu2010,andersen2010analysis,Borozan2009,dey2008lifting,Gomory1972} and references therein. For extensions of this approach to the mixed-integer conic case, we refer the reader to \cite{andersen2013intersection,kilincc2014minimal,modaresi2015split,modaresi2015intersection}. Recently, Towle and Luedtke \cite{towle2019intersection} proposed a method for constructing valid cutting planes with a similar approach to intersection cuts, but allowing $\xlp$ to not be in the $S$-free set.

Lately, there has been several developments of intersection cuts in a non-linear setting.
Fischetti et al.~\cite{FischettiLjubicMonaciSinnl2016} applied intersection cuts
to bilevel optimization.
Bienstock et al.~\cite{bienstock2016outer,Bienstock_2019} studied outer-product-free
sets; these can be used for generating intersection cuts for polynomial optimization when using an extended formulation.
Serrano~\cite{Serrano2019} showed how to construct a concave underestimator of
any factorable function and from them one can build intersection cuts for
factorable mixed integer non-linear programs.
Fischetti and Monaci~\cite{Fischetti_2019} constructed bilinear-free sets through a bound disjunction and, in each term of the disjunction, underestimating
the bilinear term with McCormick inequalities~\cite{McCormick1976}.
The complement of this disjunction is the bilinear-free set.

Of all these approaches for constructing intersection cuts in a non-linear
setting, the only one that ensures maximality of the corresponding 
$S$-free sets
is the work of Bienstock et al.~\cite{bienstock2016outer,Bienstock_2019}. While their approach can also be used to generate cutting planes in our setting (general quadratic inequalities), the definition of $S$ differs: Bienstock et al. use a moment-based extended formulation of polynomial optimization problems \cite{shor_quadratic_1987,lasserre2001global,laurent2009sums} and from there define $S$ as the set of matrices which are positive semi-definite and of rank 1, which the authors refer to as \emph{outer-products}. Maximality is computed with respect to this notion. It is unclear if a maximal \emph{outer-product}-free set can be converted into a maximal quadratic-free set, or vice-versa. There is an even more fundamental difference that makes these approaches incomparable at this point: in a quadratic setting, the approach of Bienstock et al. would compute a cutting plane in extended space of dimension proportional to $n^2$, whereas our approach can construct a maximal $S$-free set in the original space. The quadratic dimension increase can be a drawback in some applications, however stronger cuts can be derived from extended formulations in some cases \cite{bodur2017cutting}. A thorough comparison of these approaches is subject of future work.

We refer to the survey~\cite{BonamiLinderothLodi2011} and the references therein for other efforts of extending cutting planes to the nonlinear setting.

\subsection{Contribution}
The main contribution of this paper is an explicit construction of maximal $S$-free sets, when $S$ is defined using a non-convex quadratic inequality (\Cref{thm:max_nonhomo_good} and \Cref{thm:max_nonhomo_bad}). We achieve this by relying on the fact that any quadratic inequality can be represented using a homogeneous quadratic inequality intersected with a linear \emph{equality}. While these maximal $S$-free sets are constructed using semi-infinite representations, we show equivalent closed-form representations of them.

In order to construct these sets, we also derive maximal $S$-free sets for sets $S$ defined as the intersection of a homogeneous quadratic inequality intersected with a linear \emph{homogeneous inequality}. These are an important intermediate step in our construction, but they are of independent interest as well.

In order to show our results, we state and prove a criterion for maximality of $S$-free sets which
generalizes a criterion proven by Dey and Wolsey (the `only if'
of~\cite[Proposition A.4]{DeyWolsey2010}) in the case of maximal
\emph{lattice-free} sets (\Cref{def:expose} and \Cref{thm:exposed_maximality}).
We also develop a new criterion that can handle a special phenomenon that arises
in our setting and also in non-linear integer programming: the boundary of
a maximal $S$-free set may not even intersect $S$.
Instead, the intersection might be ``at infinity''.
We formalize this in \Cref{def:expose_at_infinity} and show the criterion in
\Cref{thm:mixed_exposed_maximality_wrt}.

\subsection{Notation}
We mostly follow standard notation. We denote $\mathbb{R}^n$ the space of
$n$-dimensional vectors with real entries.
$\| \cdot \| $ is the euclidean norm in $\mathbb{R}^n$ and given a positive
definite matrix $A$, we denote $\|\cdot \|_A$ the norm defined by $\sqrt{x^\T
A x}$. $B_r(x)$ and $D_r(x)$ denote the euclidean ball centered at $x$ of radius
$r$ and its boundary, respectively, i.e., $B_r(x) = \{y\in \mathbb{R}^n \,: \,
\| y - x \| \leq r \}$ and $D_r(x) = \{y\in \mathbb{R}^n \,: \, \| y - x \|
= r \}$.
Given a vector $v$ and a set $C$, we denote the distance between $v$ and $C$ as
$\dist(v, C) = \inf_{x \in C} \|v - x\|$.
We denote the set $\{v + x \st x \in C\}$ as $v + C$.
We denote the transpose operator as $(\cdot)^\T$.
For a set of vectors $\{v^1, \ldots, v^k\}\subseteq \mathbb{R}^n$, we denote
$\langle v^1,\ldots, v^k \rangle$ the subspace generated by them.

Given a set $C\subseteq \mathbb{R}^n$ and a subspace $H$ of $\mathbb{R}^n$, we denote $\proj_H C$ the projection of $C$ onto $H$. $\stcomp{\cdot}$, $\conv(\cdot), \inte(\cdot)$, $\relinte(\cdot)$ and $\rec(\cdot)$ denotes the complement, convex hull, interior, relative interior and recession cone of a set, respectively.

Perhaps the least standard notation we use is denoting an inequality $\alpha^\T
x \leq \beta$ by $(\alpha, \beta)$.
If $\beta = 0$ we denote it as well as $\alpha$.
This is based on the fact that in the polar of a convex set ---roughly, the set
of all valid inequalities--- the inequalities are points and, although we do not use any polarity results, many of the ideas in this work were originally
developed from looking at the polar.

\subsection{Outline}
The rest of the paper is organized as follows.
In \Cref{sec:preliminaries} we introduce some definitions and criteria
to prove maximality of $S$-free sets.
In particular, we define \emph{exposing points} and \emph{exposing point at
infinity} and show that if $C$ is an $S$-free set whose defining inequalities
are exposed or exposed at infinity, then $C$ is maximal.
In \Cref{sec:maximalhomo} we show how to construct maximal $S$-free sets
when $S$ is defined by a \emph{homogeneous} quadratic function.
\Cref{sec:homowithhomo} presents the construction of maximal $S$-free
sets when $S$ is defined by a \emph{homogeneous} quadratic function and
a homogeneous linear inequality constraints.
The construction of a maximal $S$-free set when $S$ is the sublevel set of any
quadratic function is presented in Section~\ref{sec:nonhomo}.
Our constructions depend on a ``canonical'' representation of the set $S$.
The effects of this representation are discussed in
\Cref{sec:transformations}.
Finally, some conclusions and directions for future work can be found in
\Cref{sec:conclusions}.

Some proofs of technical results, or results that would negatively affect the flow of the
document can be found in an appendix.

\section{Preliminaries} \label{sec:preliminaries}
In this section we collect definitions and results that are going to be useful
later on the paper.
As we mentioned above, our main object of study is the set $S = \{ x \in
\mathbb{R}^p \st q(x) \leq 0 \} \subseteq \mathbb{R}^p$, where $q$ is
a quadratic function.
To make the analysis easier, we can work on $\mathbb{R}^{p+1}$ and consider the cone generated by $S \times \{1\}$, namely, $\{ (x,z) \in \mathbb{R}^{p + 1} \st z^2 q(\frac{x}{z}) \leq 0,\ z \geq 0\}$.
To recover the original $S$, however, we must intersect the cone with $z
= 1$.
Since we are interested in maximal $S$-free sets, this motivates the following
definition, see also \cite{Basu2010}.
\begin{definition}
  Given $S, C, H \subseteq \mathbb{R}^{n}$ where $S$ is closed, $C$ is closed
  and convex and $H$ is an affine hyperplane, we say that $C$ is 
  \emph{$S$-free with respect to $H$} if $C \cap H$ is $S \cap H$-free
  w.r.t the induced topology in $H$. We say $C$ is \emph{maximal
  $S$-free with respect to $H$}, if for any $C'\supseteq C$ that is 
  $S$-free with respect to $H$ it holds that $C' \cap H \subseteq C\cap H$.
\end{definition}

\subsection{Techniques for proving maximality}
In this section we describe some sufficient conditions to prove that a convex set $C$ is
maximal $S$-free which will be used in the paper.

A sufficient (and necessary) condition for a full dimensional convex $C$
lattice-free (that is, $S = \mathbb{Z}^n$) set to be maximal is that $C$ is
a polyhedron and there is a point of $\mathbb{Z}^n$ in the relative interior of
each of its facets~\cite[Theorem 6.18]{ConfortiCornuejolsZambelli2014}.
More generally, if $C$ is a full dimensional $S$-free polyhedron such that there
is a point of $S$ in the relative interior of each facet, then $C$ is maximal.
The problem with extending this property to non-polyhedral maximal $S$-free sets
is that they might not even have facets, e.g., if $S$ is the complement of
$\inte B_1(0)$ and $C$ is $B_1(0)$ in dimension 3 or higher.
The motivation of the next definition is to capture the property of a facet that is key for proving maximality.
\begin{definition} \label{def:expose}
  Given a convex set $C \subseteq \mathbb{R}^n$ and a valid
  inequality $\alpha^\T x \leq \beta$, we say that a point $x_0 \in \mathbb{R}^n$
  \emph{exposes} $(\alpha, \beta)$ with respect to $C$ or that $(\alpha, \beta)$ is \emph{exposed}
  by $x_0$ if
  \begin{itemize}
    \item $\alpha^\T x_0 = \beta$ and,
    \item if $\gamma^\T x \leq \delta$ is any other non-trivial valid inequality for $C$
      such that $\gamma^\T x_0 = \delta$, then there exists a $\mu > 0$ such
      that $\gamma = \mu \alpha$ and $\beta = \mu \delta$.
  \end{itemize}
  In some cases we omit saying ``with respect to $C$'' if it is clear from context. 
\end{definition}
To get some intuition, if $C$ is a polyhedron and $x \in C$
exposes an inequality, then that inequality is a facet and $x$ is in the
relative interior of the facet.

\begin{remark}
It is very important to note that if there exists a point exposing a valid
inequality of $C$, then $C$ is full dimensional. The reader should keep this
in mind throughout the whole paper.
\end{remark}

\begin{remark}
  For some convex $C$, a point $x \notin C$ can expose a valid
  inequality of $C$.
  For instance, consider $C = \{ x \in \mathbb{R}^2 \st x_1 + x_2 \geq 1 \}$.
  Then $(0,0) \notin C$ and exposes $x_1 + x_2 \geq 0$.

  The name ``exposing'' comes from the concept of \emph{exposed} point.
  To simplify ideas, let us assume that $0 \in \inte(C)$.
  A point $x_0 \in C$ is exposed if there exists a valid inequality of $C$,
  $\alpha^\T x \leq 1$, such that $\{x \in C \st \alpha^\T x = 1\} = \{x_0\}$.
  If $\alpha_0$ is an exposed point of the polar of $C$, $C^\circ = \{\alpha \st
  \alpha^\T x \leq 1,\ \forall x \in C\}$, then there is a valid inequality,
  $x_0^\T \alpha \leq 1$, such that $\{\alpha \in C^\circ \st x_0^\T \alpha
  = 1\} = \{\alpha_0\}$.
  In other words, if $\alpha^\T x \leq 1$ is valid for $C$ (i.e. $\alpha \in
  C^\circ$) and $\alpha^\T x_0 = 1$, then $\alpha = \alpha_0$.
  We see that $x_0$ is a point (direction) that shows that $\alpha_0$ is an
  exposed inequality, or, that $x_0$ \emph{exposes} $\alpha_0$.
\end{remark}
We now show that our definition is indeed helpful to show maximality.
\begin{theorem} \label{thm:exposed_enters}
  Let $K, K' \subseteq \mathbb{R}^n$ be convex sets such that $K \subseteq K'$.
  If $\alpha^\T x \leq \beta$ is
  \begin{itemize}
    \item valid for $K$,
    \item not valid for $K'$, and
    \item exposed by $x_0 \in K$ with respect to $K$,
  \end{itemize}
  then $x_0\in \inte(K')$.
\end{theorem}
\begin{proof}
  As $x_0 \in K$ exposes $\alpha^\T x \leq \beta$, it holds that $\alpha^\T x_0
  = \beta$ and, thus, $x_0$ is in the boundary of $K$.
  Suppose $x_0$ is not in the interior of $K'$.
  Then it must be in the boundary of $K'$ and there is a valid inequality for
  $K'$, $\gamma^\T x \leq \delta$, such that $\gamma^\T x_0 = \delta$.

  As $K \subsetneq K'$, $\gamma^\T x \leq \delta$ is also valid for $K$.
  Given that $(\gamma, \delta)$ is tight at $x_0$ and $x_0$ exposes $(\alpha,
  \beta)$, we conclude that there is a $\mu > 0$ such that $\gamma = \mu \alpha$
  and $\beta = \mu \delta$.
  However, since $\alpha^\T x \leq \beta$ is not valid for $K'$, it follows that
  $\gamma^\T x \leq \delta$ cannot be valid for $K'$.
  This contradiction proves the claim.
\end{proof}

\begin{theorem} \label{thm:exposed_maximality}
  Let $S \subseteq \mathbb{R}^n$ be a closed set and $C \subseteq \mathbb{R}^n$
  a convex $S$-free set.
  Assume that $C = \{ x \in \mathbb{R}^n \st \alpha^T x \leq \beta, \forall
  (\alpha, \beta) \in \Gamma\}$ and that for every $(\alpha, \beta)$
  there is an $x \in S \cap C$ that exposes $(\alpha, \beta)$.
  Then, $C$ is maximal $S$-free.
\end{theorem}

\begin{proof}
  To show that $C$ is maximal we are going to show that for every $\xlp \notin
  C$, $S \cap \inte(\conv(C \cup \{\xlp \}))$ is nonempty.

  Let $\xlp \notin C$ and let $(\alpha, \beta) \in \Gamma$ be a separating
  inequality, i.e., $\alpha^\T \xlp > \beta$. Let $C' = \conv(C \cup \{\xlp\})$.

  By hypothesis, there is an $x_0 \in S \cap C$ that exposes $(\alpha, \beta)$.
  Since $(\alpha, \beta)$ is valid for $C$ and not for $C'$,
  \Cref{thm:exposed_enters} implies that  $x_0 \in \inte(C')$.
\end{proof}

With minor modifications one can also get the following sufficient condition for
maximality with respect to a hyperplane.
\begin{theorem} \label{thm:exposed_maximality_wrt}
  Let $S \subseteq \mathbb{R}^n$ be a closed set, $H$ be an affine hyperplane,
  and $C \subseteq \mathbb{R}^n$ be a convex $S$-free set.
  Assume that $C = \{ x \in \mathbb{R}^n \st \alpha^T x \leq \beta, \forall
  (\alpha, \beta) \in \Gamma\}$ and that for every $(\alpha, \beta)$
  there is an $x \in S \cap C \cap H$ that exposes $(\alpha, \beta)$.
  Then, $C$ is maximal $S$-free with respect to $H$.
\end{theorem}
\begin{remark}
  Points that expose inequalities are also called \emph{smooth} points.
  A \emph{smooth} point of $C$ is a point for which there exists a unique
  supporting hyperplane to $C$ at it~\cite{Goberna2010}.
  Therefore, if $x_0 \in C$, then $x_0$ exposes some valid inequality of $C$, if
  and only if, $x_0$ is a smooth point of $C$.

  A related concept is that of \emph{blocking points}~\cite{Basu2019}.
  However, blocking points need not to be smooth points in general, that is,
  they do not need to expose any inequality.
  As seen in Theorem~\ref{thm:exposed_maximality} we use exposing points to
  determine maximality of a convex $S$-free set.
  Similarly, in the context of lifting~\cite{Conforti2011}, blocking points are
  used to determine maximality of a translated convex cone $S\times
  \mathbb{Z}_+$-free set.
\end{remark}
There is another phenomenon that does not occur when $S = \mathbb{Z}^n$.
If $S$ is a quadratic set, the inequalities of a maximal $S$-free set might not
be exposed by any point of $S$.
For instance, consider $S = \{ (x,y) \in \RR^2 \st x^2 + 1 \leq y^2 \}$.
The boundary of $S$ is a hyperbola with asymptotes $x = \pm y$.
Thus, $C = \{ (x,y) \in \RR^2 \st x \geq |y| \}$ is a maximal $S$-free set,
because its inequalities are asymptotes of $S$, but they are not exposed by
points of $S$.
This phenomenon also occurs when $S = \mathbb{Z}^n \cap K$, with $K$
convex~\cite{Moran2011}.
However, in that case, it also turns out that maximal $S$-free sets are
polyhedral and their constructions rely on the concept of a facet (see for
instance~\cite[Theorem 3.2]{Moran2011}) which we do not have access to in the
general case.
In our case, we extend the definition of what it means for an inequality to be exposed in order to handle a situation like the one above. We do this by interpreting that asymptotes are exposed
``at infinity''.

\begin{definition} \label{def:expose_at_infinity}
  Given a convex set $C \subseteq \mathbb{R}^n$ with non-empty
  recession cone and a valid inequality $\alpha^\T x \leq \beta$, we say that
  a sequence $(x_n)_n \subseteq \mathbb{R}^n$ \emph{exposes} $(\alpha, \beta)$
  \emph{at infinity} with respect to $C$ if
  \begin{itemize}
    \item $\|x_n\| \to \infty$,
    \item $\frac{x_n}{\|x_n\|} \to d \in \rec(C)$,
    \item $d$ exposes $\alpha^\T x \leq 0$ with respect to $\rec(C)$, and
    \item there exists $y$ such that $\alpha^\T y = \beta$ such that
      $\dist(x_n, y + \langle d \rangle) \to 0$.
  \end{itemize}
  As before, we omit saying ``with respect to $C$'' if it is clear from context. 
\end{definition}
Using this definition, we can prove an analogous result to \Cref{thm:exposed_enters} for inequalities exposed at infinity.

\begin{theorem} \label{thm:exposed_at_infinity_enters}
  Let $K, K' \subseteq \mathbb{R}^n$ be convex sets
  % with full dimensional  recession cones 
  such that $K \subseteq K'$.
  If $\alpha^\T x \leq \beta$ is
  \begin{itemize}
    \item valid for $K$,
    \item not valid for $K'$, and
    \item exposed at infinity by $(x_n)_n$ with respect to $K$,
  \end{itemize}
  then there exists a $k$ such that $x_k\in \inte(K')$.
\end{theorem}
\begin{proof}
  Suppose that for all $k$, $x_k$ is not in the interior of $K'$.
  Then, for each $k$ there exists a non-trivial valid inequality for $K'$,
  $\gamma_k^\T x \leq \delta_k$, such that $\gamma_k^\T x_k \geq \delta_k$.
  We can assume without loss of generality that $\|(\gamma_k, \delta_k)\| = 1$.
  Hence, going through a subsequence if necessary, there exist $\gamma \in
  \mathbb{R}^n$ and $\delta \in \mathbb{R}$ such that $\gamma_k \to \gamma$ and
  $\delta_k \to \delta$ when $k \to \infty$ and $\|(\gamma, \delta)\| = 1$.
  Note that the inequality $(\gamma, \delta)$ is valid for $K'$.
  The idea is to show that $(\gamma, \delta)$ defines the same inequality as $(\alpha, \beta)$.

  As $d = \lim_{k \to \infty} \frac{x_k}{\|x_k\|}  \in \rec(K)$ (see Definition~\ref{def:expose_at_infinity}) and $(\gamma,
  \delta)$ is valid for $K' \supseteq K$, then $\gamma^\T x \leq 0$ is valid for
  $\rec(K)$.
  In particular, $\gamma^\T d \leq 0$. On the other hand,
  \[
    \frac{\delta_k}{\|x_k\|} \leq \gamma_k^\T \frac{x_k}{\|x_k\|} \text{ implies
    } 0 \leq \gamma^\T d,
  \]
  We conclude that $\gamma^\T d = 0$.
  As $d$ exposes $\alpha^\T x \leq 0$ with respect to $\rec(K)$ , there exists a $\mu \geq 0$ such that
  $\gamma = \mu \alpha$.
  Note that we cannot conclude that $\mu > 0$ since, at this point, we do not
  know that $(\gamma, \delta)$ is a non-trivial inequality (e.g. it could be
  $0^\T x \leq 1$).

  Let $y$ be such that $\alpha^\T y = \beta$ and $\dist(x_k, y + \langle
  d \rangle) \to 0$, which exists by \Cref{def:expose_at_infinity}.
  Let $w_k = x_k - d^\T x_k d$.
  We have that
  \[
    \dist(x_k, y + \langle d \rangle) = \dist(x_k - y, \langle d \rangle)
    = \|x_k - y - d^\T (x_k - y) d \| = \| w_k - (y - d^\T y d) \|.
  \]
  Thus, $w_k \to y - d^\T y d$ as $k \to \infty$.

  Since each $(\gamma_k, \delta_k)$ is valid for $K'$, $\gamma_k^\T d \leq 0$. Additionally, for large enough $k$ it must hold that $d^\T x_k > 0$.
  Therefore,
  \[
    \delta_k \leq \gamma_k^\T x_k = \gamma_k^\T (d^\T x_k d + w_k) \leq
    \gamma_k^\T w_k.
  \]
  Computing the limit when $k \to \infty$ we get,
  \[
    \delta \leq \mu \alpha^\T (y - d^\T y d) = \mu \alpha^\T y = \mu \beta.
  \]
  If $\mu = 0$, then $\gamma = 0$ and $\delta \leq 0$.
  As $\|(\gamma, \delta)\| = 1$, it follows that $\delta = -1$, which cannot be
  since $(\gamma, \delta)$ is a valid inequality for $K'$ and $K'$ is, by
  hypothesis, non-empty.
  We conclude that $\mu > 0$ and that $\mu \alpha^\T x \leq \mu \beta$ is valid
  for $K'$, which implies that $\alpha^\T x \leq \beta$ is valid for $K'$,
  contradicting the hypothesis of the theorem.
\end{proof}

With the previous results it is straightforward to prove the following generalization of \Cref{thm:exposed_maximality_wrt}.

\begin{theorem} \label{thm:mixed_exposed_maximality_wrt}
  Let $S \subseteq \mathbb{R}^n$ be a closed set, $H$ be an affine hyperplane,
  and $C \subseteq \mathbb{R}^n$ be a convex $S$-free set.
  Assume that $C = \{ x \in \mathbb{R}^n \st \alpha^T x \leq \beta, \forall
  (\alpha, \beta) \in \Gamma\}$ and that for every $(\alpha, \beta)$
  there is, either, an $x \in S \cap C \cap H$ that exposes $(\alpha, \beta)$,
  or sequence $(x_n)_n \subseteq S \cap H$ that exposes $(\alpha, \beta)$ at
  infinity.
  Then, $C$ is maximal $S$-free with respect to $H$.
\end{theorem}

Another useful result for studying maximal $S$-free sets is the following (see
also \cite[Lemma 6.17]{ConfortiCornuejolsZambelli2014}).
It states that in some cases we can project $S$ into a lower dimensional space and
find maximal sets that are free for the projection.
This result is also useful for visualizing higher dimensional $S$-free sets.
\begin{theorem} \label{thm:projection}
  Let $C$ be a full dimensional closed convex cone with lineality space $L$.
  Let $S \subseteq \mathbb{R}^n$ be closed.
  Then, $C$ is maximal $S$-free if and only if $( C\cap L^{\perp} )$ is maximal
  $ \cl(\proj_{L^{\perp}} S  )$-free.
\end{theorem}
\begin{proof}
  \noindent $(\Rightarrow)$
  If $ C\cap L^{\perp}$ is not maximal, let $K \subseteq L^{\perp}$ be a
  $  \cl(\proj_{L^{\perp}} S   )$-free set that contains it.
  Then, $K + L \supsetneq C$.
  Since $C$ is maximal $S$-free, there exists an $x \in S$ such that $x \in
  \inte(K+L) = \inte(K) + \inte(L)$ (\cite[Corollary 6.6.2]{Rockafellar1970}).
  That is, $x = k + \ell$ with $k \in \inte(K)$ and $\ell \in L$.
  Thus, $x - \ell \in K \subseteq L^{\perp}$ which implies that $x - \ell \in
  \proj_{L^{\perp}} S  $ and contradicts the fact that $K$ is
  $\cl(\proj_{L^{\perp}} S)$-free.

  \noindent $(\Leftarrow)$
  By contradiction, suppose that $C$ is not maximal $S$-free and let $K
  \supsetneq C$ be a closed convex $S$-free set. Then $ K \cap \L^{\perp}
  \supsetneq C \cap \L^{\perp}$, which implies that $K \cap \L^{\perp}$ is not
  $\cl(\proj_{L^{\perp}} S  )$-free. This implies that $\exists \tilde{s} \in
  \cl(\proj_{L^{\perp}} S  ) \cap \inte(K \cap \L^{\perp})$. Moreover, we can
  further assume  $\tilde{s} \in \proj_{L^{\perp}} S  \cap \inte(K \cap
  \L^{\perp})$, as any sequence contained in $\proj_{L^{\perp}} S$ converging to
  an element of  $\cl(\proj_{L^{\perp}} S  ) \cap \inte(K \cap \L^{\perp})$ must
  have an element in $ \proj_{L^{\perp}} S  \cap \inte(K \cap \L^{\perp})$.

  By the definition of  orthogonal projection, there must exist   $s \in S$ and
  $\ell \in L$ such that $\tilde{s} = s - \ell$. Thus, we obtain $s - \ell \in
  \inte( K \cap L^{\perp} )$, i.e.  \[ s  \in    \inte( K \cap L^{\perp} ) + L.
  \] Since the lineality space of $K$ must contain $L$,  we conclude $s \in
  \inte(K)$; a contradiction with $K$ being $S$-free.
\end{proof}

\section{Maximal quadratic-free sets for homogeneous quadratics} \label{sec:maximalhomo}
In this section we construct maximal $\shomo$-free sets that contain a vector $\bar x \not\in \shomo$ for $\shomo = \{ x \in
\mathbb{R}^p \st x^\T Q x \leq 0 \}$. This is our building block towards maximality in the general case.
After a change of variable, we can assume that
\begin{align*}
  \shomo
  &= \{ (x,y,z) \in \mathbb{R}^{n+m+l} \st \sum_{i = i}^n x_i^2 - \sum_{i = i}^m
  y^2_i \leq 0 \} \\
  &= \{ (x,y) \in \mathbb{R}^{n+m} \st \sum_{i = i}^n x_i^2 - \sum_{i = i}^m
  y^2_i \leq 0 \} \times \mathbb{R}^l.
\end{align*}

Thus, we will only focus on $\shomo = \{ (x,y) \in \mathbb{R}^{n+m} \st \sum_{i = i}^n x_i^2 - \sum_{i = i}^m y^2_i \leq 0 \}$ and assume we are given $(\bar x, \bar y )$ such that $\|\bar x\|^2 > \| \bar y\|^2$.

\begin{remark}\label{remark:transformationhomo}
  The transformation used to bring $\shomo$ to the last ``diagonal'' form is, in general, not unique. Nonetheless, maximality of the $\shomo$-free sets is preserved, as there always is such transformation that is one-to-one. In \Cref{sec:transformations} we discuss the effect different choices of this transformation have.
\end{remark}

\subsection{Removing strict convexity matters}
A simple way of obtaining an $\shomo$-free set is via a concave underestimator of $f(x,y) = \sum_{i = i}^n x_i^2 - \sum_{i = i}^m
y^2_i = \|x\|^2 - \|y\|^2$ directly. It is not hard to see that such an underestimator, tight at $(\xlp, \ylp)$, is given by $\|\xlp\|^2
+ 2 \|\xlp\|(x - \xlp) - \|y\|^2$.
The concave underestimator yields the $\shomo$-free set $\{ (x,y) \in
\mathbb{R}^{n+m} \st \|\xlp\|^2 + 2 \|\xlp\|(x - \xlp) - \|y\|^2 \geq 0 \}$.
However, simple examples show that such an $\shomo$-free set is not maximal.
\begin{example}\label{ex:simple}
  The case $n=m=1$ with $\xlp = 3$ yields the $\shomo$-free set
  \[ C = \{ (x,y) \in
  \mathbb{R}^{2} \st -9 +  6 x - y^2 \geq 0 \} \]
  In \Cref{fig:simpleexample} we can see that the set is not maximal $\shomo$-free.
\end{example}

\begin{figure}
  \centering
  \includegraphics[scale=0.3]{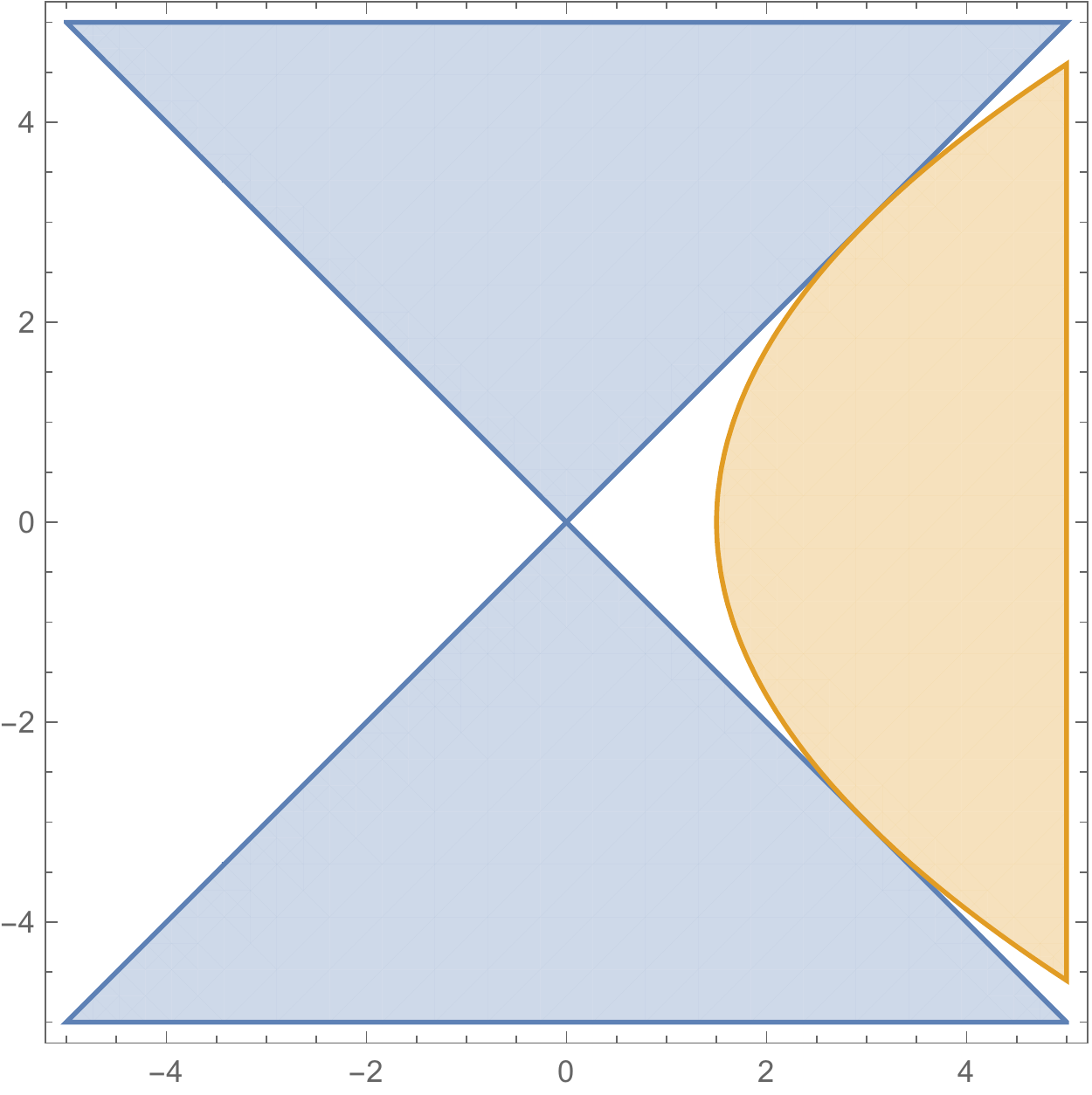}
  \caption{$\shomo$ in \Cref{ex:simple} (blue) and the $\shomo$-free set constructed using a concave underestimator of $\|x\|^2 - \|y\|^2$ (orange).}
  \label{fig:simpleexample}
\end{figure}
The problem seems to be that $\|x\|^2$ is a strictly convex function.
Indeed, suppose $S = \{ x \in \mathbb{R}^n \st f(x) \leq 0 \}$ where $f$ is
strictly convex.
The $S$-free set obtained via a concave underestimator at $\xlp$ is
$C = \{ x \in \mathbb{R}^n \st f(\xlp) + \nabla f(\xlp)(x - \xlp) \geq 0\}$.
It is not hard to see that the strict convexity of $f$ implies that $C$ is not
maximal $S$-free.
The reason is that the linearization of $f$ at $\xlp \notin S$ will not support
the region $S$.
On the other hand, if $f$ is instead sublinear, then any linearization of $f$
will support $S$.

The previous observation motivates the following.
The set $\shomo$ can be equivalently be described by $\shomo = \{ (x,y) \in
\mathbb{R}^{n+m} \st \|x\| - \|y\| \leq 0 \}$.
Now, the function $f(x,y) = \|x\| - \|y\|$ has the following concave
underestimator at $\xlp \neq 0$, $\frac{\xlp^\T x}{\|\xlp\|} - \|y\|$, which
yields the $\shomo$-free set
\begin{equation}\label{eq:clambdadef}
  \maxhomo = \{ (x,y) \in \mathbb{R}^{n+m} \st \lambda^\T x \geq \|y\| \},
\end{equation}
where $\lambda = \frac{\xlp}{\|\xlp\|}$. This set turns out to be maximal, even if we consider any other $\lambda \in D_1(0)$.

\subsection{Maximal $\shomo$-free sets}
We now prove that $\maxhomo$ is maximal $\shomo$-free.
The main idea is to exploit that every inequality describing $\maxhomo$ has
a point in $\shomo \cap \maxhomo$ exposing it and use \Cref{thm:exposed_maximality}. We begin with a Lemma whose technical proof we leave in the appendix. We recall that a function is sublinear if and only if it is convex and positive homogeneous.

\begin{restatable*}{lemma}{exposingIneq}
  %\begin{lemma} 
    \label{lemma:exposing_ineq}
    Let $\phi : \mathbb{R}^n \to \mathbb{R}$ be a sublinear function, $\lambda \in
    D_1(0)$, and let
    \[
      C = \{ (x,y) \st \phi(y) \leq \lambda^\T x \}.
    \]
    Let $(\xlp, \ylp) \in C$ be such that $\phi$ is differentiable at $\ylp$ and
    $\phi(\ylp) = \lambda^\T \xlp$.
    Then $(\xlp, \ylp)$ exposes the valid inequality $-\lambda^\T x + \nabla
    \phi(\ylp)^\T y \leq 0$.
    %
  %\end{lemma}
  \end{restatable*}

\begin{theorem} \label{thm:max_homo}
  Let $\shomo = \{(x,y) \in \mathbb{R}^{n+m} \st \|x\| \leq \|y\|\}$ and $\maxhomo
  = \{ (x,y) \in \mathbb{R}^{n+m} \st \lambda^\T x \geq \|y\| \}$ 
  for $\lambda \in D_1(0)$.
  Then, $\maxhomo$ is a maximal $\shomo$-free set. Furthermore, if $\lambda = \frac{\xlp}{\|\xlp\|}$, $\maxhomo$ contains $(\xlp,\ylp)$ in its interior.
\end{theorem}
\begin{proof}
  The $\shomo$-freeness follows by construction. To show that $\maxhomo$ is maximal, we first notice that
  \[
    \maxhomo = \{ (x,y) \in \mathbb{R}^{n+m} \st -\lambda^\T x + \beta^\T y \leq 0,
    \ \forall \beta \in D_{1}(0) \}.
  \]
  We just need to show that every inequality $(-\lambda, \beta)$ is exposed by a point
  $(x,y) \in \shomo \cap \maxhomo$.

  Since the norm function $\|\cdot \|$ is sublinear, differentiable everywhere but in the origin, and 
  \(\|\beta\| = 1 = \lambda^\T\lambda \),
  \Cref{lemma:exposing_ineq} shows that $(\lambda, \beta) \in \shomo \cap \maxhomo$
  exposes $(-\lambda, \beta)$.
  From \Cref{thm:exposed_maximality} we conclude that $\maxhomo$ is maximal
  $\shomo$-free.
  
  The fact that $(\xlp,\ylp)\in \inte(\maxhomo)$ when $\lambda = \frac{\xlp}{\|\xlp\|}$, can be verified directly.
\end{proof}

\section{Homogeneous quadratics with a single homogeneous linear constraint} \label{sec:homowithhomo}
Finding maximal $S$-free sets for $S$ defined using a
non-homogeneous quadratic function is much more challenging than the previous case.
In general, using a homogenization and diagonalization, any such $S$ can be described as
\begin{equation} \label{eq:nonhomogeneric}
  \{ (x,y,z) \in \mathbb{R}^{n+m+l} \st \|x\| \leq \|y\|,\, a^\T x + d^\T
  y + h^\T z = -1 \}.
\end{equation}
\begin{remark}
  Similarly to our discussion in \Cref{remark:transformationhomo}, the choice of transformation to bring a non-homogenous quadratic to the form \eqref{eq:nonhomogeneric} is not unique. Different choices can produce different vectors $a,d,h$. Nonetheless, maximality of $S$-free sets is preserved through these transformations if they are one-to-one. We discuss the effect of the different choices of such transformations in \Cref{sec:transformations}.
\end{remark}
First of all, we note that the case $h \neq 0$ can be tackled directly using \Cref{sec:maximalhomo}. Indeed, if this is the case it is not hard to see that $C\times \mathbb{R}^l$ is maximal $S$-free (with respect to the corresponding hyperplane), where $C$ is any maximal $\shomo$-free. This follows from \Cref{thm:projection}.
Thus, in what follows we consider
\begin{align*}
  \snonhomo = \{ (x,y) \in \mathbb{R}^{n+m} \st \|x\| \leq \|y\|, a^\T x + d^\T y = -1 \}.
\end{align*}
Also note that using transformations that yield the latter form of $\snonhomo$ allow us to assume that the given point $(\bar x, \bar y)\not\in \snonhomo$ satisfies 
\[ \|\bar x\| > \| \bar y\|,\, a^\T \bar x + d^\T \bar y = -1. \]
We elaborate on this point in \Cref{sec:transformations}.

The set $\snonhomo$ above is our final goal. However, at this point, a simpler set to study is
\begin{align*}
  \shomob
  &= \{ (x,y) \in \mathbb{R}^{n+m} \st \|x\| \leq \|y\|, a^\T x + d^\T y \leq 0 \}.
\end{align*}
In this section we construct maximal $\shomob$-free sets that contain $(\bar x, \bar y)$ satisfying 
\[ \|\bar x\| > \| \bar y\|,\, a^\T \bar x + d^\T \bar y \leq 0. \]
While this set is interesting on its own, it provides an important intermediate step into our construction of maximal $S$-free sets.

As it turns out, the construction of maximal $\shomob$-free sets
depends on whether $\|a\| < \|d\|$ or $\|a\| \geq \|d\|$ and on the value of $m$. Unfortunately, each case requires different ideas. The following remark dismisses a simple case:

\begin{remark}
If $m = 1$ and $\|a\| < \|d\|$ then $\shomob$ is convex.
To see this, assume that $d > 0$ and let $(x,y) \in \shomob$ with $y \neq 0$.
Then, $dy \leq -a^\T x \leq \|a\| \|x\| \leq \|a\| |y| < d |y|$.
This can only happen if $y < 0$.
Therefore, $\shomob$ is the second order cone $\{ (x,y) \st \|x\| \leq -y \}$.
The case $d < 0$ is analogous.
We remark that the assumption $\|a\| < |d|$ is fundamental for the argument.
As we show in \Cref{ex:counterex_easycase}, $\shomob$ is not necessarily
convex if $\|a\| = |d|$.
\end{remark}
We divide the remaining cases in the following:
\begin{description}
  \item[Case 1] $\|a\| \leq \|d\|\, \wedge \, m > 1$.
  \item[Case 2] $\|a\| \geq \|d\|$.
\end{description}
Note that both our strategies allow us to handle the overlapping case $\|a\|=\|d\| \, \wedge\, m>1$. 
We start with the more natural idea that follows from our previous discussions.
This yields the proof of Case 1 and motivates our case distinction.

\subsection{Case 1: $\|a\| \leq \|d\|\, \wedge \, m > 1$}

The strategy for proving maximality of $\maxhomo$ was to write $\maxhomo$ as
\[
  \maxhomo = \{ (x,y) \in \mathbb{R}^{n+m} \st -\lambda^\T x + \beta^\T y \leq 0,
  \ \forall \beta \in D_{1}(0) \},
\]
and to find an exposing point in $\shomo \cap \maxhomo$ for each of the
inequalities defining $\maxhomo$.
As $\shomob \subseteq \shomo$, $\maxhomo$ is clearly $\shomob$-free.
However, if we try to prove it is maximal following the same technique,
we find that it is not clear that some inequalities have exposing
points in $\shomob \cap \maxhomo$.
The exposing point of the inequality $(-\lambda, \beta)$, $(\lambda, \beta)$ is
in $\shomob$ if and only if $a^\T \lambda + d^\T \beta \leq 0$.
Let 
\[G(\lambda) = \{ \beta \st \|\beta\| = 1,\ a^\T \lambda
+ d^\T \beta \leq 0\}. \]
It is natural to ask, then, if
\[
  \maxhomobg = \{ (x,y) \in \mathbb{R}^{n+m} \st -\lambda^\T x + \beta^\T y \leq 0,
  \ \forall \beta \in G(\lambda) \}
\]
is maximal $\shomob$-free. Intuitively, $C_{G(\lambda)}$ is obtained from $C_\lambda$ by removing from its description all inequalities that do not have an exposing point in $a^\T \lambda + d^\T \beta \leq 0$.
It is reasonable to expect maximality, as, by construction, every inequality has
a point exposing it.
Indeed,
\begin{proposition} \label{prop:largest_maybe_free}
  If $\maxhomobg \neq \emptyset$ and $C$ is any $\shomob$-free set such that
  $\maxhomo \subseteq C$, then $C \subseteq \maxhomobg$.
\end{proposition}
\begin{proof}
  Suppose, by contradiction, that $C \not\subseteq \maxhomobg$.
  This implies that there must exist $ \beta_0 \in G(\lambda)$ such that
  $-\lambda^\T x + \beta_0^\T y \leq 0$ is not valid for $C$.
  As $\maxhomo \subseteq \maxhomobg$, $-\lambda^\T x + \beta_0^\T y \leq 0$ is
  valid for $\maxhomo$.
  
  As we saw in \Cref{thm:max_homo}, $(\lambda, \beta_0) \in \maxhomo$ exposes
  $-\lambda^\T x + \beta_0^\T y \leq 0$, and since $\maxhomo \subseteq
  C$, \Cref{thm:exposed_enters} implies that $(\lambda, \beta_0) \in
  \inte(C)$.
  However, since $\beta_0 \in G(\lambda)$, we have $(\lambda, \beta_0) \in
  \shomob$. This contradicts the $\shomob$-freeness of $C$.
\end{proof}
This result shows that $\maxhomobg$ is the largest (inclusion-wise) set that one
can aspire to obtain from $\maxhomo$. However, it is unclear if $\maxhomobg$ is
$\shomob$-free.
Even more, it is unclear whether $G(\lambda)$ is non-empty or not.
In the following we study when $\maxhomobg$ is $\shomob$-free

We start by showing that when $\lambda = \frac{\xlp}{\|\xlp\|}$, $G(\lambda)$ is non-empty.
\begin{proposition} \label{prop:d_equal_0}
  Let $(\xlp, \ylp) \notin \shomob$ such that $a^\T \xlp + d^\T \ylp \leq 0$ and
  let $\lambda = \frac{\xlp}{\|\xlp\|}$.
  Then,
  \[
    G(\lambda) \neq \emptyset.
  \]
  If, in addition, $d = 0$, then $G(\lambda) = D_{1}(0)$
  and  $\maxhomobg = \maxhomo$ is maximal $\shomob$-free.
\end{proposition}
\begin{proof}
  As $(\xlp, \ylp) \notin \shomob$, we have that $\|\ylp\| < \|\xlp\|$.
  Since $m > 1$, then we can find $z \in \mathbb{R}^m \setminus \{0\}$ such that
  $d^\T z = 0$ and $\|\frac{\ylp}{\|\xlp\|} + z\| = 1$.
  Also, $a^\T \xlp + d^\T \ylp \leq 0$ and $d^\T z = 0$ imply that $a^\T \lambda
  + d^\T(\frac{\ylp}{\|\xlp\|} + z) \leq 0$.
  Thus, $\frac{\ylp}{\|\xlp\|} + z \in G(\lambda)$.

  Regarding the second statement of the proposition, if $d = 0$ then clearly either $G(\lambda) = D_1(0)$ or $G(\lambda) = \emptyset$. Since we are in the case $G(\lambda) \neq \emptyset$, this immediately implies $\maxhomobg = \maxhomo$. Thus, \Cref{prop:largest_maybe_free}
  implies its maximality.
\end{proof}

In light of \Cref{prop:largest_maybe_free}, we just need for $\maxhomobg$ to be
$\shomob$-free for it to be maximal.
Note that
\begin{equation}\label{eq:cglambdadef}
  \maxhomobg = \{ (x,y) \in \mathbb{R}^{n+m} \st  \max_{\beta \in G(\lambda)} y^\T
  \beta \leq \lambda^\T x \},
\end{equation}
and so to prove $\shomob$-freeness, it is enough to show that for every $(x,
y) \in \shomob$, $\max_{\beta \in G(\lambda)} y^\T \beta \geq \lambda^\T x$.
In trying to prove this inequality is where the conditions of this case
naturally arise.
\begin{proposition} \label{prop:max_homo_good}
  Let $(\xlp, \ylp) \notin \shomob$ such that $a^\T \xlp + d^\T \ylp \leq 0$ and
  $\lambda = \frac{\xlp}{\|\xlp\|}$.
  If $\|d\| \geq \|a\|$ and $m > 1$, then $\maxhomobg$ is maximal
  $\shomob$-free and contains $(\xlp, \ylp)$ in its interior.
\end{proposition}

\begin{proof}
  As discussed above, it is enough to show that 
  \begin{equation} \label{eq:prop:max_homo_good:to_show}
    \max_{\beta \in G(\lambda)} y^\T \beta \geq \lambda^\T x \text{ for every
    } (x, y) \in \shomob.
  \end{equation}
  Informally, the strategy is to find a dual of $\max_{\beta \in G(\lambda)}
  y^\T \beta$ so that the inequality we have to prove is of the form ``minimum
  of something greater or equal than $\lambda^\T x$'', which often times is
  easier to reason about.
  As the objective function of $\max_{\beta \in G(\lambda)} y^\T \beta$ is
  linear and $m > 1$, we can replace the $\|\beta\| =1$ constraint with an inequality and obtain
  \begin{equation} \label{eq:maximizationoverG}
    \max_{\beta \in G(\lambda)} y^\T \beta = \max \{ y^\T \beta \st \|\beta\|
    \leq 1, a^\T \lambda + d^\T \beta \leq 0 \}.
  \end{equation}
  As $G(\lambda)$ is constructed from an infeasible point $(\xlp, \ylp) \notin
  \shomob$ such that $a^\T \xlp + d^\T \ylp \leq 0$, i.e., $\|\ylp \| < \| \xlp
  \|$, we have $\left\| \ylp/\|\xlp\|  \right\| < 1$.
  Moreover, perturbing the latter we can argue that the rightmost optimization
  problem in \eqref{eq:maximizationoverG} has a strictly feasible point.
  Thus, Slater's condition holds and we have that
  \begin{equation} \label{eq:strongduality}
    \max \{ y^\T \beta \st \|\beta\| \leq 1, a^\T \lambda + d^\T \beta \leq 0 \}
    =
    \inf_{\theta \geq 0} \|y - d \theta\|  - \lambda^\T a \theta.
  \end{equation}
  Using \eqref{eq:strongduality}, \eqref{eq:prop:max_homo_good:to_show} is equivalent to
  \begin{equation} \label{eq:prop:max_homo_good:to_show_easier}
    \inf_{\theta \geq 0} \|y - d \theta\|  - \lambda^\T a \theta \geq \lambda^\T
    x \text{ for every } (x, y) \in \shomob.
  \end{equation}
  We now prove that if $(x,y) \in \shomob$, then $\lambda^\T (x + a \theta) \leq
  \|y - d \theta\|$, which implies the result.

  By Cauchy-Schwarz and $\|\lambda\| = 1$, we have that $\lambda^\T(x
  + a \theta) \leq \|x + a \theta\|$.
  Furthermore,
  $\|x + a \theta\|^2 = \|x\|^2 + 2 \theta a^T x + \|a \theta\|^2$.
  Since $\theta \geq 0$, $\theta a^\T x \leq -\theta d^\T y$.
  Together with $\|x\|^2 \leq \|y\|^2$ they imply
  \begin{align*}
    \|x + a \theta\|^2
    &\leq \|y\|^2 - 2 \theta d^\T y + \|a\|^2 \theta^2 \\
    &= \|y - d\theta\|^2 + (\|a\|^2 - \|d\|^2) \theta^2 \\
    &\leq \|y - d\theta\|^2,
  \end{align*}
  where the last inequality follows since $\|d\| \geq \|a\|$.

  We have shown that $\|x + a \theta\| \leq \|y - d\theta\|$.
  Hence, $\lambda^\T(x + a\theta) \leq \|y - d\theta\|$ as we wanted to show,
  which implies that $\maxhomobg$ is $\shomob$-free.
  Finally, \Cref{prop:largest_maybe_free} implies the maximality of
  $\maxhomobg$, and $(\xlp, \ylp)\in\inte(\maxhomobg)$ since $\maxhomo\subseteq \maxhomobg$.
\end{proof}

\begin{remark}
  Using \Cref{prop:phi} one can show that $\max_\beta \{ y^\T \beta \st \|\beta\| \leq 1,  a^\T
  \lambda + d^\T \beta \leq 0 \}$ is 
  \begin{equation}  \label{eq:sublinearcglambda}
    \begin{cases}
      \|y\|, &\text{ if } a^\T \lambda \|y\| + y^\T d \leq 0 \\
      \sqrt{(1 - (\frac{a^\T \lambda}{\|d\|})^2)(\|y\|^2 - (\frac{y^\T d}{\|d\|^2})^2)} - \frac{a^\T \lambda y^\T d}{\|d\|^2},
             &\text{ otherwise. }
    \end{cases}
  \end{equation}
  Note that this is well defined since if $\|d\| = 0$, then $\|a\| = 0$ and so $\eqref{eq:sublinearcglambda} = \|y\|$.
  This yields a closed-form expression for $\maxhomobg$ of the form
  \begin{equation}\label{eq:closedformcglambda}
    \maxhomobg = \{ (x,y) \in \mathbb{R}^{n+m} \st  \eqref{eq:sublinearcglambda} \leq \lambda^\T x \}.
  \end{equation}
\end{remark}

The last proposition provides certain guarantees of when a simple modification
of $\maxhomo$ yields maximal $\shomob$-free sets. Our proof heavily relies on
our assumptions $\|a\| \leq \|d\|$ (to show
\eqref{eq:prop:max_homo_good:to_show_easier}) and $m > 1$ (to show
\eqref{eq:maximizationoverG}), so the natural question is
whether these conditions are actually necessary for our statement to be true.
Thus, before moving on to the next case, we argue why these conditions are
indeed necessary in our statements. The following examples motivate our case
distinction and illustrate all cases we have covered.

\begin{example}\label{ex:simple3D-m2}
   Consider the following set of the type $\shomob$, which we denote $S_{\leq 0}^1$:
  \[S_{\leq 0}^1 = \{(x, y_1, y_2) \in \mathbb{R}^3 \st |x | \leq \|y\|, \quad a x + d^\T y \leq 0 \} \]

  with $a = 1 $ and $d=(1,-1)^\T$. Let us consider the point $(\bar{x},\bar{y})
  = (-1, 0 ,0)^\T$, clearly satisfying the linear inequality, but not in $S_{\leq 0}^1$.
  In \Cref{fig:simpleexample3dgood} we show $S_{\leq 0}^1$, the $S_{\leq 0}^1$-free set
  given by $\maxhomo$ and the set $\maxhomobg$ for $\lambda = \frac{\xlp}{\|\xlp\|}$. Since in this case $|a|
  = 1 \leq \sqrt{2} = \| d\|$ and $m>1$, we know $\maxhomobg$ is maximal $S_{\leq 0}^1$-free.
\end{example}

\begin{figure}
  \begin{subfigure}{.45\textwidth}
    \centering
    \includegraphics[scale=0.3]{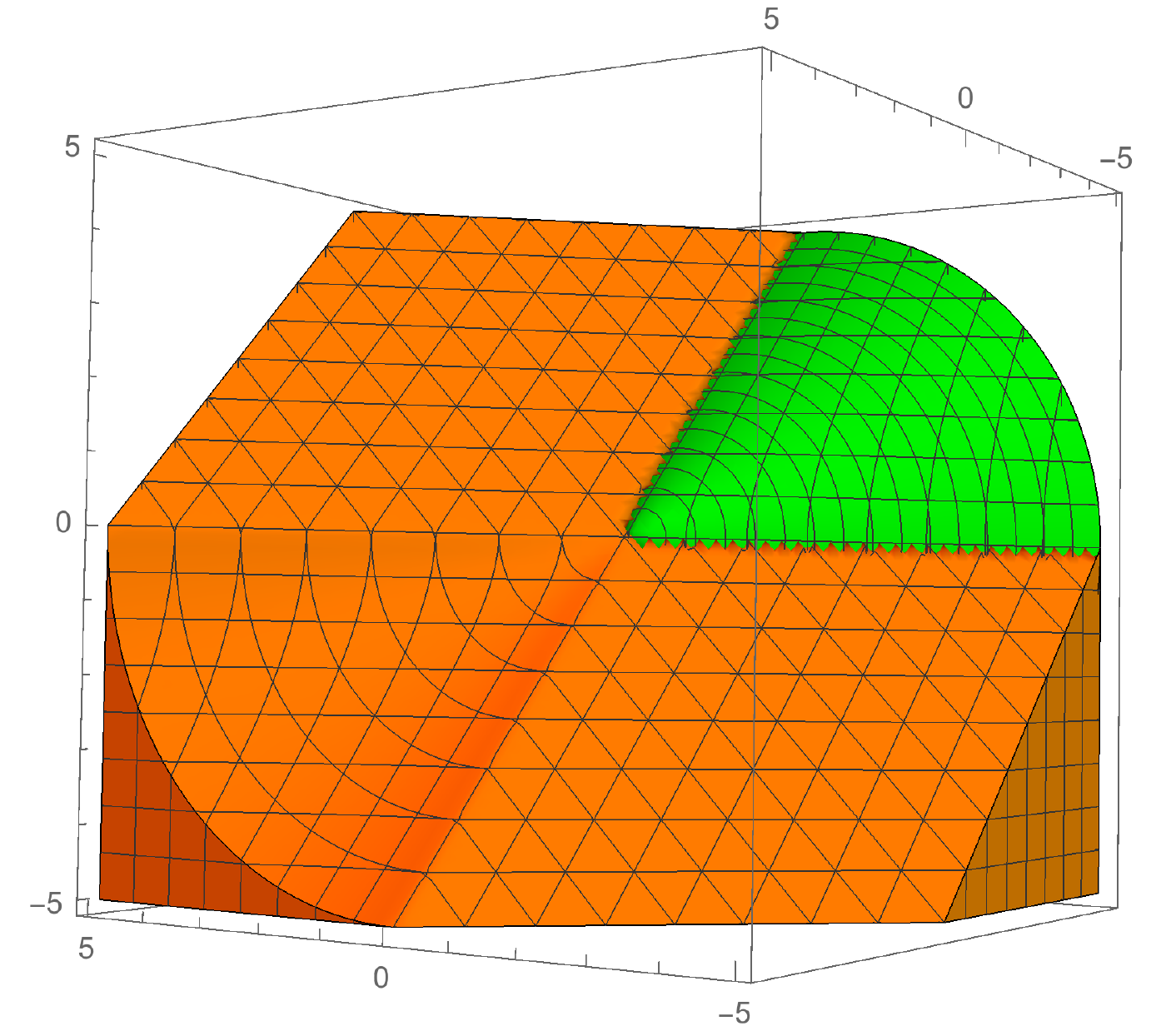}
    \includegraphics[scale=0.3]{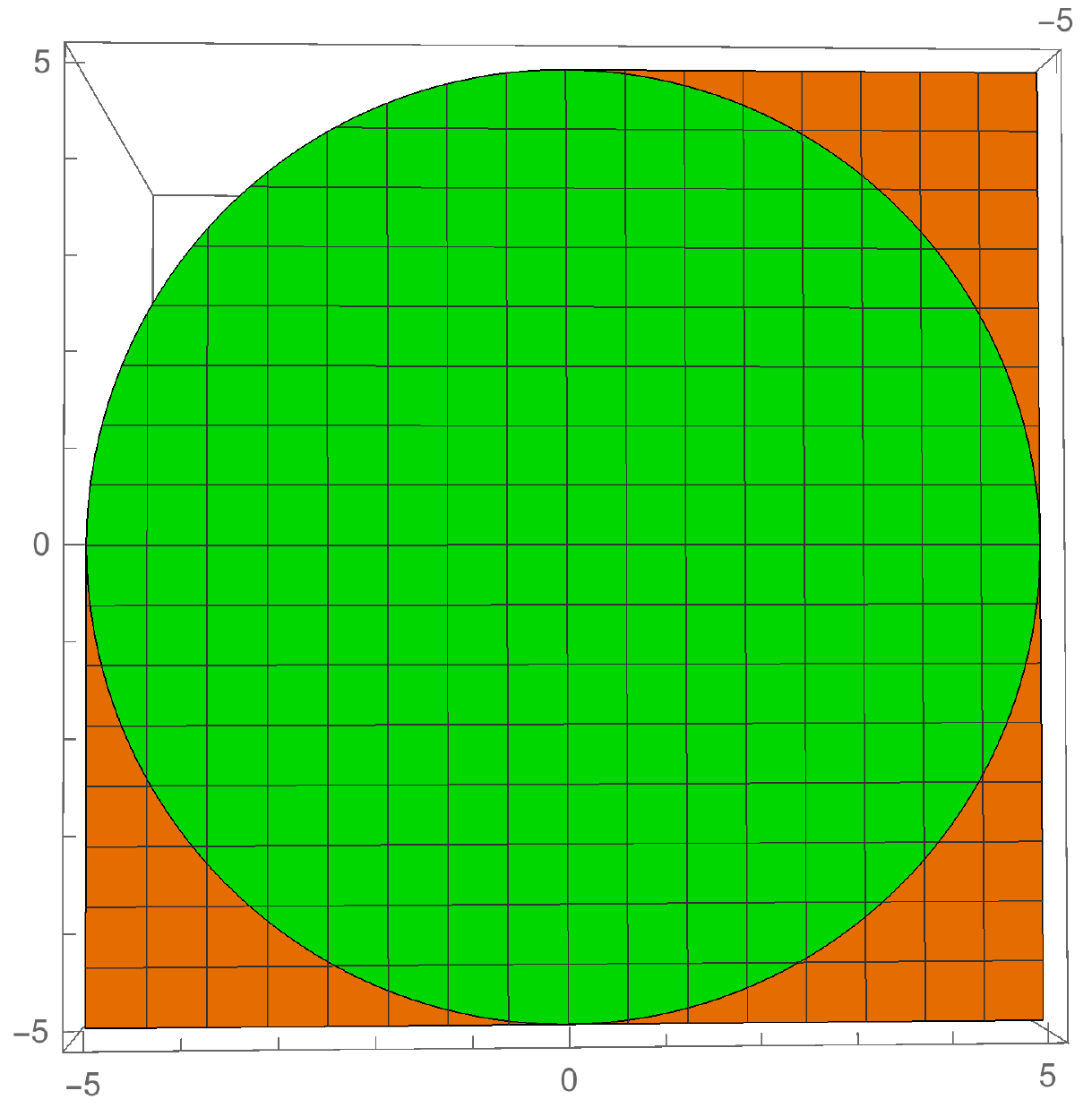}
    \caption{$S_{\leq 0}^1$ in \Cref{ex:simple3D-m2} (orange) and the corresponding $\maxhomo$ set (green). The latter is $S_{\leq 0}^1$-free but not maximal. }
    \label{fig:sub-simpleexample3dgood-clambda}
  \end{subfigure} \hspace{.1cm}
  \begin{subfigure}{.45\textwidth}
    \centering
    \includegraphics[scale=0.3]{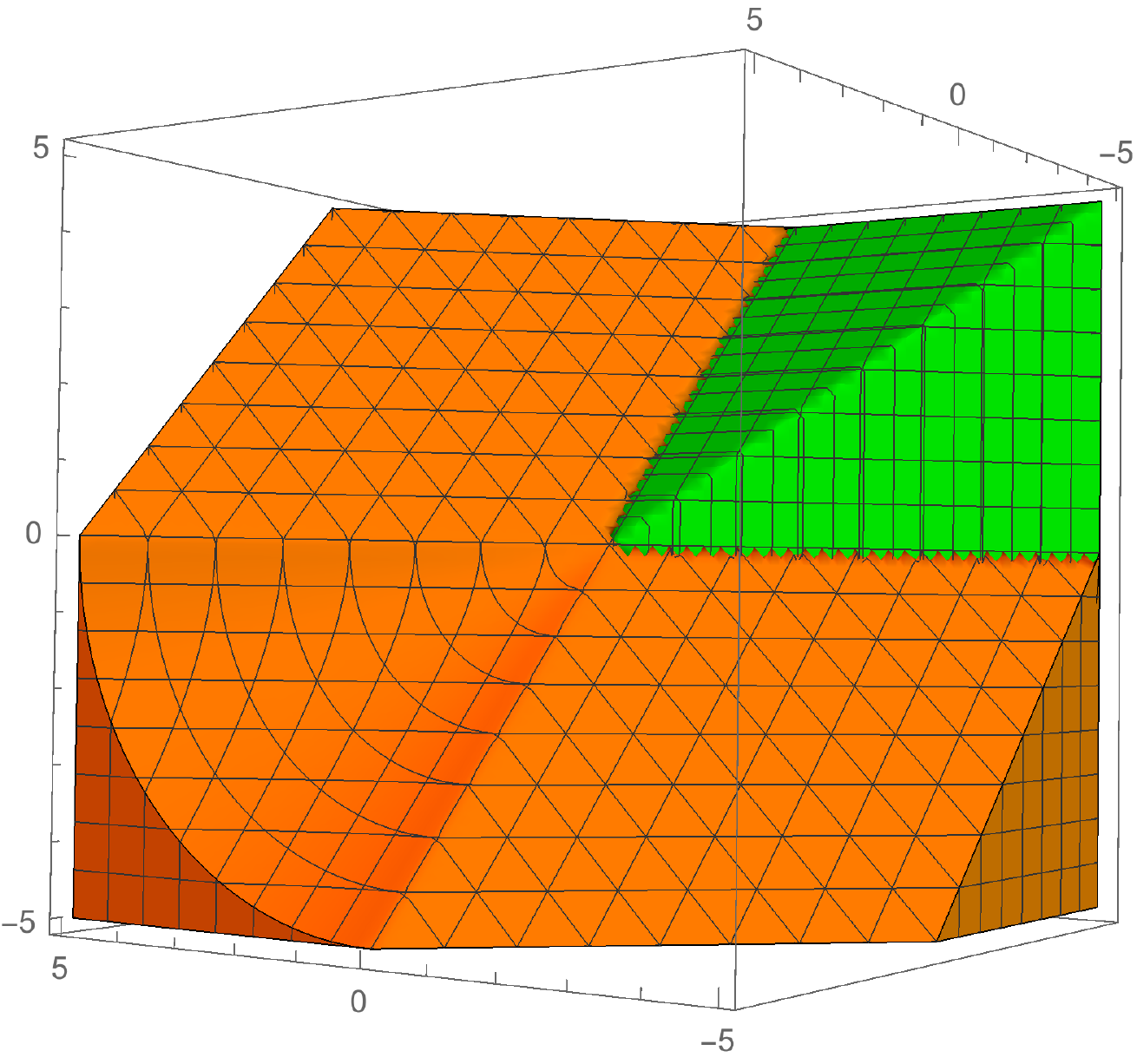}
    \includegraphics[scale=0.3]{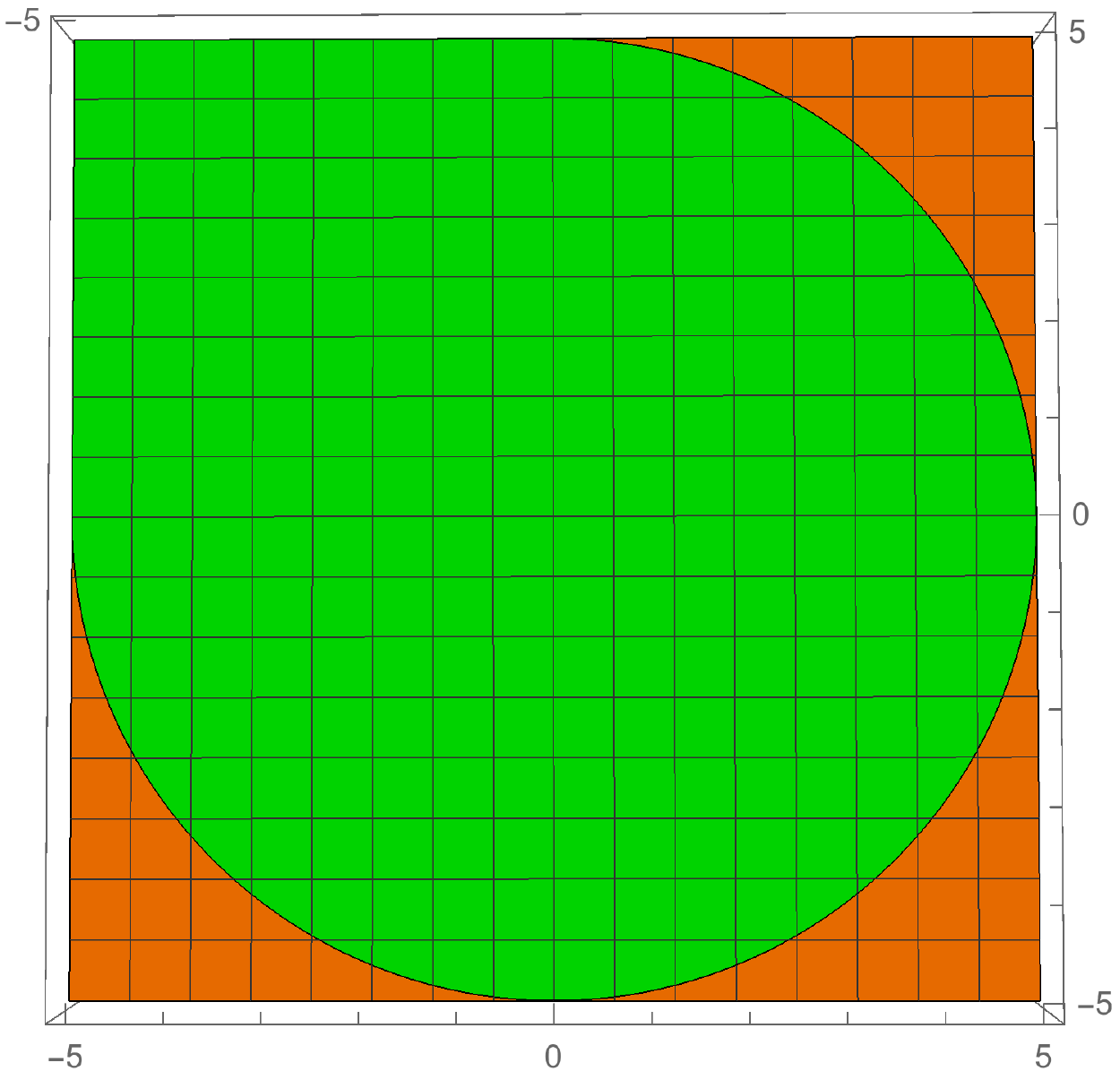}
    \caption{$S_{\leq 0}^1$ in \Cref{ex:simple3D-m2} (orange) and the corresponding $\maxhomobg$ set (green). The latter is maximal $S_{\leq 0}^1$-free.}
    \label{fig:sub-simpleexample3dgood-cglambda}
  \end{subfigure}
  \caption{Sets $\maxhomo$ and $\maxhomobg$ in \Cref{ex:simple3D-m2} for the case $\|a\| \leq \|d\|$.}
  \label{fig:simpleexample3dgood}
\end{figure}

  \begin{example}\label{ex:simple3D}
  Consider the set $\shomorunning$, defined as 
  \[\shomorunning = \{(x_1, x_2 , y) \in \mathbb{R}^3 \st \|x \| \leq |y |, \quad a^\T x + d y \leq 0 \} \]
  with $a
  = (-1/\sqrt{2},1/\sqrt{2})^\T $ and $d=1/\sqrt{2}$ (the $1/\sqrt{2}$ terms are not really important now as we can scale the inequality, but we reuse this example in subsequent sections where they do matter), and $(\bar{x},\bar{y})=(-1, -1 ,0)^\T$. This point satisfies the linear inequality in $\shomorunning$, but it is not in $\shomorunning$. Let $\lambda = \frac{\xlp}{\|\xlp\|}$. 
  
  In this case $a^\T \lambda  = 0$, and as a consequence the corresponding set $G(\lambda)$ is
  given by the singleton $\{-1\}$. In \Cref{fig:simpleexample3dbad} we
  show $\shomorunning$, the $\shomorunning$-free set given by $\maxhomo$ and the set
  $\maxhomobg$. In this case $\|a\| = 1 > 1/\sqrt{2} = | d |$, so we have no
  guarantee on the $\shomorunning$-freeness of $\maxhomobg$. Even more, it is not
  $\shomorunning$-free.
\end{example}

\begin{figure}
  \begin{subfigure}{.45\textwidth}
    \centering
    \includegraphics[scale=0.3]{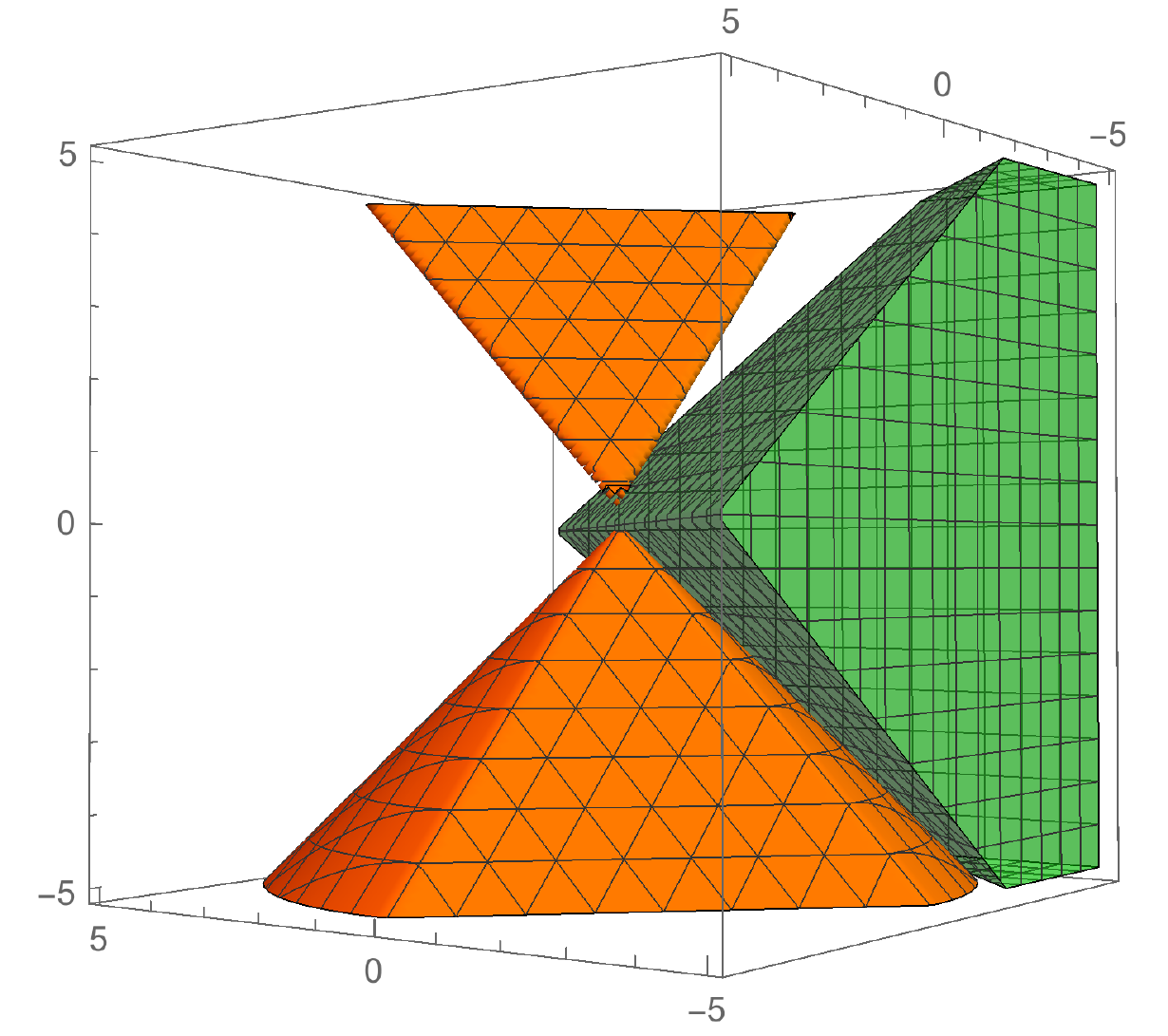}
    \caption{$\shomorunning$ in \Cref{ex:simple3D} (orange) and the corresponding $\maxhomo$ set (green). The latter is $\shomorunning$-free but not maximal.}
    \label{fig:sub-simpleexample3dbad-clambda}
  \end{subfigure}\hspace{.1cm}
  \begin{subfigure}{.45\textwidth}
    \centering
    \includegraphics[scale=0.3]{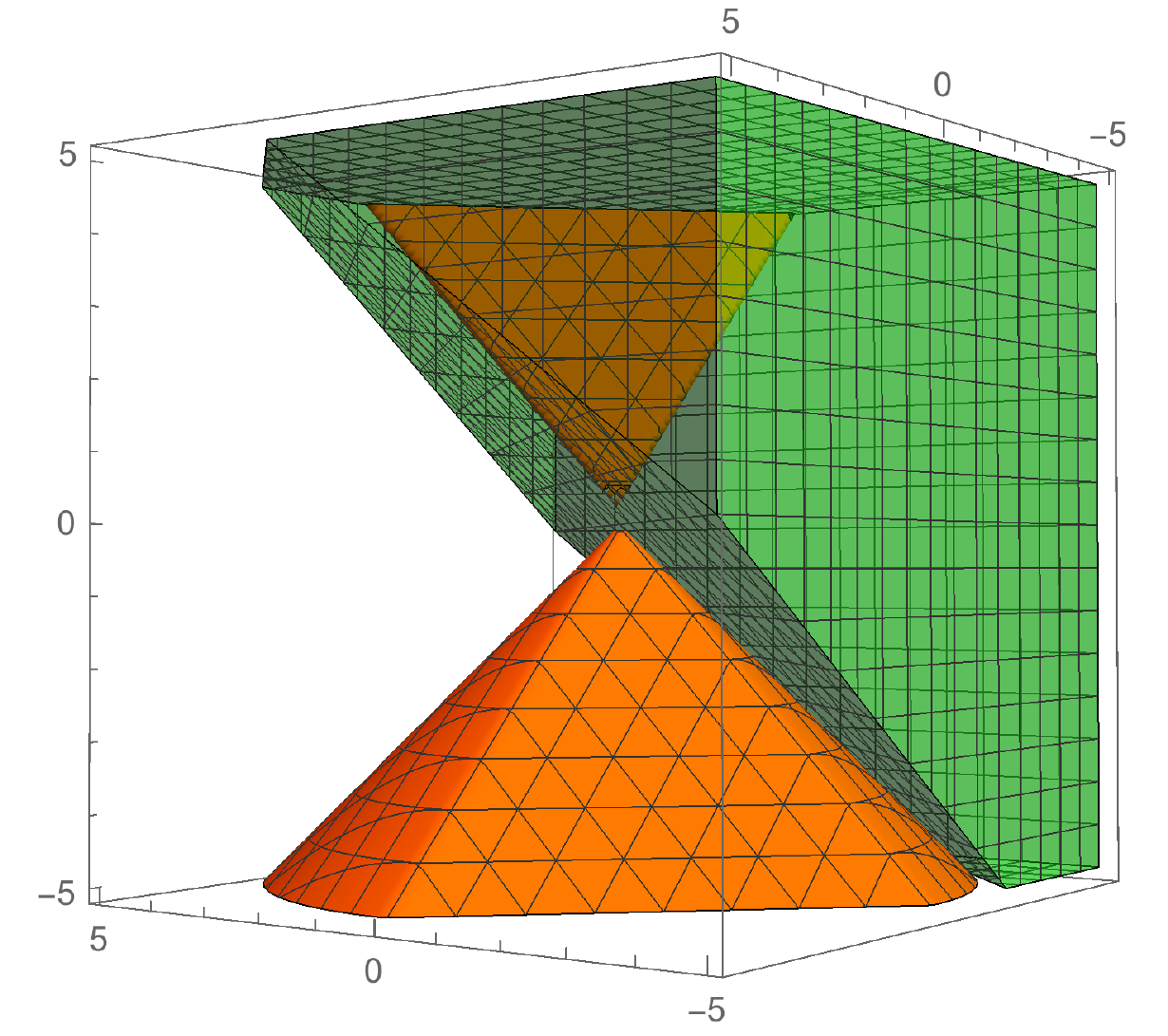}
    \caption{$\shomorunning$ in \Cref{ex:simple3D} (orange) and the corresponding
    $\maxhomobg$ set (green). The latter is not $\shomorunning$-free.}
    \label{fig:sub-simpleexample3dbad-cglambda}
  \end{subfigure}
  \caption{Sets $\maxhomo$ and $\maxhomobg$ in \Cref{ex:simple3D} for the case $\|a\| > \|d\|$.}
  \label{fig:simpleexample3dbad}
\end{figure}

\begin{example} \label{ex:counterex_easycase} Let us consider the following example with $n = 2$, $m=1$ and $\|d\| = \| a\| $. Let $a = (-3,4)^\T, d = 5$ and consider \( (\bar{x},\bar{y}) = (-4,-3,-1) \) and $\lambda = \frac{\xlp}{\|\xlp\|}$. Clearly $(\bar{x}, \bar{y}) \not\in \shomob$, but satisfies the linear constraint. In this case, $\beta \in G(\lambda)$ must satisfy
 \[ 5\cdot \beta \leq 0, |\beta| = 1 \]
thus $G(\lambda) = \{-1\}$. Nonetheless, $(x,y) = (3,-4, 5) \in \shomob$, and 
\[\lambda^\T x + y = 0  + 5 > 0\]
This means $(x,y)\in \inte(C_{G(\lambda)})$. Thus, $C_{G(\lambda)}$ is not $\shomob$-free.
\end{example}

\begin{remark} \label{rmk:less_cases_in_non_homo}
  The situation in Example~\ref{ex:counterex_easycase} is similar to the one
  depicted in Figure~\ref{fig:sub-simpleexample3dbad-cglambda}.
  Roughly speaking, when $\|a\| = \|d\|$ the upper region becomes a single line
  and this line intersects the interior of $\maxhomobg$.
  Intuitively, when we consider $\snonhomo$ where $a^\T x + d^\T
  y = -1$, this line should not appear.
  Even more, $\snonhomo$ should be convex.
  We will see that this is the case in the \Cref{subsec:easycasenonhomo}.
\end{remark}

\subsection{Case 2: $\|a\| \geq \|d\|$} \label{subsec:homowithhomohard}
As we have seen in \Cref{ex:simple3D}, when $\|a\| \leq \|d\|$ does not hold, $\maxhomobg$ is not necessarily $\shomob$-free.
On the other hand, $\maxhomo$ is $\shomob$-free but not necessarily maximal.
As before, we are looking for a convex set $C$ that is maximal $\shomob$-free
set that \emph{contains} $\maxhomo$. We point out that in not all statements of this section we require $\lambda = \frac{\xlp}{\|\xlp\|}$.

\subsubsection{Projecting-out the lineality space}
The lineality space of $\maxhomo$ is $L = \{ (x,y) \st \lambda^\T x = 0,
y = 0 \}$ and as $\maxhomo \subseteq C$, it must be that $L$ is
contained in the lineality space of $C$.
By \Cref{thm:projection}, $\proj_{L^\perp} C$ is maximal $\proj_{L^\perp}
\shomob$-free, thus, it might be possible (and we show it is) to find $C$ by studying maximal
$\proj_{L^\perp} \shomob$-free sets.
We note that $L^\perp = \langle \lambda \rangle \times \mathbb{R}^m$ and
\[
  \proj_{L^\perp} \shomob = \{ (\lambda^\T x, y) \st \|x\| \leq \|y\|,\ a^\T
  x + d^\T y \leq 0 \}.
\]
After analyzing low dimensional instances of $\proj_{L^\perp} \shomob$ we conjecture that 
$\stcomp{\proj_{L^\perp} \shomob }$ is formed by the union of two disjoint convex sets. If this is true, it would directly provide maximal
$\proj_{L^\perp} \shomob$-free sets. 

In order to show that this is actually true, we follow the following strategy.
For each point $y \in \mathbb{R}^m$, the points $(\lambda^Tx, y) \in
\proj_{L^\perp} \shomob$ lie on an interval, namely, $\{ \lambda^\T x \st \|x\|
\leq \|y\|,\ a^\T x + d^\T y \leq 0\}$.
Thus, we define the functions
\begin{align*}
  y &\mapsto \max \{ \lambda^\T x \st \|x\| \leq \|y\|, a^\T x + d^\T y \leq
  0 \} \text{ and } \\
  y &\mapsto \min \{ \lambda^\T x \st \|x\| \leq \|y\|, a^\T x + d^\T y \leq
  0 \}.
\end{align*}
If the first function is convex and the second is concave, then the closure of $\stcomp{\proj_{L^\perp} \shomob}$ is the union of the epigraph of the
first one and the hypograph of the second one.
Thus, it suffices to show that
\begin{equation} \label{eq:phi_function}
  \phifun(y) = \max_x \{ \lambda^\T x \st \|x\| \leq \|y\|, a^\T x + d^\T
  y \leq 0 \}
\end{equation}
is convex for every $\lambda \in D_1(0)$, as the second function is $-\phi_{-\lambda}$.

We first show that $\phifun$ is defined over all $\mathbb{R}^m$.
\begin{proposition} \label{prop:interiornonempty}
  If $\|d\| \leq \|a\|$, then for every $y$ the set $\{(x,y) \st \|x\| \leq
  \|y\|, a^\T x \leq -d^\T y\}$ is not empty.
\end{proposition}
\begin{proof}
  Note that $x = -d^\T y \frac{a}{\|a\|^2}$ belongs to the set.
  Indeed, $a^\T x = -d^\T y$, in particular, $a^\T x \leq -d^\T y$.
  Also, $\|d\| \leq \|a\|$ implies that $\|x\| \leq \frac{\|d\|}{\|a\|} \|y\|
  \leq \|y\|$.
\end{proof}

We now show that $\phifun$ is convex. Furthermore, we prove that $\phifun$ is
sublinear, that is, convex and positive homogeneous.
The proof is basically to find $\phifun$ explicitly and then verify its
properties.
Note that in this case $\|a\| = 0$ implies that the linear inequality in $\shomob$ is trivial. Thus,we assume without loss of generality, that $\|a\| = 1$.
\begin{proposition} \label{prop:phifun_sublinear}
  Let $\lambda, a \in D_1(0) \subseteq \mathbb{R}^n$ and $d \in \mathbb{R}^m$
  such that $\|d\| \leq 1$.
  Then,
  \begin{equation} \label{eq:phi_characterization}
    \phifun(y)
    =
    \begin{cases}
      \|y\|,
    &\text{ if } \lambda^\T a \|y\| + d^\T y \leq 0 \\
    \sqrt{(\|y\|^2 - (d^\T y)^2)(1 - (\lambda^\T a)^2)} - d^\T y \lambda^\T a,
    &\text{ otherwise. }
    \end{cases}
  \end{equation}
  Furthermore, $\phifun$ is sublinear and
  \begin{itemize}
    \item if $\|d\| = 1 \wedge m > 1$, then $\phifun$ is differentiable
      $\mathbb{R}^m \setminus d \mathbb{R}_+$,
    \item otherwise $\phifun$ is differentiable in $\mathbb{R}^m \setminus
      \{0\}$.
  \end{itemize}
\end{proposition}
\begin{proof}
  The fact that $\phifun$ is positive homogeneous can be easily verified.
  We leave the proof that $\phifun$ is of the form
  \eqref{eq:phi_characterization} in the appendix (\Cref{prop:phi}). Thus convexity and differentiability remains.

  First, note that if $\lambda = a$, then $\phifun(y) = -d^\T y$.
  This function is clearly sublinear and differentiable everywhere.
  On the other hand, if $\lambda = -a$, then $\phifun(y) = \|y\|$.
  This function is clearly sublinear and differentiable everywhere but the
  origin.

  We now consider $\lambda \neq \pm a$.
  Let
  \begin{equation} \label{eq:A_primal}
    \begin{split}
      A_1 &= \{ y \st \lambda^\T a \|y\| + d^\T y \leq 0 \}, \\
      A_2 &= \{ y \st \lambda^\T a \|y\| + d^\T y \geq 0 \},
    \end{split}
  \end{equation}
  and let $\phifun^1$ and $\phifun^2$ be the restriction of $\phifun$ to $A_1$
  and $A_2$, respectively.

  To show that $\phifun$ is convex we are going to use~\cite[Theorem
  3]{Solovev1983}. In our particular case, since $\phifun$ is
  positively homogeneous, this theorem implies that we just need to check that $\phifun$ is convex on each
  convex subset of $A_1$ and $A_2$, $\phifun^1 = \phifun^2$ on $A_1 \cap A_2$,
  and that
  \begin{equation} \label{eq:sum_directional_derivatives}
    \phi'_\lambda(y;\rho) + \phi'_\lambda(y;-\rho) \geq 0, \text{ for all
    } \rho \in \mathbb{R}^m \setminus \{0\}, y \in A_1 \cap A_2.
  \end{equation}
  Here, $\phifun'(y;\rho)$ is the directional derivative of $\phifun$ at $y$ in
  the direction of $\rho$.

  Clearly, $\phifun$ is convex in each convex subset of $A_1$.
  The function $\phifun^2$ is of the form $c_1 \|y\|_W - c_2 d^\T y$, where $W
  = I - d d^\T \succeq 0$ and $c_1, c_2$ are constants.
  Thus, $\phifun$ is convex on each convex subset of $A_2$.

  It is not hard to see that $\phifun^1(y) = \phifun^2(y)$ for $y \in A_1 \cap
  A_2$.

  Let us verify \eqref{eq:sum_directional_derivatives} for $y\neq 0$.
  For this, first notice that $\phifun^1(y)$ is differentiable whenever $y \neq
  0$.
  Likewise, $\phifun^2(y)$ is differentiable whenever $y \neq 0$ if $\|d\| < 1$
  or whenever $y \notin d \mathbb{R}_+$ if $\|d\| = 1$.
  However, if $y \in A_1 \cap A_2 \setminus \{0\}$ and $\|d\| = 1$, then $y
  \notin d \mathbb{R}_+$, thus $\phifun^2$ is differentiable in a neighborhood
  of $y$.
  Furthermore,
  \begin{align*}
    \nabla \phifun^2(y) &= \frac{(1 - (\lambda^\T a)^2) (I - dd^\T) y}{\sqrt{(\|y\|^2 - (d^\T y)^2)(1 - (\lambda^\T a)^2)} } - \lambda^\T a d \\
                        & = \frac{1}{\|y\|} (I - dd^\T) y - \lambda^\T a d \\
                        & = \frac{y}{\|y\|}  \\
                        & = \nabla \phifun^1(y).
  \end{align*}
  Therefore, $\phifun$ is differentiable in whenever $y \neq 0$ if $\|d\| < 1$
  or whenever $y \notin d \mathbb{R}_+$ if $\|d\| = 1$.
  Thus, \eqref{eq:sum_directional_derivatives} holds with equality for $y \in
  A_1 \cap A_2 \setminus \{0\}$.

  It remains to verify \eqref{eq:sum_directional_derivatives} for $y = 0$.
  Let $\rho$ be such that $\rho \in A_1$ and $-\rho \in A_2$.
  As $\phifun$ is positively homogeneous, $\phifun'(0;\cdot) =\phifun(\cdot)$.
  Hence,
  \begin{equation*}
    \phifun'(0; \rho) = \|\rho\| \text{ and } \phi'_\lambda(0; -\rho)
    = \sqrt{1 - (\lambda^T a)^2}\sqrt{\|\rho\|^2 - (d^\T \rho)^2} + d^\T \rho
    \lambda^T a.
  \end{equation*}

  We need to prove that
  \[
    \sqrt{1 - (\lambda^T a)^2}\sqrt{\|\rho\|^2 - (d^\T \rho)^2} + d^\T \rho
    \lambda^T a + \|\rho\| \geq 0.
  \]
  By Cauchy-Schwarz, $| d^\T \rho \lambda^T a | \leq \|d\|\|\rho\| < \|\rho\|$.
  Thus, $d^\T \rho \lambda^T a + \|\rho\| > 0$.
  Since $\sqrt{1 - (\lambda^T a)^2}\sqrt{\|\rho\|^2 - (d^\T \rho)^2} \geq 0$,
  the inequality follows.
  Therefore, $\phifun$ is convex.

  We have proved that $\phifun$ is convex and differentiable in $\mathbb{R}^m
  \setminus \{0\}$ if $\|d\| < 1$ and in $\mathbb{R}^m \setminus d\mathbb{R}_+$
  if $\|d\| = 1$.
  It remains to show that if $m = 1$ and $\|d\|=1$, then $\phifun$ is
  differentiable in $\mathbb{R}^m \setminus \{0\}$.
  This follows from \eqref{eq:phi_characterization} since $\phifun^2(y) = - dy
  \lambda^\T a$ in this case.
  This concludes the proof.
\end{proof}

With this, we have completed the proof of sublinearity of $\phifun$. Moreover,
we have explicitly described the function. As a corollary:

\begin{corollary}
  The epigraph of $\phifun$ and the hypograph of $-\phi_{-\lambda}$ are maximal $\proj_{L^\perp} \shomob$-free sets.
\end{corollary}

While this result provides two convex sets, it is not clear which one to chose. This means, which of these two constructed $\proj_{L^\perp} \shomob$-free sets will yield an $\shomob$-free containing the given solution $(\xlp, \ylp)$. We answer this next.
\begin{lemma} \label{lemma:s_free_contains_inf_point}
  Consider $(\xlp, \ylp)$ such that $\|\xlp\| > \|\ylp\|$ and $a^\T \xlp + d^\T
  \ylp \leq 0$ and $\lambda = \frac{\xlp}{\|\xlp\|}$.
  Then, the projection of $(\xlp, \ylp)$ onto $L^\perp$ is in the interior of
  the epigraph of $\phifun$. 
\end{lemma}
\begin{proof}
  The projection of $(\xlp, \ylp)$ onto $L^\perp$ is given by $(\lambda^\T \xlp,
  \ylp)$.
  Then, $\phifun(\ylp) = \max_x \{ \lambda^\T x \st \|x\| \leq \|\ylp\|, a^\T
  x + d^\T \ylp \leq 0 \} \leq \lambda^\T \lambda \|\ylp\| = \|\ylp\|$.
  Thus, $\lambda^\T \xlp = \|\xlp\| > \|\ylp\| \geq \phifun(\ylp)$.
\end{proof}

\subsubsection{Back to the original space}
Finally, we use the above to construct $\shomob$-free sets, i.e., in the original space.
Embedded in $\mathbb{R}^{n+m}$, the epigraph of $\phifun$ is
$\{ (t \lambda, y) \st y \in \mathbb{R}^m,\ \phifun(y) \leq t \}$.
Thus,
\begin{align}
  \maxhomobb &= \{ (t \lambda, y) \st y \in \mathbb{R}^m,\ \phifun(y) \leq t \} + L \nonumber \\
  &= \{ (t \lambda + z, y) \st y \in \mathbb{R}^m,\ \lambda^T z = 0,\ \phifun(y) \leq t \}  \nonumber  \\
  &= \{ (x, y) \st \phifun(y) \leq \lambda^\T x \}. \label{eq:cphilambdadef}
\end{align}
As a summary we prove that $\maxhomobb$ is maximal $\shomob$-free without going
through the projection.
\begin{proposition} \label{prop:max_homo_bad}
  Let $\lambda \in D_1(0)$ and $\phifun(y) = \max_x \{\lambda^\T x \st (x,y) \in \shomob\}$.
  If $ \|a\| = 1 \geq \|d\|$, then $\maxhomobb = \{ (x, y) \st \phifun(y) \leq
  \lambda^\T x \}$ is maximal $\shomob$-free. 
  
  Additionally, if $(\xlp, \ylp) \notin \shomob$ is such that $a^\T \xlp + d^\T \ylp \leq 0$, letting $\lambda = \frac{\xlp}{\|\xlp\|}$ ensures $(\xlp, \ylp) \in \inte(\maxhomobb)$.
\end{proposition}
\begin{proof}
  We will prove that $\maxhomobb$ is convex, free and maximal.

  The convexity of $\maxhomobb$ follows directly from
  $\Cref{prop:phifun_sublinear}$.
  Also, $\maxhomobb$ is $\shomob$-free since if $(x,y) \in \shomob$, then
  $\phifun(y) \geq \lambda^\T x$.
  Therefore, $(x,y)$ is not in the interior of $\maxhomobb$.

  We now focus on proving maximality. In the cases where $\phifun$ is differentiable in $\mathbb{R}^m\setminus \{0\}$ we can directly write 
  \[
    \maxhomobb = \{ (x,y) \in \mathbb{R}^{n+m} \st \nabla \phifun(\beta)^\T
    y \leq \lambda^\T x,\ \forall \beta \in D_1(0)\}.
  \]
  Let $\beta \in D_1(0)$ and let $\xbeta$ be the optimal
  solution of the problem \eqref{eq:phi_function} which defines $\phifun(\beta)$. 
  %(see \eqref{eq:phi_primal_solution} for the explicit expression for $\xbeta$). 
  That is, $\lambda^\T \xbeta = \phifun(\beta)$.
  By \Cref{lemma:exposing_ineq}, the inequality $-\lambda^\T x + \nabla \phifun(\beta)^\T y \leq 0$ is exposed by $(\xbeta,\beta)$.
  
  The only remaining case is $\|d\| = 1 \, \wedge \, m>1$, where $\phifun$ is only differentiable in $D_1(0) \setminus \{d\}$. Since in this case $m>1$ we can safely remove a single inequality from the outer-description of $\maxhomobb$ without affecting it, i.e., 
  \[
    \maxhomobb = \{ (x,y) \in \mathbb{R}^{n+m} \st \nabla \phifun(\beta)^\T
    y \leq \lambda^\T x,\ \forall \beta \in D_1(0)\setminus \{d\}\}.
  \]
  Using the same argument as above we can find an exposing point of each inequality $-\lambda^\T x + \nabla \phifun(\beta)^\T y \leq 0$ for $\beta \in D_1(0) \setminus \{d\}$.

  The fact that $(\xlp, \ylp) \in \inte(\maxhomobb)$ when $\lambda = \frac{\xlp}{\|\xlp\|}$ follows directly since $\maxhomo \subseteq \maxhomobb$. 

\end{proof}

\begin{example} \label{ex:simple3dmaximalhomo}
  Let us recall the set $\shomorunning$ in \Cref{ex:simple3D}.
   \[\shomorunning = \{(x_1, x_2, y) \in \mathbb{R}^3 \st \|x \| \leq |y|, \quad a^\T x + dy \leq 0 \} \]
 
   with $a = (-1/\sqrt{2},1/\sqrt{2})^\T $, $d=1/\sqrt{2}$, and $(\bar{x},\bar{y}) = (-1,-1,0)^\T$. In \Cref{fig:simpleexample3dbad} we showed that the set $\maxhomo$ is $\shomorunning$-free but not maximal, and $\maxhomobg$ is not $\shomorunning$-free. In \Cref{fig:simpleexample3dmaximal} we show the set $\maxhomobb$, which is maximal $\shomorunning$-free. For this example, we know $\lambda^\T a = 0$, thus 
   \[\lambda^\T a \| y\|  + d^\T y \leq 0 \Longleftrightarrow y\leq 0. \]
   A simple calculation using \eqref{eq:phi_characterization} yields
   \[\phifun(y) =  \begin{cases}
     - y ,
      &\text{ if } y \leq 0 \\
      \frac{y}{\sqrt{2}} 
      &\text{ if } y > 0 
  \end{cases} \]
   \begin{figure}
      \centering
      \includegraphics[scale=0.3]{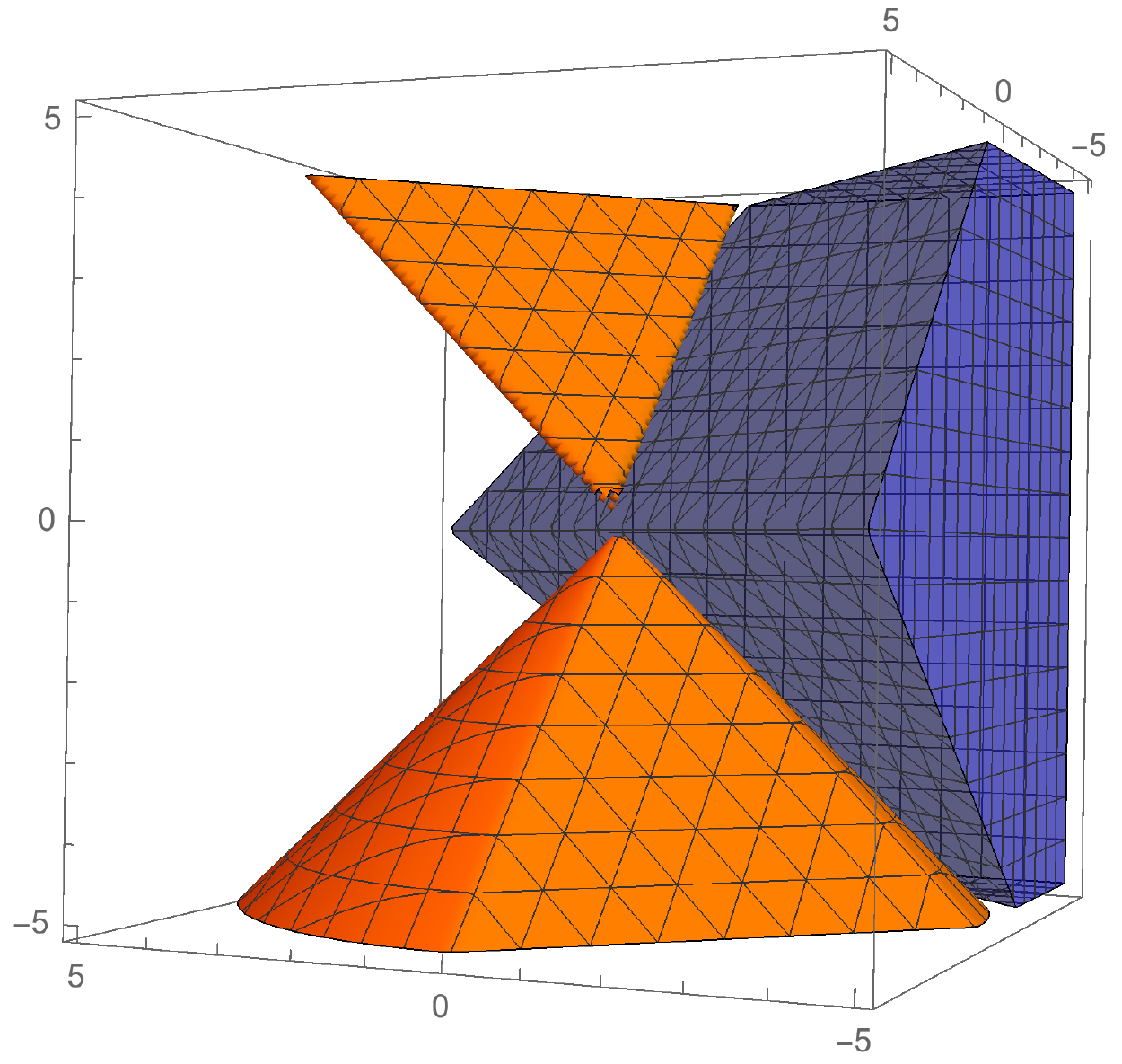}
      \caption{$\shomorunning$ in \Cref{ex:simple3D} (orange) and $\maxhomobb$ set (blue). The latter is maximal $\shomorunning$-free.}
      \label{fig:simpleexample3dmaximal}
  \end{figure}
\end{example}

\begin{remark} \label{rmk:lambda_eq_a}
  As we saw in the proof of \Cref{prop:phifun_sublinear} if $\lambda = a$, then
  $\phifun(y) = -d^\T y$.
  This implies that $\maxhomobb = \{(x,y) \st a^\T x + d^\T y \geq 0\}$.
  By definition, this set does not contain any point from $a^\T x + d^\T
  y \leq 0$ in its interior, thus, it is a very uninteresting maximal $\shomob$-free set.
  One is usually interested in constructing a maximal $\shomob$-free set that
  contain a point $(\xlp, \ylp)$ that satisfies $a^\T x + d^\T y \leq 0$.
  Hence, by Lemma~\ref{lemma:s_free_contains_inf_point}, whenever we assume that
  $\lambda = \frac{\xlp}{\|\xlp\|}$ where $a^\T \xlp + d^\T \ylp \leq 0$ and
  $\|\xlp\| > \|\ylp\|$, it will automatically hold that $\lambda \neq a$.
\end{remark}

\begin{remark} \label{rmk:whats_going_on}
  At this point we would like to show some relations between $\maxhomo$,
  $\maxhomobb$ and $\maxhomobg$.
  The inequalities defining $\maxhomo$ are $(-\lambda,\beta)$ for $\beta \in
  D_1(0)$.
  Equivalently, the polar of $\maxhomo$ is the cone generated by $\{-\lambda\} \times \conv D_1(0)
  = \{-\lambda\} \times B_1(0)$.

  The inequalities defining $\maxhomobg$ are $(-\lambda,\beta)$ for $\beta \in
  G(\lambda) = \{ \beta \in D_1(0) \st \beta \lambda^\T a + d^\T \beta \leq
  0 \}$.
  Equivalently, the polar of $\maxhomobg$ is the cone generated by $\{-\lambda\} \times \conv
  G(\lambda)$.

  The inequalities defining $\maxhomobb$ are $(-\lambda, \nabla \phifun(\beta))$
  for $\beta \in D_1(0)$.
  When $\beta \in G(\lambda)$, then $\phifun(y) = \|y\|$ and so the inequalities
  are $(-\lambda, \beta)$.
  In other words, some inequalities defining $\maxhomobb$ coincide with the
  inequalities defining $\maxhomobg$ and $\maxhomo$.
  Thus, when $\maxhomobb$ is convex (i.e., when $\|a\| \geq \|d\|$), there is
  a region where all three convex sets look the same.
  In terms of the polars, when $\|a\| \geq \|d\|$, the polar of $\maxhomobb$ is
  between the polars of $\maxhomobg$ and $\maxhomo$.
  This is depicted in Figure~\ref{fig:polars}.
  \begin{figure}
    \centering
    \includegraphics[scale=0.3]{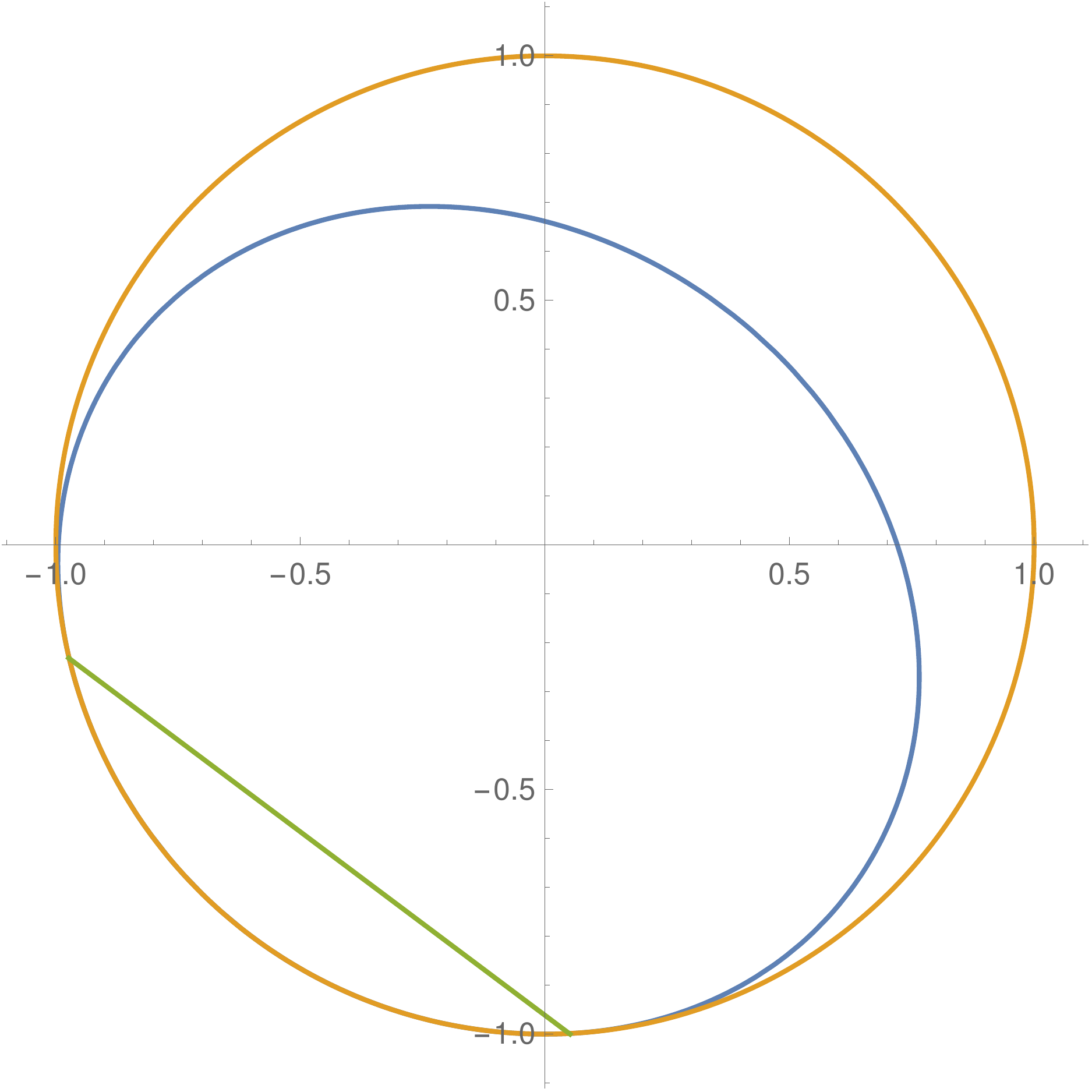}
    \caption{%
      Let $a = (\tfrac{3}{5},-\frac{4}{5}), d = (\frac{3}{10},\frac{2}{5})$, and
      $\lambda = (\frac{63}{65},\frac{16}{65})$. The boundary of the $y$
      coordinates of the polars of $\maxhomo$, $\maxhomobg$, and $\maxhomobb$
      are depicted in orange, green, and blue, respectively.
      They all coincide below the green line.
    }
    \label{fig:polars}
  \end{figure}
  %% code for figure
  % a = {3/5, -4/5}
  % l = {63/65, 16/65}
  % d = {3/10, 2/5}
  % phifun[y1_, y2_] = Evaluate[MaxValue[ l.{x1, x2}, {x1, x2}.{x1, x2} <= {y1,
  % y2}.{y1, y2} && a.{x1, x2} + d.{y1, y2} <= 0, {x1, x2}]]
  %
  % Show[ParametricPlot[ Evaluate[Grad[phifun[y1, y2], {y1, y2}]] /. {y1 ->
  % Cos[t], y2 -> Sin[t]}, {t, 0, 2 Pi}, PlotPoints -> 60], ContourPlot[{0 ==
  % 0 (*just to shift the colors of the plots*), {y1, y2}.{y1, y2} == 1, a.l
  % + d.{y1, y2} == 0}, {y1, -1, 1}, {y2, -1, 1}, RegionFunction ->
  % Function[{y1, y2}, {y1, y2}.{y1, y2} <= 1], PlotPoints -> 35], PlotRange ->
  % Full]
  %
\end{remark}

\section{Non-homogeneous quadratics} \label{sec:nonhomo}

As discussed at the beginning of the previous section, we now study a general
non-homogeneous quadratic which can be written as
\begin{align*}
  \snonhomo = \{ (x,y) \in \mathbb{R}^{n+m} \st \|x\| \leq \|y\|,\, a^\T x + d^\T y = -1 \}.
\end{align*}
We assume we are given $(\bar x, \bar y)$ such that
\[ \|\bar x\| > \| \bar y\|,\, a^\T \bar x + d^\T \bar y = -1. \]
Much like in \Cref{sec:homowithhomo}, we begin by dismissing a simple case.

\begin{remark}
  The case $\|a\|\leq \|d\| \, \wedge \, m=1$ can be treated separately. Note that, as opposed to the analogous analysis at the beginning of \Cref{sec:homowithhomo}, here we include the case where the norms are equal.
As already noted in \Cref{rmk:less_cases_in_non_homo}, we should expect
$\snonhomo$ to be convex in this case.
Indeed, as $d\neq 0$ (if not, then $a = 0$ and $\snonhomo = \emptyset$) we can
write $y = \frac{1}{d} (-1 - a^\T x ) $ and consequently
\begin{align*}
\snonhomo &=\{ (x,y) \in \mathbb{R}^{n+1} \st \|x\|^2 \leq \frac{1}{d^2}(1 + 2a^\T x  + (a^\T x)^2 ), a^\T x + d^\T y = -1\}\\
%& =\{ (x,y) \in \mathbb{R}^{n+1} \st x^\T I x \leq \frac{1}{d^2}(1 + 2a^\T x  + x^T a a^\T x ), a^\T x + d^\T y = -1\}\\
& =\{ (x,y) \in \mathbb{R}^{n+1} \st x^\T \left(I - \frac{1}{d^2}aa^\T \right) x  - \frac{1}{d^2}(1 + 2a^\T x ) \leq 0, a^\T x + d^\T y = -1\}.
\end{align*}
Since $I - \frac{1}{d^2}aa^\T$ is positive semi-definite whenever $|d| \geq \|
a\|$, the set $\snonhomo$ is convex.
Thus, a maximal $\snonhomo$-free set, or even directly a cutting plane, can be
obtained using a supporting hyperplane.
\end{remark}

Similarly to \Cref{sec:homowithhomo}, we distinguish the following cases:
\begin{description}
  \item[Case 1] $\|a\| \leq \|d\|\, \wedge \, m > 1$.
  \item[Case 2] $\|a\| > \|d\|$.
\end{description}

Since $\snonhomo \subsetneq \shomob$, then $\maxhomobg$ ($\maxhomobb$) is
$\snonhomo$-free in Case 1 (Case 2) as per \Cref{sec:homowithhomo}.
It is natural to wonder whether these sets are maximal already.
\subsection{Case 1: $\|a\| \leq \|d\|\, \wedge \, m > 1$}
\label{subsec:easycasenonhomo}

The technique we used to prove maximality of $\maxhomobg$ with respect to $\shomob$ is to exploit that $\maxhomobg$ is defined by the inequalities of $\maxhomo$ exposed by elements in $\shomob$.
Following this approach, we study which inequalities of $\maxhomobg$ are exposed by
a point of $\snonhomo$.
Recall that
\[
  \maxhomobg = \{ (x,y) \in \mathbb{R}^{n+m} \st - \lambda^\T x + \beta^\T
  y \leq 0,\ \forall \beta \in G(\lambda)\},
\]
where
\[
  G(\lambda) = \{ \beta \in \mathbb{R}^m \st \|\beta\| = 1,\ a^\T \lambda + d^\T
  \beta \leq 0\}.
\]
Consider an inequality in the definition of $\maxhomobg$ given by $(-\lambda, \beta)$ such that $a^\T \lambda + d^\T \beta
< 0$.
Then, the point $(\lambda, \beta) \in \shomob$ can be scaled by $\mu
= \frac{-1}{a^\T \lambda + d^\T \beta}$ to the exposing point $\mu(\lambda,
\beta) \in \snonhomo$.
Thus, almost every inequality describing $\maxhomobg$ is exposed by points of
$\snonhomo$.
Furthermore, we can simply remove the inequalities that are not exposed by points of
$\snonhomo$ from $\maxhomobg$ without changing the set $\maxhomobg$. We specify this next.

\begin{theorem}\label{thm:max_nonhomo_good}
  Let $\lambda = \frac{\xlp}{\|\xlp\|}$, 
  \[
    \HH = \{ (x,y) \in \mathbb{R}^{n+m} \st a^\T x + d^\T y = -1 \}
  \]
  and
  \[
    \shomob = \{ (x,y) \mathbb{R}^{n+m} \st \|x\| \leq \|y\|,\ a^\T x + d^\T
    y \leq 0 \},
  \]
  where $\|a\| \leq \|d\|\, \wedge \, m > 1$.
  Then, $\maxhomobg$ is maximal $\shomob$-free with respect to $\HH$ and contains $(\xlp,\ylp)$ in its interior.
\end{theorem}
\begin{proof}
  By \Cref{prop:max_homo_good}, we know that $\maxhomobg$ is maximal $\shomob$-free.
  Thus, $\maxhomobg$ is $\shomob$-free with respect to $\HH$. To prove maximality, we note that thanks to $m>1$:
  \[
    \maxhomobg = \{ (x,y) \in \mathbb{R}^{n+m} \st -\lambda^\T x + \beta^\T
    y \leq 0,\ \forall \beta \in \relinte(G(\lambda)) \},
  \]
  where
  \[
    \relinte(G(\lambda)) = \{ \beta \in \mathbb{R}^m \st \|\beta\| = 1, a^\T
    \lambda + d^\T \beta < 0\}
  \]
  is the relative interior of $G(\lambda)$.
  Consider $\beta_0 \in \relinte(G(\lambda))$.
  As we saw in \Cref{prop:largest_maybe_free}, $(\lambda, \beta_0) \in \maxhomobg
  \cap \shomob$ exposes the inequality $(-\lambda,\beta_0)$.
  As $\maxhomobg \cap \shomob$ is a (non-convex) cone, we have that for any $\mu
  > 0$, $\mu(\lambda, \beta_0) \in \maxhomobg \cap \shomob$ exposes the
  inequality $(-\lambda,\beta_0)$.
  Since $a^\T \lambda + d^\T \beta_0 < 0$, $\mu = -\frac{1}{a^\T \lambda + d^\T
  \beta_0} > 0$ and so
  \begin{equation}\label{eq:exposingpoint}
    -\frac{(\lambda, \beta_0)}{a^\T \lambda + d^\T \beta_0} \in \shomob \cap
    \HH  \cap \maxhomobg,
  \end{equation}
  exposes the inequality $(-\lambda,\beta_0)$.
  The claim now follows from \Cref{thm:exposed_maximality_wrt}.
\end{proof}

The above theorem states that obtaining a maximal $\snonhomo$-free set in this case amounts to simply using the maximal $\shomob$-free set $\maxhomobg$, and then intersecting with $H$. Recall that $\snonhomo = \shomob \cap H$. The next case is considerably different.

\subsection{Case 2: $\|a\| > \|d\|$}

We begin with an important remark regarding an assumption made in the analogous case of the previous section.
\begin{remark}
  Since in this case $\|a\| > 0$, we can, again, assume that $\|a\|=1$. Indeed, we can always rescale the variables $(x,y)$ by $\|a\|$ to obtain such requirement.

  Also note that since $\|d\| < \|a\| = 1$, then $\phifun$ is differentiable in
  $D_1(0)$. See \Cref{prop:phifun_sublinear}.
\end{remark}

Unfortunately, in this case the maximality of $\maxhomobb$ with respect to $\shomob$ does not carry over to $\snonhomo$, as the following example
shows.

\begin{example} \label{ex:simpleexnonhomo}
  We continue with $\shomorunning$ defined in \Cref{ex:simple3D}. In \Cref{fig:simpleexample3dmaximal} we showed how $\maxhomobb$ gives us a maximal $\shomorunning$-free set. If we now consider    
  \[\HH = \{ (x,y) \in \mathbb{R}^{n+m} \st a^\T x + d^\T y = -1 \} \]
  with $a = (-1/\sqrt{2},1/\sqrt{2})^\T $ and $d=1/\sqrt{2}$, we do not necessarily obtain that $\maxhomobb \cap H$ is maximal $\shomorunning \cap H$-free. In \Cref{fig:simpleexamplenonhomo} we illustrate this issue.

  \begin{figure}
    \begin{subfigure}{.45\textwidth}
      \centering
      \includegraphics[scale=0.3]{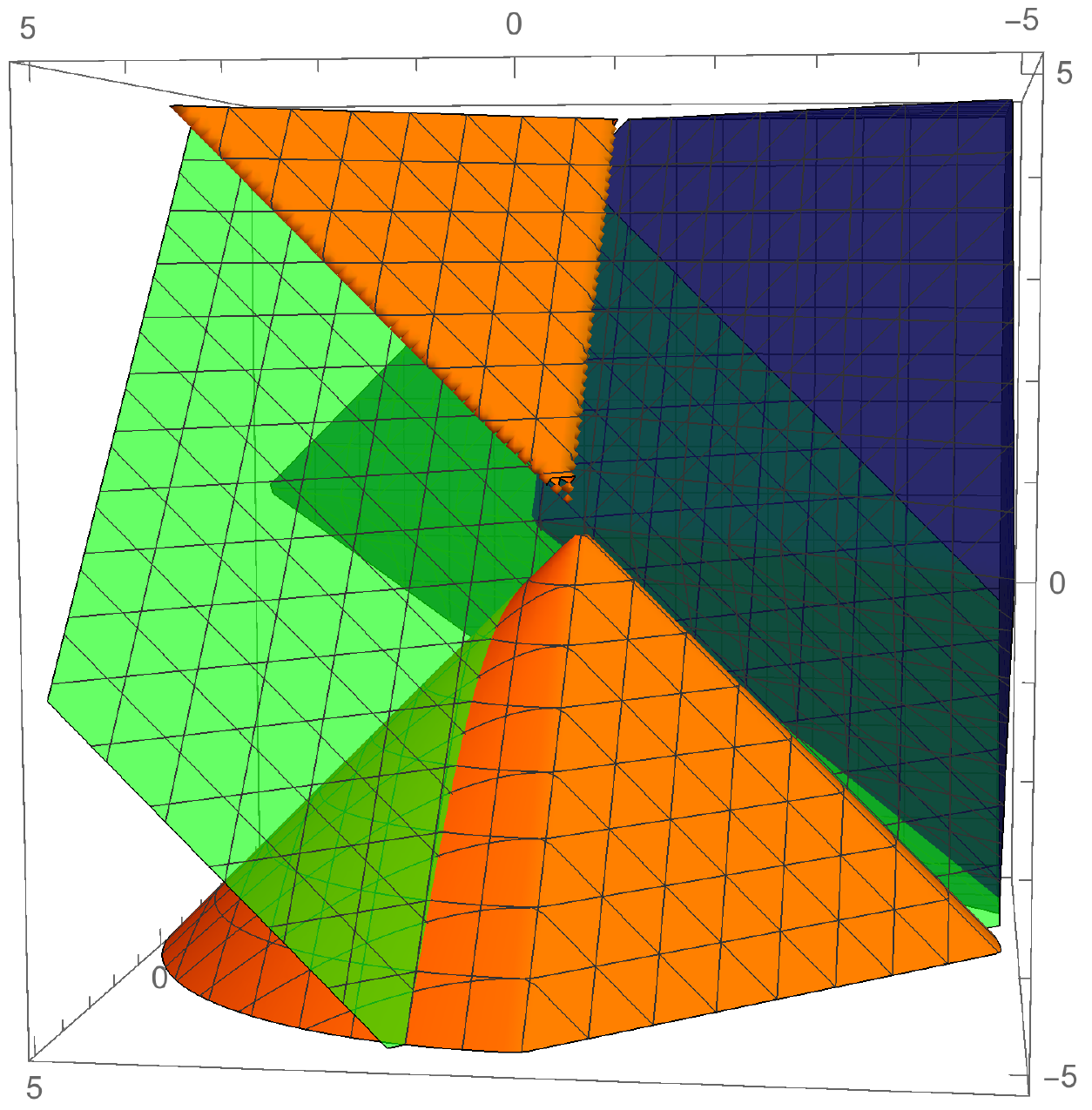}
      \caption{$\shomorunning$ (orange), $H$ (green) and $\maxhomobb$ (blue).}
      \label{fig:sub-simpleexamplenonhomo-3d}
    \end{subfigure}\hspace{.1cm}
    \begin{subfigure}{.45\textwidth}
      \centering
      \includegraphics[scale=0.3]{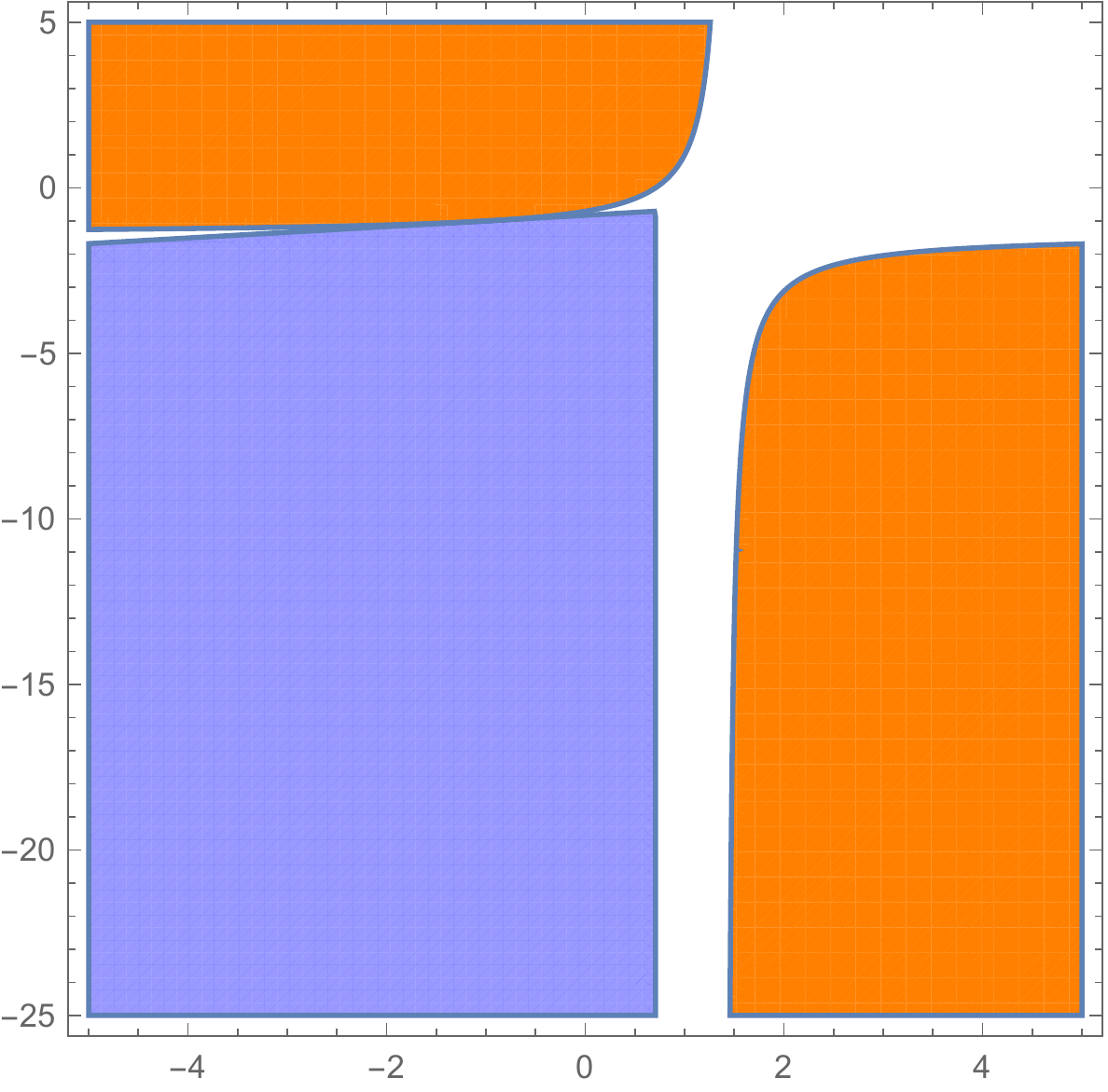}
      \caption{Projection onto $(x_1,x_2)$ of $\shomorunning\cap H$ (orange) and $\maxhomobb \cap H$ (blue). One of the facets of $\maxhomobb \cap H$ has a gap with the boundary of $\shomorunning\cap H$.}
      \label{fig:sub-simpleexamplenonhomo-2d}
    \end{subfigure}
    \caption{Plots of $\shomorunning$, $H$ and $\maxhomobb$ as defined in \Cref{ex:simpleexnonhomo} showing that $\maxhomobb$ is not necessarily maximal $\shomorunning$-free with respect to $H$ in the case $\|a\| > \|d\|$.}
    \label{fig:simpleexamplenonhomo}
  \end{figure}
\end{example}
\Cref{fig:simpleexamplenonhomo} of the previous example displays an interesting feature though: the
inequalities defining $\maxhomobb$ seem to have the correct ``slope'' and just
need to be translated.
We conjecture, then, that in order to find a maximal $\snonhomo$-free set, we
only need to adequately relax the inequalities of $\maxhomobb$.

\subsubsection{Set-up}
Recall that
\begin{align*}
  \maxhomobb &= \{ (x, y) \st \phifun(y) \leq \lambda^\T x \} \\
  &= \{ (x, y) \st -\lambda^\T x + \nabla \phifun(\beta)^\T y \leq 0, \forall \beta
  \in D_1(0) \}. 
\end{align*}
We denote by $r(\beta)$ the amount by which we need to relax each inequality of
$\maxhomobb$ such that
\begin{equation} \label{eq:maxnonhomob}
  \maxnonhomob
  = \{ (x,y) \st -\lambda^\T x + \nabla \phifun(\beta)^\T y \leq r(\beta), \forall
  \beta \in D_1(0) \},
\end{equation}
is $\snonhomo$-free.
Note that when $\beta$ satisfies $\lambda^\T a + d^\T \beta < 0$, the
inequalities of $\maxhomobb$ are the same as the ones of $\maxhomobg$ (see also
Remark~\ref{rmk:whats_going_on}) and, just like in \Cref{subsec:easycasenonhomo},
they have exposing points in $\snonhomo$.
An inequality of this type can be seen in
\Cref{fig:sub-simpleexamplenonhomo-2d}: it is the inequality of $\maxhomobb$ tangent to $S$ at one of its exposing points.
Thus, we expect that $r(\beta) = 0$ when $\lambda^\T a + d^\T \beta < 0$.
In the following we find $r(\beta)$ when $\lambda^\T a + d^\T
\beta \geq 0$ and show maximality of the resulting set. 

Following the spirit of \Cref{subsec:homowithhomohard}, not all statement in this section require $\lambda = \frac{\xlp}{\|\xlp\|}$. However, we assume $\lambda \neq \pm a$. This assumption, however, is not restrictive when constructing maximal $\snonhomo$-free sets, as the following remark shows.

\begin{remark} \label{rmk:lambda_eq_pma}
  If $\lambda = -a$, then for every $\beta \in D_1(0)$ it holds  that $\lambda^\T a + d^\T \beta < 0$. In this case $r(\beta)$ will be simply defined as 0 everywhere and $\maxnonhomob = \maxhomobb$. This means all inequalities defining $\maxnonhomob$ have an exposing point in $\snonhomo$ and maximality follows directly.
  
  On the other hand, if we take $\lambda = \frac{\xlp}{\|\xlp\| }$ with $(\bar{x}, \bar{y}) \in H$ and $\|\bar x\| > \| \bar{y} \|$, we have that if additionally $\lambda = a$
  \begin{align*}
    a^\T \bar{x} + d^\T \bar{y} = -1 &\Longleftrightarrow \|\bar{x}\| + d^\T \bar{y} = -1 \\
    & \Longrightarrow \|\bar{y}\| + d^\T \bar{y} < -1.
  \end{align*}
  The latter cannot be, as $\|d\| < 1$.
\end{remark}
\begin{remark}
  The assumption $\lambda \neq \pm a$ has an unexpected consequence: as $\lambda \neq \pm a$ and $\|a\| = \|\lambda\| = 1$, it must hold that $n \geq 2$. This implicit assumption, however, does not present an issue: whenever $n = 1$ either $\lambda = a$ or $\lambda = -a$. By \Cref{rmk:lambda_eq_pma}, if we use $\lambda = \frac{\xlp}{\|\xlp\| }$, then $\lambda = -a $. Thus, $\maxnonhomob = \maxhomobb$ and maximality holds.
\end{remark}

\subsubsection{Construction of $r(\beta)$}
Let $\beta \in D_1(0)$ be such that $\lambda^\T a + d^\T \beta \geq 0$.
Then, the face of $\maxhomobb$ defined by the valid inequality $-\lambda^\T x + \nabla \phifun(\beta)^\T y \leq 0$ does not intersect $\snonhomo$. See \Cref{lemma:donttouch} for a proof of this statement.

In particular, the inequality is not exposed by any point in $\snonhomo \cap
\maxhomobb$.
However, it is exposed by $(\xbeta, \beta) \in \shomob$, where $\xbeta$ is given
by \eqref{eq:phi_primal_solution} (see the proof of \Cref{prop:max_homo_bad}).
Note that $(\xbeta, \beta) \in \Hzero = \{ (x,y) \st a^\T x + d^\T y = 0\}$, as
otherwise we can scale it so that it belongs to $\snonhomo$.

The quantity $r(\beta)$ is the amount we need to relax the inequality in order
to be an ``asymptote'', and we compute it as follows.
We first find a sequence of points, $(x_n ,y_n)_{n\in \mathbb{N}}$, in $\shomob$ that converge
to $(\xbeta, \beta)$, enforcing that no element of the sequence belongs to $\Hzero$.
If we find such sequence, then every $(x_n, y_n) \in \shomob$ can be scaled to
be in $\snonhomo$:
\[
  z_n = -\frac{(x_n, y_n)}{a^\T x_n + d^\T y_n} \in \snonhomo.
\]
This last scaled sequence diverges, as the denominator goes to 0 due to $(x_n, y_n) \to (\xbeta, \beta) \in \Hzero$. The idea is that the violation $(-\lambda, \nabla
\phifun(\beta) )^\T z_n$ given by this sequence will give us, in the limit, the maximum relaxation that will ensure $\snonhomo$-freeness (see \Cref{fig:sub-simpleexamplenonhomo-2d-points}).
Then, we would define
\[
  r(\beta) = \lim_{n\to \infty}(-\lambda, \nabla
  \phifun(\beta) )^\T z_n = -\lim_{n \to \infty} \frac{-\lambda^\T x_n + \nabla
  \phifun(\beta)^\T y_n }{a^\T x_n + d^\T y_n}.
\]

We remark that this limit is what we intuitively aim for, but it might not even
be well defined in general. In what follows, we construct a sequence that yields
a closed-form expression for the above limit. Additionally, we show that such
definition of $r(\beta)$ yields the desired maximal $\snonhomo$-free set.

\paragraph{The sequence.}
Our goal is to find a sequence $(x_n ,y_n)_n$ such that $(x_n, y_n) \in \shomob$,
$a^\T x_n + d^\T y_n < 0$ and $(x_n, y_n) \to (\xbeta, \beta)$.
We take $y_n = \beta$ and $x_n$ such that $\|x_n\| = \|\beta\| = 1$, $a^\T x_n + d^\T \beta < 0$ and $x_n \to
\xbeta$.
Note that these always exists as $\|a\| = 1$ and $\|d\| < 1$. We illustrate such a sequence with our running example.
\begin{example}\label{ex:simpleexnonhomo-points}
  We continue with  \Cref{ex:simpleexnonhomo}. As we mentioned in \Cref{ex:simple3dmaximalhomo}, in this case
  \[\phifun(y) =  \begin{cases}
     - y ,
      &\text{ if } y \leq 0 \\
      \frac{y}{\sqrt{2}} 
      &\text{ if } y > 0 
  \end{cases} 
  \]
  and since $\lambda = \frac{1}{\sqrt{2}}(-1,-1)^\T$, we see that
  \begin{subequations}
  \begin{align}
  \maxhomobb = \{ (x,y) \st & \frac{1}{\sqrt{2}}(x_1 + x_2) - y \leq 0, \label{eq:ineq1}\\
  & \frac{1}{\sqrt{2}}(x_1 + x_2) + \frac{1}{\sqrt{2}}y \leq 0 \}. \label{eq:ineq2}
  \end{align}

\end{subequations}
It is not hard to check that $-(\tfrac{1}{\sqrt{2}},\tfrac{1}{\sqrt{2}},\sqrt{2}) \in \shomorunning \cap H \cap \maxhomobb$ exposes inequality \eqref{eq:ineq1}. This is the tangent point in \Cref{fig:sub-simpleexamplenonhomo-2d} we discussed above.

On the other hand, \eqref{eq:ineq2}, which is obtained from $\beta=1$, does not have an exposing point in $\shomorunning \cap H \cap \maxhomobb$, and corresponds to an inequality we should relax as per our discussion. This inequality, however, is exposed by $(\xbeta, \beta) =(0,-1, 1) \in \shomorunning \cap \maxhomobb$. Consider now the sequence defined as
\[(x_n,y_n) = \left(\frac{1}{\sqrt{n^2+1}} , -\frac{n}{\sqrt{n^2+1}}, 1\right) \in \shomorunning.\]
Clearly the limit of this sequence is $(0,-1, 1)$ and 
\[a^\T x_n + d^\T y_n = \frac{1}{\sqrt{2}}\left(-\frac{1}{\sqrt{n^2+1}} -\frac{n}{\sqrt{n^2+1}} + 1\right) < 0. \]

Now we let 
\[z_n = -\frac{(x_n, y_n)}{a^\T x_n + d^\T y_n} \in \shomorunning\cap H. \]
As we mention above, this sequence diverges. Continuing with \Cref{fig:simpleexamplenonhomo}, in \Cref{fig:sub-simpleexamplenonhomo-2d-points}, we plot the first two components of the sequence $(z_n)_{n\in \mathbb{N}}$ along with $\shomorunning\cap H$ and $\maxhomobb \cap H$. From this figure we can anticipate where our argument is going: the sequence $(z_n)_{n\in \mathbb{N}}$ moves along the boundary of $\shomorunning \cap H$ towards an ``asymptote'' from where we can deduce $r(\beta)$. The latter is given by the gap between inequality \eqref{eq:ineq2} and the asymptote.
\begin{figure}
  \centering
  \includegraphics[scale=0.3]{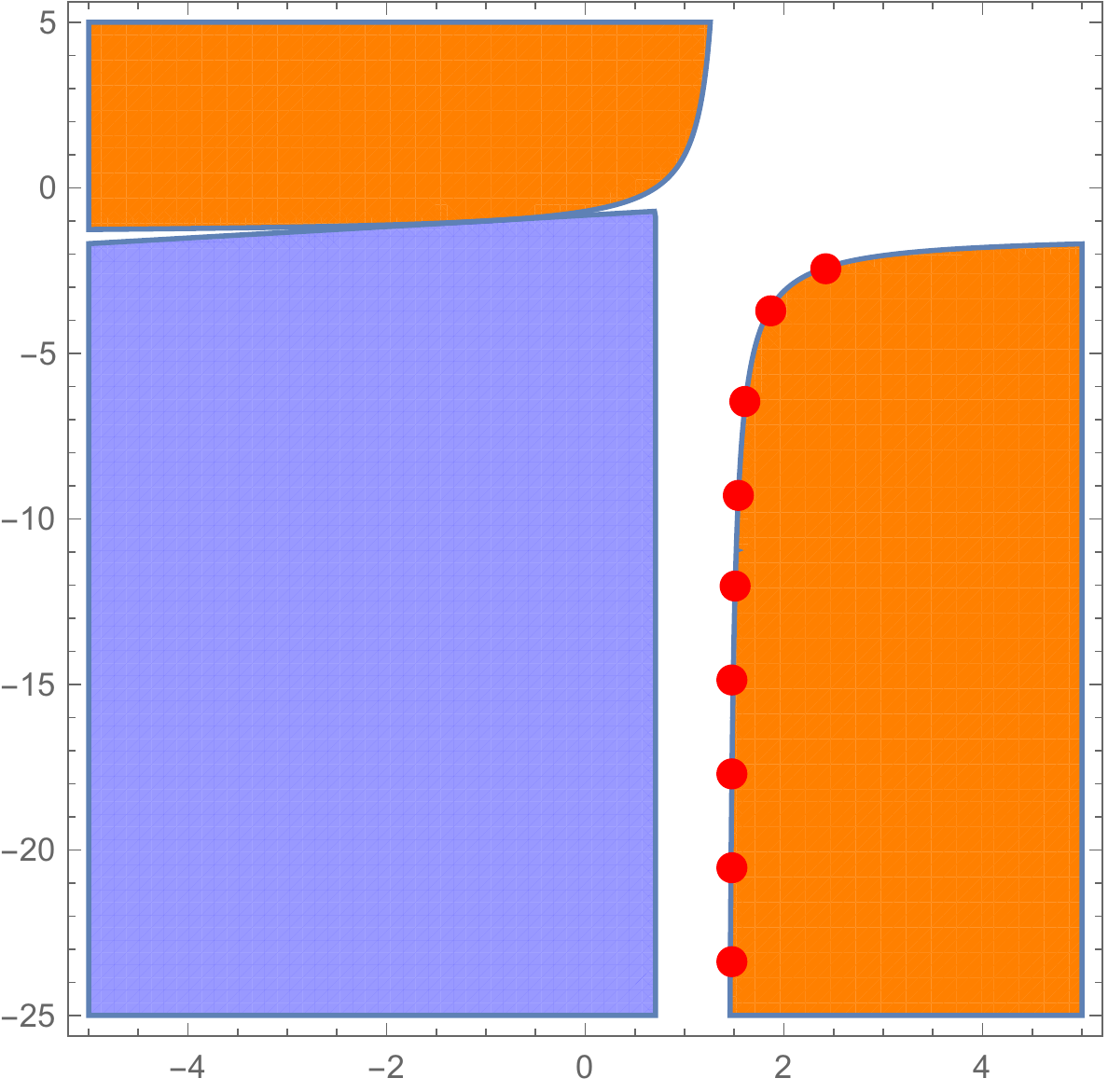}
  \caption{Projection onto $(x_1,x_2)$ of $\shomorunning\cap H$ (orange) and $\maxhomobb$ (blue), along with the first two coordinates of the sequence $(z_n)_{n\in \mathbb{N}}$ defined in \Cref{ex:simpleexnonhomo-points} for several values of $n$ (red). The sequence is diverging ``downwards''.}
  \label{fig:sub-simpleexamplenonhomo-2d-points}
\end{figure}
\end{example}

\paragraph{Computing the limit.}
Here we compute
\[
  r(\beta) = -\lim_{n \to \infty} \frac{-\lambda^\T x_n + \nabla
  \phifun(\beta)^\T y_n }{a^\T x_n + d^\T y_n}.
\]
We proceed to rewrite the limit.

Since $y_n = \beta$ and $\xbeta$ is the optimal solution of
\eqref{eq:phi_function}, we have:
\begin{align*}
  \nabla \phifun(\beta)^\T y_n
  &= \phifun(\beta) = \lambda^\T \xbeta \\
  d^\T y_n
  &= -a^\T \xbeta.
\end{align*}
Thus,
\begin{align*}
  r(\beta)
  &= -\lim_{n \to \infty} \frac{-\lambda^\T x_n + \nabla \phifun(\beta)^\T
  y_n }{a^\T x_n + d^\T y_n} \\
  &= -\lim_{n \to \infty} \frac{-\lambda^\T x_n + \lambda^\T \xbeta}{a^\T x_n
  - a^\T \xbeta}\\
  &= \lim_{n \to \infty} \frac{\lambda^\T (x_n - \xbeta)}{a^\T (x_n
  - \xbeta)}.
\end{align*}
Notice that $\xbeta$ belongs to the 2 dimensional space generated by $\lambda$
and $a$, which we denote by $\Lambda$.
Note that it is indeed 2 dimensional, since $\lambda \neq \pm a$, see
Remark~\ref{rmk:lambda_eq_pma}.
Furthermore, we can assume that $x_n$ also belongs to $\Lambda$ as any other
component of $x_n$ is irrelevant for the value of the limit.
Indeed, as $\mathbb{R}^n = \Lambda \oplus \Lambda^\perp$, then $x_n = x_n^\|
+ x_n^\perp$, where $x_n^\| \in \Lambda$ and $x_n^\perp \in \Lambda^\perp$, and
\[
  \frac{\lambda^\T (x_n - \xbeta)}{a^\T (x_n - \xbeta)}
  =
  \frac{\lambda^\T (x_n^\| - \xbeta)}{a^\T (x_n^\| - \xbeta)}.
\]

To compute the limit observe that
\[
  \frac{\lambda^\T (x_n - \xbeta)}{a^\T (x_n - \xbeta)} = \frac{\lambda^\T
  \frac{x_n - \xbeta}{\|x_n - \xbeta\|}}{a^\T \frac{x_n - \xbeta}{\|x_n
- \xbeta\|}}.
\]
Notice that $\frac{x_n - \xbeta}{\|x_n - \xbeta\|}$ converges, as $x_n \in
\Lambda$, $\|x_n\| = 1$, and $x_n \to \xbeta$.
Let $\hat x$ be the limit and note that $\hat x$ is orthogonal to $\xbeta$.
Indeed,
\begin{align*}
  \xbeta^\T \hat x
  &= \lim_{n \to \infty} \xbeta^\T \frac{x_n - \xbeta}{\|x_n - \xbeta\|}
  \\
  &= \lim_{n \to \infty}  \frac{\xbeta^\T x_n - 1}{\|x_n - \xbeta\|} \\
  &= \lim_{n \to \infty}  -\frac{\|x_n - \xbeta\|^2}{2\|x_n - \xbeta\|} \\
  &= 0.
\end{align*}
Hence,
\[
  r(\beta) = \lim_{n \to \infty} \frac{\lambda^\T (x_n - \xbeta)}{a^\T (x_n
  - \xbeta)}
  = \frac{\lambda^\T \hat x}{a^\T \hat x}.
\]
Since we are interested in the quotient of $\lambda^\T \hat x$ and $a^\T \hat
x$, any multiple of $\hat x$ can be used, that is, any vector orthogonal to
$\xbeta$ in $\Lambda$.
Using $\lambda$ and $a$ as basis for $\Lambda$, we have that for
$x \in \Lambda$ with coordinates $x_\lambda$ and $x_a$, the vector $y$ with
coordinates $y_\lambda = -(x_a + x_\lambda \lambda^\T a)$ and
$y_a = x_\lambda + x_a \lambda^\T a$ is orthogonal to $x$.
Indeed,
\begin{align*}
  x^\T y
  &= (x_\lambda \lambda + x_a a)^\T (y_\lambda \lambda + y_a a) \\
  &= x_\lambda y_\lambda + x_a y_a + (x_\lambda y_a + x_a y_\lambda) \lambda^\T
  a \\
  &= (x_\lambda + x_a \lambda^\T a)y_\lambda + (x_a + x_\lambda \lambda^\T a)
  y_a \\
  &= 0.
\end{align*}
Thus, let $\tilde x = -(\xbeta_a
+ \xbeta_\lambda \lambda^\T a) \lambda + (\xbeta_\lambda + \xbeta_a \lambda^\T
a) a$.
Given that $\lambda^\T a + d^\T \beta \geq 0$, from
\eqref{eq:phi_primal_solution} (appendix) we have
\begin{equation} \label{eq:x_beta}
  \xbeta =
  \sqrt{\frac{1 - (d^\T \beta)^2}{1 - (\lambda^T a)^2}} \lambda
  -\left( d^\T \beta + \lambda^T a \sqrt{\frac{1 - (d^\T \beta)^2}{1
  - (\lambda^T a)^2}} \right) a.
\end{equation}
Note that while this last explicit formula for $\xbeta$ is the one stated for the case $\lambda^\T a + d^\T \beta > 0$, it also holds when $\lambda^\T a + d^\T \beta = 0$.
Therefore,
\begin{align*}
  \tilde x
  &=
  (d^\T \beta) \lambda
  +
  \left(
    \sqrt{\frac{1 - (d^\T \beta)^2}{1 - (\lambda^T a)^2}} -
    \left(
      d^\T \beta + \lambda^T a \sqrt{\frac{1 - (d^\T \beta)^2}{1 - (\lambda^T
      a)^2}}
    \right) \lambda^\T a
  \right)a \\
  &= (d^\T \beta) \lambda + \phifun(\beta) a.
\end{align*}

All together, we obtain
\begin{align*}
  r(\beta)
  = \frac{\lambda^\T \tilde x}{a^\T \tilde x}
  = \frac{d^\T \beta + \lambda^\T a \phifun(\beta)}{\phifun(\beta)
  + d^\T \beta \lambda^\T a}.
\end{align*}
Note that if $\lambda^\T a + d^\T \beta = 0$, then $r(\beta) = 0$. We summarize the above discussion in the following result.

\begin{lemma}\label{lemma:characterizationrbeta}
  Let $a, \lambda, \beta\in D_1(0)$, $d \in B_1(0)$, and $\lambda \neq \pm a$ be
  such that $\|d\| < \|a\|$ and $\lambda^\T a + d^\T \beta \geq 0$.
  Then, every sequence $(x_n)_{n\in \mathbb{N}} \subseteq \langle \lambda,
  a\rangle$ converging to $\xbeta$ such that $\|x_n\| = 1$ and $a^\T x_n + d^\T
  \beta < 0$, satisfies
  \[
    r(\beta) = \lim_{n \to \infty} \frac{\lambda^\T (x_n
      - \xbeta)}{a^\T (x_n - \xbeta)} = \frac{d^\T \beta + \lambda^\T
    a \phifun(\beta)}{\phifun(\beta) + d^\T \beta \lambda^\T a}.
  \]
  Such sequences are always guaranteed to exist.
\end{lemma}

Therefore, for $\beta \in D_1(0)$, we define
\begin{align*}
  r(\beta)
  =
  \begin{cases}
    0,
    &\text{ if } \lambda^\T a + d^\T \beta \leq 0 \\
    \frac{d^\T \beta + \lambda^\T a \phifun(\beta)}{\phifun(\beta) 
    + d^\T \beta \lambda^\T a},
    &\text{ otherwise. }
  \end{cases}
\end{align*}
We extend $r$ to $y \in \mathbb{R}^m\setminus \{0\}$ by $r(y)
= r(\tfrac{y}{\|y\|})$ and leave it undefined at $0$.

\begin{example}\label{ex:simpleexnonhomo-maximal}
  We continue with our running example in \Cref{ex:simpleexnonhomo-points}. In
  this case $r(-1) = 0$, and since $\phifun(\beta) = 1/\sqrt{2}$, $ \lambda^\T
  a = 0$ and $d=1/\sqrt{2}$ it can be checked that
  \[r(1) = 1. \]
  Now, let
  \begin{align*}
    C_1 = & \{ (x,y) \st -\lambda^\T x + \nabla \phifun(\beta)^\T y \leq
    r(\beta), \text{ for all } \beta \in D_1(0) \}\\
     = & \{ (x,y) \st \frac{1}{\sqrt{2}}(x_1 + x_2) - y \leq 0, \frac{1}{\sqrt{2}}(x_1 + x_2) + \frac{1}{\sqrt{2}}y \leq 1 \}. 
  \end{align*}

  \Cref{fig:simpleexamplenonhomo-maximal} shows the same plots as
  \Cref{fig:simpleexamplenonhomo} with $C_1$ instead of $\maxhomobb$.

\begin{figure}
  \begin{subfigure}{.45\textwidth}
    \centering
    \includegraphics[scale=0.3]{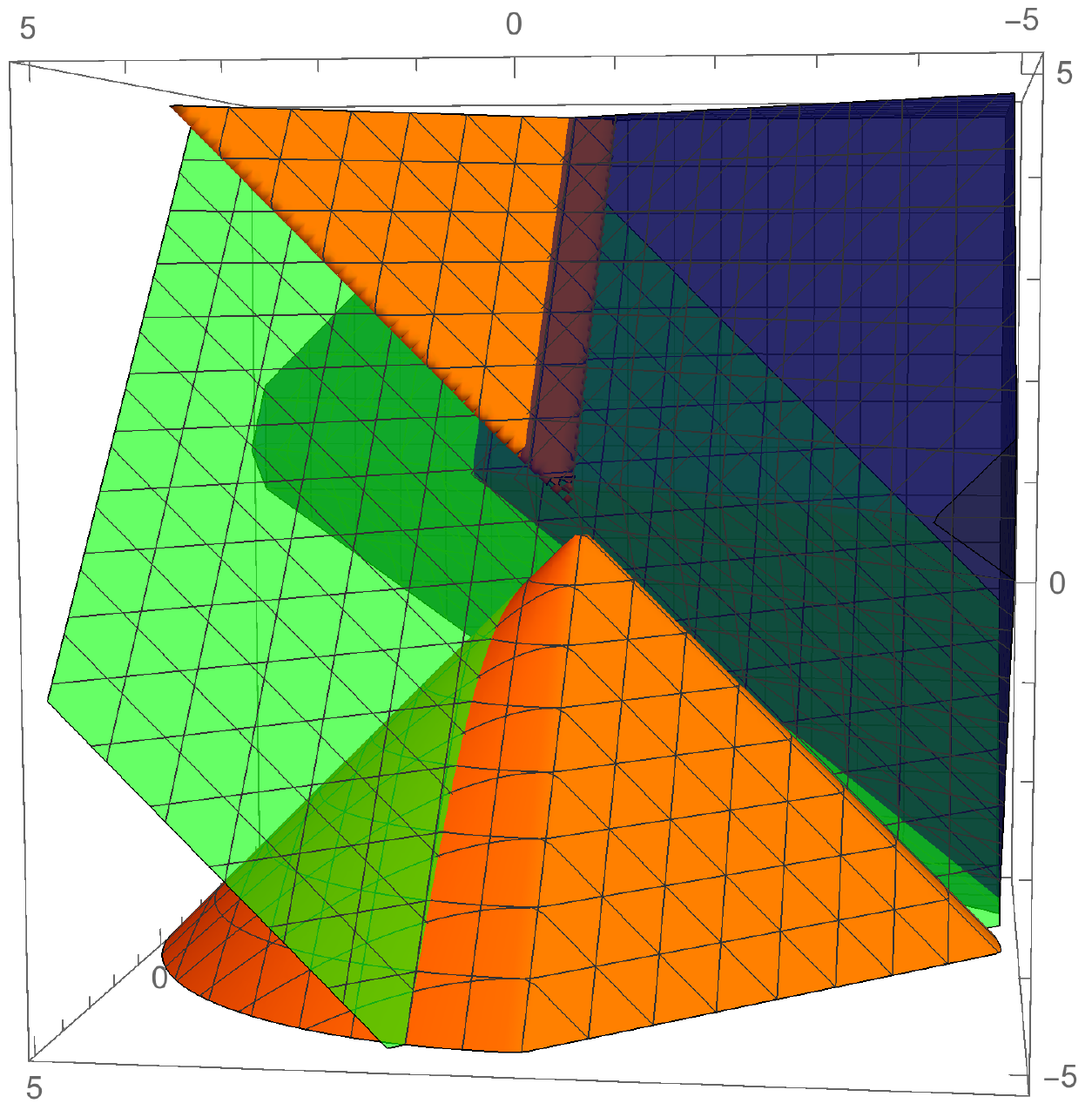}
    \caption{$\shomorunning$ (orange), $H$ (green) and $C_1$ (blue). In this case $C_1$ is no longer $\shomorunning$-free.}
    \label{fig:sub-simpleexamplenonhomo-3d-maximal}
  \end{subfigure}\hspace{.1cm}
  \begin{subfigure}{.45\textwidth}
    \centering
    \includegraphics[scale=0.3]{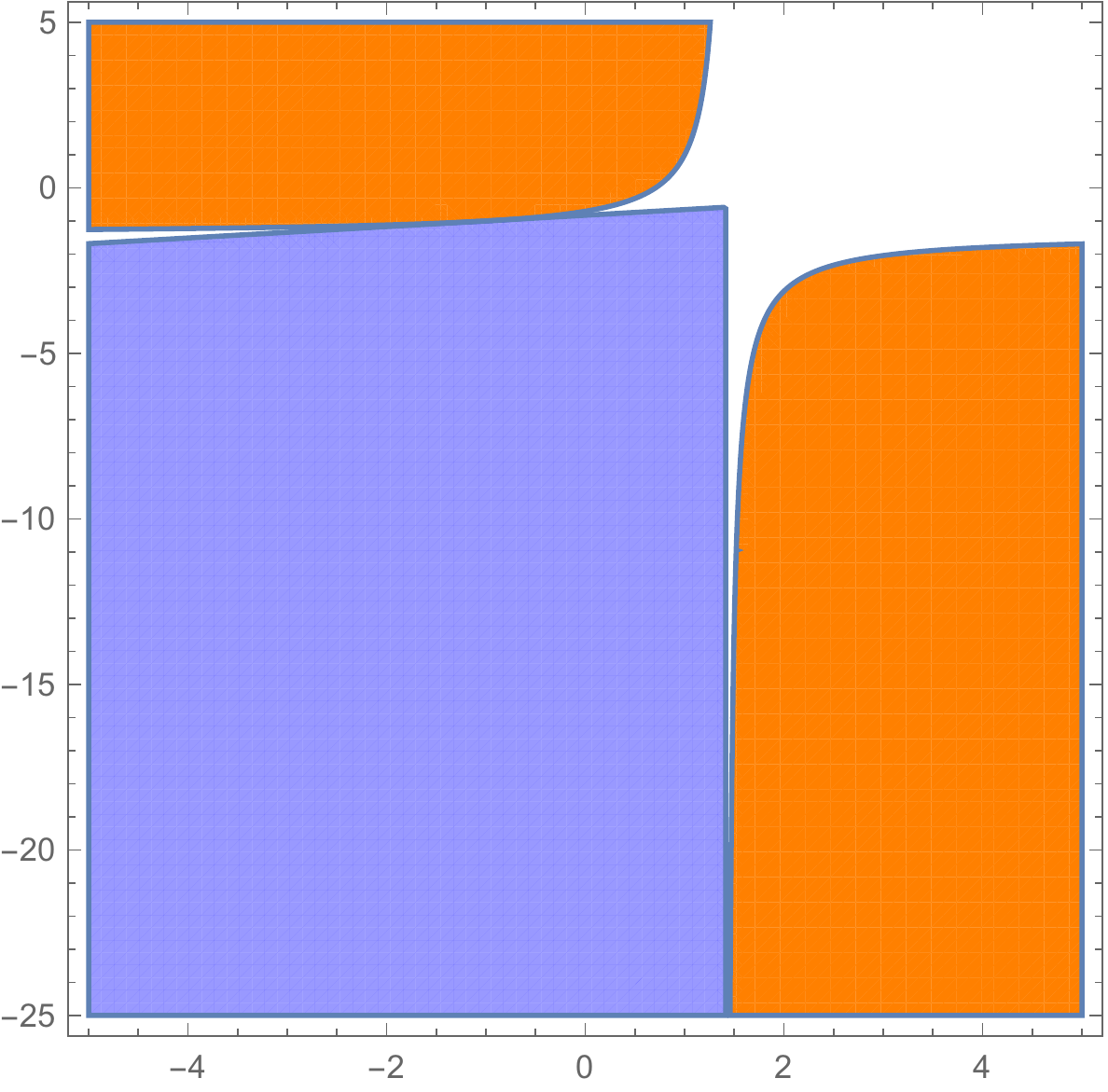}
    \caption{Projection onto $(x_1,x_2)$ of  $\shomorunning\cap H$ (orange) and $C_1 \cap H$ (blue).}
    \label{fig:sub-simpleexamplenonhomo-2d-maximal}
  \end{subfigure}
  \caption{Plots of $\shomorunning$, $H$ and $C_1$ as defined in \Cref{ex:simpleexnonhomo-maximal} showing that $C_1$ is maximal $\shomorunning$-free with respect to $H$.}
  \label{fig:simpleexamplenonhomo-maximal}
\end{figure}
\end{example}

As we see below, the characterization of $r$ as a limit is going to be useful to
prove maximality of $\maxnonhomob$.
However, to show that $\maxnonhomob$ is free, we need a different
interpretation of $r$.

\begin{lemma} \label{lemma:r_is_dual_sol}
  For every $\beta \in D_1(0)$, $r(\beta) = \theta(\beta)$, where
  $\theta(\beta)$ is defined in \eqref{eq:phi_dual_solution} and corresponds to the optimal dual solution of the optimization problem defining $\phifun(\beta)$.
\end{lemma}
\begin{proof}
  If $\lambda^\T a + d^\T \beta \leq 0$, $r(\beta) = 0 = \theta(\beta)$.
  Let $\beta \in D_1(0)$ be such that $\lambda^\T a + d^\T \beta > 0$.
  Then,
  \begin{align*}
    r(\beta)
    &= \frac{d^\T \beta + \lambda^\T a \phifun(\beta)}{\phifun(\beta) + d^\T
    \beta \lambda^\T a} \\
    &= \frac{d^\T \beta + \lambda^\T a \sqrt{1 - (\lambda^\T a)^2} \sqrt{1
    - (d^\T \beta)^2} - d^\T \beta (\lambda^\T a)^2}{\sqrt{1 - (\lambda^\T a)^2}
    \sqrt{1 - (d^\T \beta)^2}} \\
    &= \frac{d^\T \beta \sqrt{1 - (\lambda^\T a)^2}}{\sqrt{1 - (d^\T \beta)^2}}
    + \lambda^\T a \\
    &= \theta(\beta).
  \end{align*}
\end{proof}

\subsubsection{$\snonhomo$-freeness and maximality proofs}
We now show that $\maxnonhomob$ is $\snonhomo$-free and then that it is maximal.
\begin{theorem} \label{thm:free_nonhomo_bad}
  Let  $\lambda \in D_1(0)$ such that $\lambda \neq \pm a$, 
  \[
    \maxnonhomob = \{ (x,y) \st -\lambda^\T x + \nabla \phifun(\beta)^\T y \leq
    r(\beta), \text{ for all } \beta, \|\beta\| = 1 \}.
  \]
  and $\snonhomo = \{(x,y) \st \|x\| \leq \|y\|, a^\T x + d^\T y = -1 \}$, with $\|d\|  < \|a\| = 1$.
  Then, $\maxnonhomob$ is $\snonhomo$-free.
\end{theorem}
\begin{proof}
  Let $(x_0,y_0) \in \snonhomo$ and let $\beta_0 = \frac{y_0}{\|y_0\|}$.
  The claim will follow if we are able to show that
  $-\lambda^\T x_0 + \nabla \phifun(\beta_0)^\T y_0 \geq r(\beta_0)$.

  Since $x_0$ satisfies $\|x_0\| \leq \|y_0\|$ and $a^\T x_0 + d^\T y_0 = -1$, it
  follows that
  \[
    \lambda^\T x_0 \leq \max_x \{ \lambda^\T x \st \|x\| \leq \|y_0\|, a^\T x
    + d^\T y_0 \leq -1 \}.
  \]
  By weak duality we have
  \[
    \max_x \{ \lambda^\T x \st \|x\| \leq \|y_0\|, a^\T x + d^\T y_0 \leq -1 \}
    \leq \inf_{\theta \geq 0} \|y_0\| \|\lambda - a \theta\| - (d^\T y_0 + 1)\theta.
  \]
  Recall that $\theta(y_0)$ is the optimal dual solution to the optimization problem defining $\phifun(y_0)$. Thus, it holds that $\theta(y_0) \in \mathbb{R}_+$ and $\theta(y_0) < +\infty$
  because $\|d\| < 1$. Consequently,
  \[
    \inf_{\theta \geq 0} \|y_0\| \|\lambda - a \theta\| - (d^\T y_0 + 1)\theta
    \leq \|y_0\| \|\lambda - a \theta(y_0)\| - (d^\T y_0 + 1)\theta(y_0)
    = \phifun(y_0) - \theta(y_0),
  \]
  where the last equality follows from the strong duality between the optimization problem that defines $\phifun$ and its dual. See \Cref{prop:phi}.
  All the inequalities together show that
  \[
    \lambda^\T x_0 \leq \phifun(y_0) - \theta(y_0).
  \]
  From \eqref{eq:phi_dual_solution} and \Cref{lemma:r_is_dual_sol} it follow
  $\theta(y_0) = \theta(\beta_0) = r(\beta_0)$.
  Thus,
  \[
    -\lambda^\T x_0 + \phifun(y_0) \geq r(\beta_0),
  \]
  as we wanted to establish.
\end{proof}
\begin{theorem} \label{thm:max_nonhomo_bad}
  Let $\lambda \in D_1(0)$ such that $\lambda \neq \pm a$, 
  \[
    \HH = \{ (x,y) \in \mathbb{R}^{n+m} \st a^\T x + d^\T y = -1 \},
  \]
  \[
    \shomob = \{ (x,y) \mathbb{R}^{n+m} \st \|x\| \leq \|y\|,\ a^\T x + d^\T
    y \leq 0 \},
  \]
  and
  \[
    \maxnonhomob = \{ (x,y) \st -\lambda^\T x + \nabla \phifun(\beta)^\T y \leq
    r(\beta), \text{ for all } \beta \in D_1(0) \}.
  \]
  where $\|d\| < \|a\| = 1$. Then, $\maxnonhomob$ is maximal $\shomob$-free with respect to $\HH$. 
  
  Additionally, if $\lambda = \frac{\xlp}{\|\xlp\|}$ with $(\bar{x}, \bar{y}) \in H$ and $\|\bar x\| > \| \bar{y} \|$, then $(\xlp,\ylp) \in \inte(\maxnonhomob)$.
\end{theorem}
\begin{proof}
  Let $\snonhomo = \shomob \cap \HH$.
  By \Cref{thm:free_nonhomo_bad}, $\maxnonhomob$ is $\snonhomo$-free.

  To show maximality we will use Theorem~\ref{thm:mixed_exposed_maximality_wrt},
  that is, we will show that every inequality of $\maxnonhomob$ is either
  exposed by a point in $\snonhomo \cap \maxnonhomob$ or exposed at
  infinity by a sequence in $\snonhomo$.

  Let $\beta_0 \in D_1(0)$ and consider the valid inequality $-\lambda^\T
  x + \nabla \phifun(\beta_0)^\T y \leq r(\beta_0)$.
  Assume, first, that $a^\T \lambda + d^\T\beta_0 < 0$
  As $a^\T \lambda + d^\T\beta_0 < 0$, we have that $r(\beta_0) = 0$,
  $\phifun(\beta_0) = \|\beta_0\| = 1$, and $\nabla \phifun(\beta_0) = \beta_0$.
  Hence, the inequality is $-\lambda^\T x + \beta_0^\T y \leq 0$.
  It is exposed by
  \[
    \frac{-1}{a^\T \lambda + d^\T\beta_0}(\lambda, \beta_0) \in \snonhomo
    \cap \maxhomobb \subseteq \snonhomo \cap \maxnonhomob.
  \]

  Now, let us assume that $a^\T \lambda + d^\T\beta_0 \geq 0$.
  We will show that there is a sequence in $\snonhomo$ that exposes
  $-\lambda^\T x + \nabla \phifun(\beta_0)^\T y \leq r(\beta_0)$ at infinity.
  Let $(x_n)_n \subseteq \langle \lambda, a \rangle$ be a sequence converging to
  $\xbetaO$ such that $\|x_n\| = 1$, $a^\T x_n + d^\T \beta_0 < 0$ (\Cref{lemma:characterizationrbeta}).
  \[
    r(\beta_0) = \lim_{n \to \infty} \frac{\lambda^\T (x_n - \xbetaO)}{a^\T (x_n
    - \xbetaO)}.
  \]
  Consider the sequence conformed by
  \[
    z_n = -\frac{(x_n, \beta_0)}{a^\T x_n + d^\T \beta_0} = \frac{(x_n,
    \beta_0)}{a^\T (\xbetaO - x_n)} \in \snonhomo,
  \]
  where the equality above follows from $a^\T \xbetaO + d^\T \beta_0 = 0$.
  We proceed to verify that $z_n$ exposes $-\lambda^\T x + \nabla
  \phifun(\beta_0)^\T y \leq r(\beta_0)$ at infinity.

  As $x_n \to \xbetaO$, we have that $\|z_n\| \to \infty$.
  Also, $\frac{z_n}{\|z_n\|} = \frac{1}{\sqrt{2}} (x_n, \beta_0)$ converges to
  $v = \frac{1}{\sqrt{2}} (\xbetaO, \beta_0) \in \maxhomobb = \rec (\maxnonhomob)$
  and exposes $-\lambda^\T x + \nabla \phifun(\beta_0)^\T y \leq 0$.

  Finally, we have to show that there exists a $w$ such that $(-\lambda, \nabla
  \phifun(\beta_0))^\T w = r(\beta_0)$ and $\dist(z_n, w + \langle v \rangle)
  \to 0$.
  Let $\hat x = \lim_{n \to \infty} \frac{x_n - \xbetaO}{\|x_n - \xbetaO\|}$ and
  let $w = (-\frac{\hat x}{a^\T \hat x}, 0)$.
  We have that $(-\lambda, \nabla \phifun(\beta_0))^\T w = r(\beta_0)$.
  Also,
  \[
    z_n - \frac{\sqrt{2}}{a^\T (\xbetaO - x_n)} v = \frac{1}{a^\T (\xbetaO
    - x_n)}(x_n - \xbetaO, 0) \to - (\frac{\hat x}{a^\T \hat x}, 0) = w.
  \]
  Thus, $\dist(z_n, w + \langle v \rangle) \to 0$.
\end{proof}

\subsubsection{A closed-form formula for $\maxnonhomob$}
Since the construction of $\maxnonhomob$ involves translating some of the
inequalities of $\maxhomobb$ of its outer-description, it is natural to ask if
this translation yields a translation of the whole function $\phifun$. This
would yield a closed-form formula for $\maxnonhomob$ which is much more appealing from a computational standpoint.

In what follows, we ask whether there exists an $(x_0, y_0)$ such that for every $\beta$ such that
\begin{align*}
&\{(x,y) \st -\lambda^\T x + \nabla \phifun(\beta)^\T y \leq r(\beta), \text{
for all } \beta, \lambda^\T a + d^\T \beta  \geq 0 \}\\
=\, & \{(x,y) \st -\lambda^\T (x - x_0) + \nabla \phifun(\beta)^\T (y - y_0) \leq 0,  \text{ for all } \beta, \lambda^\T a + d^\T \beta  \geq 0 \}.
\end{align*}
In order to reach this equality it would suffice to satisfy 
\begin{equation}\label{eq:equationtosatisfy}
\lambda^\T x_0 - \nabla \phifun(\beta)^\T y_0 = -r(\beta).
\end{equation}
Note that since $\lambda^\T a + d^\T \beta  \geq 0$
\begin{align}
  \nabla \phifun(\beta) &= \sqrt{1 - (\lambda^T a)^2}
  \frac{W\beta}{\|\beta\|_W} - \lambda^\T a d \label{eq:gradientofbeta} \\
  r(\beta) &= \lambda^T a + d^\T \beta \frac{\sqrt{1 - (\lambda^T
  a)^2}}{\|\beta\|_W}. \nonumber
\end{align}
where $W = I - d d^\T$. Thus \eqref{eq:equationtosatisfy} becomes
\[
 \lambda^\T (x_0 + a d^\T y_0) - \sqrt{1 - (\lambda^T a)^2} \frac{\beta^\T W y_0}{\|\beta\|_W} = -\lambda^T a - d^\T \beta \frac{\sqrt{1 - (\lambda^T
a)^2}}{\|\beta\|_W}.
\]
From the last expression, we see that if we are able to find $(x_0, y_0)$ such that
\begin{subequations} 
\begin{align}
x_0 + a d^\T y_0 =& -a \label{eq:tosatisfy1}\\ 
d^\T \beta =& \beta^\T W y_0 \label{eq:tosatisfy2}
\end{align}
\end{subequations}
then \eqref{eq:equationtosatisfy} would hold. Note that $d$ is an eigenvector of $W = I - d d^\T$  with eigenvalue $1 - \|d\|^2$.
Thus, with $y_0 = \frac{d}{1 - \|d\|^2}$ we can easily check that \eqref{eq:tosatisfy2} holds.
With $y_0$ defined, in order to satisfy \eqref{eq:tosatisfy1} it suffices to set
\[x_0 = -a (d^\T y_0 +1) = -\frac{a}{1 - \|d\|^2}.\]

In summary, we arrive to the following expression for $\maxnonhomob$,
\begin{equation}\label{eq:representationofC}
  \maxnonhomob = \left\{ (x,y) \st
    \begin{split}
      \phifun(y) \leq \lambda^\T x &&\text{ if } \lambda^\T a \|y\| + d^\T y \leq 0\\
      \phifun\left(y - \frac{d}{1 - \|d\|^2} \right) \leq \lambda^\T \left(x + \frac{a}{1
    - \|d\|^2}\right)  &&\text{otherwise}
    \end{split}
  \right\}.
\end{equation}

\section{On the diagonalization and homogenization of quadratics}
\label{sec:transformations}
Consider an arbitrary quadratic set
\[
  \cQ = \{ s \in \mathbb{R}^p \st s^\T Q s + b^\T s + c \leq 0\}.
\]
Given a point $\slp \notin \cQ$ we can construct a maximal $\cQ$-free set that
contains $\slp$ using the techniques developed in the previous sections.
The idea to do this is first to find a one-to-one map $T$ such that
\begin{align*}
  T(\cQ) &= \shomob \cap \HH = \{ (x,y,z) \in \mathbb{R}^{n+m+l} \st \|x\|
  \leq \|y\|,\, a^\T x + d^\T y  + h^\T z = -1 \}\\
  T(\slp) &\in \HH \setminus \shomob,
\end{align*}
for some hyperplane $H$, that is, for some $a, d$ and $h$.

Then, we construct a maximal $\cQ$-free set using the following fact which can be easily verified: if $C$ is a maximal $\shomob$-free set with respect to $\HH$ that contains
$T(\slp)$, then $T^{-1}(C)$ is a maximal $\cQ$-free set containing $\slp$.

Here we show a surprising fact: \emph{which} maximal $\cQ$-free set is obtained
heavily depends on the choice of $T$.
We illustrate this interesting feature with our running example.

\begin{example} \label{ex:transformations}
  Let
  \[
    \cQ = \{ s \in \mathbb{R}^2 \st -2 + 2 \sqrt{2}s_1 - 2 \sqrt{2} s_2 + 2 s_1
    s_2 \leq 0\}
  \]
  and $\slp = (-2,-2) \not\in \cQ$.
  The following map
  \[
    \tau_1(s_1,s_2) = (s_1, s_2, \sqrt{2} + s_1 - s_2)
  \]
  is one-to-one and satisfies
  \[
    \tau_1(\cQ) = \shomorunning\cap H_1,
  \]
  where $\shomorunning\cap H_1$ is defined in \Cref{ex:simpleexnonhomo}
  and is given by
  \[
    \shomorunning \cap H_1  = \{ (x_1,x_2, y) \in \mathbb{R}^3 \st  \|x_1,x_2 \|
    \leq |y|,\, -x_1 + x_2 + y = -\sqrt{2} \}.
  \]
  Computing a maximal $\shomorunning$-free set with respect to $H_1$ containing
  $\tau_1(\bar{s}) = (-2,-2,\sqrt{2})$ yields the same maximal $S^2\cap H_1$-free set
  we compute in \Cref{ex:simpleexnonhomo-maximal}, that is
  \begin{align*}
    C_1 \cap H_1 = \{ (x,y) \st
    & \frac{1}{\sqrt{2}}(x_1 + x_2) - y \leq 0, \\
    & \frac{1}{\sqrt{2}}(x_1 + x_2) + \frac{1}{\sqrt{2}}y \leq 1 \\
    & -x_1 + x_2 + y = -\sqrt{2} \}.
  \end{align*}
  As $\tau_1^{-1}$ is simply the projection onto the first two coordinates, we have
  that
  \[
    \tau^{-1}_1(C_1) = \left\{
      s \in \mathbb{R}^2 \st \left(\frac{1}{\sqrt{2}}
      - 1 \right) s_1 + \left(\frac{1}{\sqrt{2}} + 1\right) s_2 +\sqrt{2} \leq
      0,\,  \sqrt{2} s_1 -2 \leq 0
    \right\}
  \]
  is our maximal $\cQ$-free set. This is exactly the set we show in
  \Cref{fig:sub-simpleexamplenonhomo-2d-maximal}.

  Now we consider a different transformation for $\cQ$.
  Let
  \begin{align*}
    T_1(s_1,s_2) &= \frac{1}{2}
    \left[
      \begin{array}{cc}
        -1 & 1 \\
        1 & 1
      \end{array}
    \right]
    \left[
      \begin{array}{cc}
        s_1 - \sqrt{2} \\
        s_2 + \sqrt{2}
      \end{array}
    \right], \\
    T_2(s_1,s_2) &= (-1, s_1, s_2), \text{ and}\\
    \tau_2 &= T_2 \circ T_1.
  \end{align*}
  For the curious reader, $T_1$ is obtained from an
  eigen-decomposition of the quadratic form.
  After some algebraic manipulation we can see that
  \begin{align*}
    T_1(\cQ) &=  \{w \in \mathbb{R}^2 \st T_1^{-1}(w_1,w_2) \in \cQ\} \\
             &= \{w \in \mathbb{R}^2 \st 1 - w_1^2 + w_2^2 \leq 0 \}.
  \end{align*}
  Thus, $\tau_2$ is one-to-one and
  \[
    \tau_2(\cQ) = \{ (x_1,x_2,y) \mathbb{R}^2 \st \|x_1,x_2\| \leq |y|,\, x_1 = -1 \}.
  \]
  Letting $S^3_{\leq 0} = \{ (x_1,x_2,y) \mathbb{R}^3 \st \|x_1,x_2\| \leq
  |y|,\, x_1 \leq 0 \}$ and $H_2 = \{ (x_1,x_2,y) \mathbb{R}^3 \st \, x_1 = -1
  \}$, we have that $\tau_2(\cQ) = S^3_{\leq 0} \cap H_2$.
  We can now construct a maximal $S^3_{\leq 0}$-free set with respect to $H_2$.
  For this, note that in this case $a = (1,0)$ and $d = 0$.
  Also, $\tau_2(\slp) = (-1, -2, \sqrt{2})$ and so $\lambda
  =\frac{1}{\sqrt{5}}(-1,-2)$.
  As $a^\T \lambda |y| + d y < 0$ for every $y \in \mathbb{R}$, we
  have that $r(y) = 0$ and $\phifun(y) = |y|$.
  By Theorem~\ref{thm:max_nonhomo_bad},
  \[
    C_2 = \{(x_1, x_2, y) \in \mathbb{R}^3 \st |y| \leq \lambda^\T x\}
  \]
  is maximal $S^3_{\leq 0}$-free set with respect to $H_2$.
  Therefore, $\tau_2^{-1}(C_2)$ is maximal $\cQ$-free.
  In \Cref{fig:transformations} we show the sets $\cQ$ and both maximal
  $\cQ$-free sets given by $\tau_1^{-1}(C_1)$ and $\tau_2^{-1}(C_2)$.
  Note that in this case, the set $\tau_2^{-1}(C_2)$ does not have an asymptote,
  and both its facets have an exposing point.
\end{example}

\begin{figure}[t]
  \centering
  \includegraphics[scale=0.4]{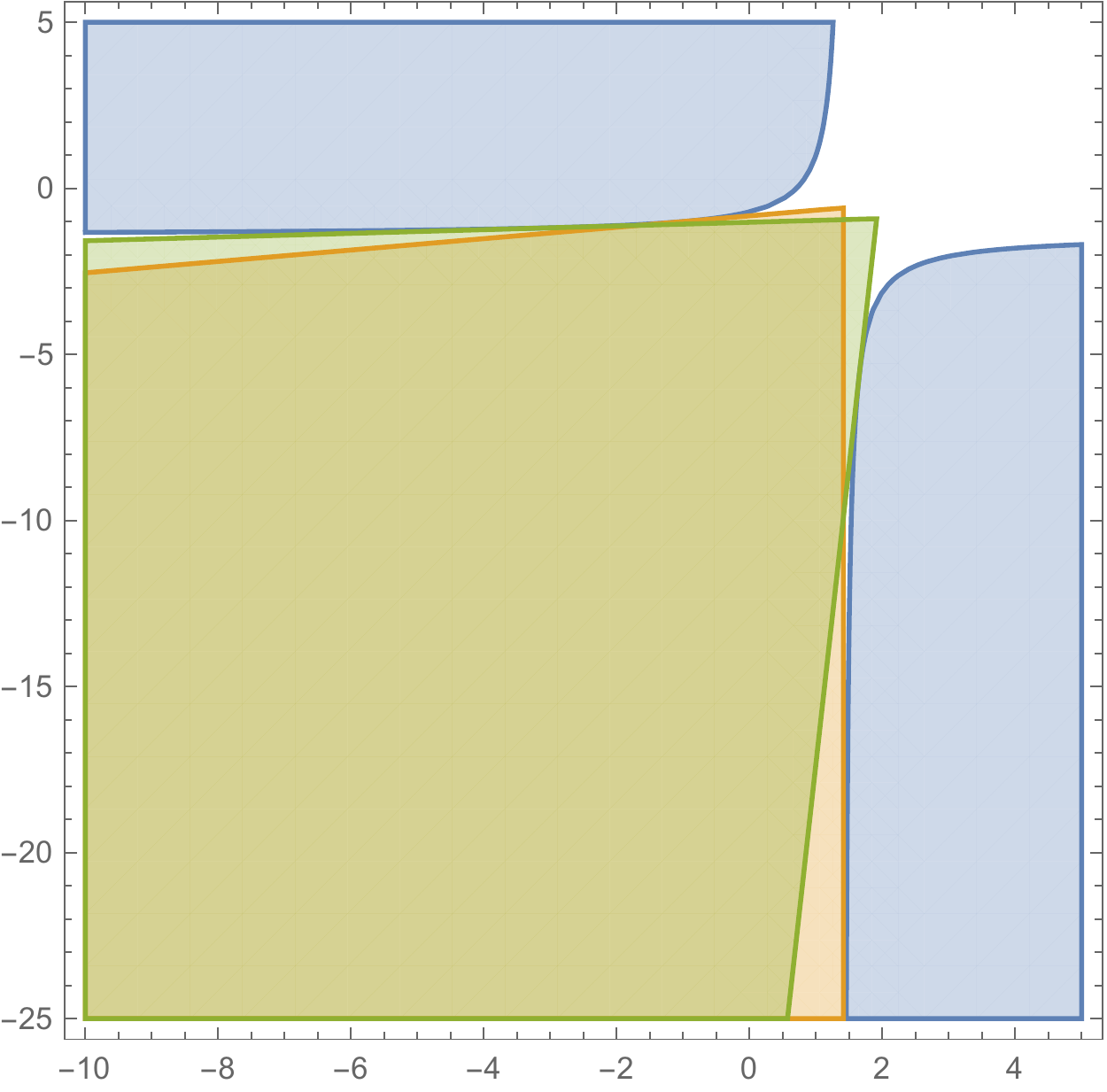}
  \caption{
    Different maximal $S$-free sets obtained from different
    transformations, as discussed in \Cref{ex:transformations}.
    The quadratic set $\cQ$ (blue), a maximal $\cQ$-free set obtained from $\tau_1$
    (orange), and another such set obtained from $\tau_2$ (green).
  }
  \label{fig:transformations}
\end{figure}

This example shows the important role of 
the transformation used to bring the quadratic set to 
the form $\snonhomo$. The resulting maximal $\snonhomo$-free set can significantly change.
This opens an array of interesting questions regarding the role of
transformations in our approach: Can we distinguish the transformations that
generate $\snonhomo$-free sets with asymptotes? Is there a benefit/downside from
using the latter sets? These an other questions are left for future work.

\section{Summary and future work} \label{sec:conclusions}

In this work we have shown how to construct \emph{maximal} quadratic-free sets, i.e., convex sets whose interior does not intersect the sublevel set of a quadratic function. Using the long-studied \emph{intersection cut} framework, these sets can be used in order to generate deep cutting planes for quadratically constrained problems. We strongly believe that, by carefully laying a theoretical framework for quadratic-free sets, this work provides an important contribution to the understanding and future computational development of non-convex quadratically constrained optimization problems.

The maximal quadratic-free sets we construct in this work allow for an \emph{efficient} computation of the corresponding intersection cuts. Computing such cutting planes amount to solving a simple one-dimensional convex optimization problem using the quadratic-free sets we show here. Moreover, even if in our constructions and maximality proofs we use semi-infinite outer-descriptions of $\snonhomo$-free sets such as \eqref{eq:maxnonhomob}, all
of them have closed-form expressions that are more adequate for
computational purposes: see \eqref{eq:clambdadef},
\eqref{eq:closedformcglambda}, \eqref{eq:cphilambdadef},
\eqref{eq:representationofC} for these expressions for the sets $\maxhomo$, $\maxhomobg$, $\maxhomobb$ and $\maxnonhomob$, respectively, and \eqref{eq:phi_characterization} for the explicit description of the $\phifun$ function. This ensures efficient separation in LP-based methods for quadratically constrained optimization problems.

The empirical performance of these intersection cuts remains to be seen, and it is part of future work. On these lines, the development of a cut strengthening procedure is likely to be important for obtaining good empirical performance.  Other important open questions involve the better understanding of the role different transformations of quadratic inequalities have (\Cref{sec:transformations}), a theoretical and empirical comparison with the method proposed by Bienstock et al. \cite{bienstock2016outer,Bienstock_2019}, and devising new methods for producing other families of quadratic-free sets. On this last point, we conjecture our method produces all maximal quadratic-free sets in 3 dimensions, and have evidence showing this is not the case for dimension 4. This is also subject of ongoing work. 

\section*{Acknowledgements}
We would like to thank Franziska Schl\"osser for several inspiring conversations. We would also like to thank Stefan Vigerske, Antonia Chmiela, Ksenia Bestuzheva and Nils-Christian Kempke for helpful discussions. Lastly, we would like to acknowledge the support of the IVADO Institute for Data Valorization for their support through the IVADO Post-Doctoral Fellowship program and to the IVADO-ZIB academic partnership.

\newpage
\bibliographystyle{alpha}
\bibliography{biblio}

\newpage
\appendix
\section{Appendix}
\counterwithin{lemma}{section}
\counterwithin{proposition}{section}

\begin{lemma} \label{lemma:valid_ineqs}
  Let $\phi : \RR^m \to \RR$ be a sublinear function, $\lambda \in D_1(0)$, and
  let
  \[
    C = \{ (x,y) \st \phi(y) \leq \lambda^\T x \}.
  \]
  Let $\alpha^\T x + \gamma^\T y \leq \delta$ be a nontrivial valid inequality
  for $C$.
  Then $\alpha \neq 0$ and if $\|\alpha\| = 1$, then
  \[
    \alpha = -\lambda \text{ and } \gamma \in \partial \phi(0).
  \]
\end{lemma}
\begin{proof}
  As $\phi$ is defined over $\RR^m$, if $\alpha = 0$, then $\gamma = 0$.
  Given that the inequality is nontrivial, it must be that $\alpha \neq 0$.

  Assume that $\alpha \in D_1(0)$.
  Since $\{ (x,0) \st \lambda^\T x = 0 \}$ is contained in the lineality space
  of $C$, it follows that $\alpha$ is either $\lambda$ or $-\lambda$.
  As $\mu (\lambda, 0) \in C$ for every $\mu \geq 0$, we conclude that $\alpha
  = -\lambda$.

  Let $\mu > 0$ and $(x,y)$ be such that $\lambda^\T x = \phi(y)$.
  From the sublinearity of $\phi$ follows that $\mu(x,y) \in C$.
  The inequality evaluated at $\mu(x,y)$ reads $\gamma^\T \mu y \leq \delta
  + \mu \phi(y)$.
  Dividing by $\mu$ and letting $\mu \to \infty$, we conclude that
  $\gamma^\T y \leq \phi(y)$, i.e., $\gamma \in \partial \phi(0)$.
\end{proof}

\exposingIneq
  
\begin{proof}
  We need to verify both conditions of \Cref{def:expose}.
  As $\phi$ is positively homogeneous and differentiable at $\ylp$, then
  $\phi(\ylp) = \nabla \phi(\ylp) \ylp$.
  Thus, evaluating $-\lambda^\T x + \nabla \phi(\ylp)^\T y$ at
  $(\xlp, \ylp)$ yields $-\lambda^\T \xlp + \phi(\ylp)$, which is 0 by
  hypothesis.
  This shows that the inequality is tight at $(\xlp, \ylp)$.

  Now, let $\alpha^\T x + \gamma^\T y \leq \delta$ be a non-trivial valid
  inequality tight at $(\xlp, \ylp)$.
  Then, $\delta = \alpha^\T \xlp + \gamma^\T \ylp$ and we can rewrite the
  inequality as $\alpha^\T (x - \xlp) + \gamma^\T (y - \ylp) \leq 0$.
  Notice that $(\phi(y) \lambda, y) \in C$, thus,
  $\alpha^\T \lambda (\phi(y) - \phi(\ylp)) + \gamma^\T (y - \ylp) \leq 0$ for
  every $y \in \mathbb{R}^m$.
  Subtracting $\alpha^\T \lambda \nabla \phi(\ylp)^\T (y - \ylp)$ and dividing
  by $\|y - \ylp\|$ we obtain the equivalent expression
  \[
    \alpha^\T \lambda \frac{\phi(y) - \phi(\ylp) - \nabla \phi(\ylp)^\T (y
    - \ylp)}{\|y - \ylp\|} \leq  (-\gamma -\alpha^\T \lambda \nabla
    \phi(\ylp))^\T \frac{y - \ylp}{\|y - \ylp\|}.
  \]
  Since $\phi$ is differentiable at $\ylp$, the limit when $y$ approaches $\ylp$
  of the left hand side of the above expression is 0.
  However, one can make the expression $\frac{y - \ylp}{\|y - \ylp\|}$ converge
  to any point of $D_1(0)$.
  Therefore,
  \[
    0 \leq  (-\gamma -\alpha^\T \lambda \nabla \phi(\ylp))^\T \beta
  \]
  for every $\beta \in D_1(0)$.
  This implies that $\gamma = -\alpha^\T \lambda \nabla \phi(\ylp)$.
  From here we see that $\alpha \neq 0$ as otherwise $\alpha = \gamma = 0$ and
  the inequality would be trivial.

  Given that any $(x,0)$ such that $\lambda^T x = 0$ belongs to $C$, it follows
  that $\alpha$ is parallel to $\lambda$, i.e., there exists $\nu \in
  \mathbb{R}$ such that $\alpha = \nu \lambda$.
  Furthermore, $(\mu \lambda, 0) \in C$ for every $\mu \geq 0$, implies that
  $0 > \alpha^\T \lambda = \nu$.
  Therefore, $\gamma = - \nu \nabla \phi(\ylp)$ and the inequality reads
  $\nu \lambda^\T(x - \xlp) - \nu \nabla \phi(\ylp)^\T(y - \ylp) \leq 0$.
  Dividing by $|\nu|$ and using that $-\lambda^\T x + \nabla \phi(\ylp)^\T
  y \leq 0$ is tight at $(\xlp, \ylp)$, we conclude that the inequality can be
  written as
  \[
    -\lambda^\T x + \nabla \phi(\ylp)^\T y \leq 0.
  \]
\end{proof}

\begin{proposition} \label{prop:phi}
  Let $a,\lambda \in D_1(0)$, $\lambda \neq \pm a$ and let $d \in \mathbb{R}^m$
  be such that $\|d\| \leq 1$.
  The (Lagrangian) dual problem of
  \begin{equation} \label{eq:phi_primal}
    \max_x \{ \lambda^\T x \st \|x\| \leq \|y\|, a^\T x + d^\T
    y \leq 0 \}
  \end{equation}
  is
  \begin{equation} \label{eq:phi_dual}
    \inf_\theta \{ \|\lambda - \theta a\| \|y\| - \theta d^\T y \st \theta \geq
    0 \}.
  \end{equation}
  The optimal solution to \eqref{eq:phi_primal} is $x : \mathbb{R}^m \to
  \mathbb{R}^n$,
  \begin{equation} \label{eq:phi_primal_solution}
    x(y) =
    \begin{cases}
      \lambda \|y\|,
    &\text{ if } \lambda^\T a \|y\| + d^\T y \leq 0 \\
    \sqrt{\frac{\|y\|^2 - (d^\T y)^2}{1 - (\lambda^T a)^2}} \lambda
    -\left( d^\T y + \lambda^T a \sqrt{\frac{\|y\|^2 - (d^\T y)^2}{1
    - (\lambda^T a)^2}} \right) a,
    &\text{ otherwise. }
    \end{cases}
  \end{equation}
  The optimal dual solution is $\theta : \mathbb{R}^m \to \mathbb{R}_+ \cup
  \{+\infty\}$,
  \begin{equation} \label{eq:phi_dual_solution}
    \theta(y) =
    \begin{cases}
      0,
    &\text{ if } \lambda^\T a \|y\| + d^\T y \leq 0 \\
    \lambda^T a + d^\T y \frac{\sqrt{1 - (\lambda^T
    a)^2}}{\sqrt{\|y\|^2 - (d^\T y)^2}},
    &\text{ otherwise. }
    \end{cases}
  \end{equation}
  Here, $\frac{1}{0} = +\infty$ and $r + (+\infty) = +\infty$ for every $r \in
  \mathbb{R}$.
  Moreover, strong duality holds, that is, $\eqref{eq:phi_primal} =
  \eqref{eq:phi_dual}$, and
  \begin{equation} \label{eq:phi_value}
    \eqref{eq:phi_primal} =
    \begin{cases}
      \|y\|,
    &\text{ if } \lambda^\T a \|y\| + d^\T y \leq 0 \\
    \sqrt{(\|y\|^2 - (d^\T y)^2)(1 - (\lambda^T a)^2)} - d^\T y \lambda^\T a,
    &\text{ otherwise. }
    \end{cases}
  \end{equation}
  Finally, \eqref{eq:phi_value} holds even if $\lambda = \pm a$.
\end{proposition}
\begin{proof}
  First, note that since $\lambda \neq \pm a$ and $\|d\| \leq 1$, $x(y)$ and
  $\theta(y)$ are defined for every $y \in \mathbb{R}^m$.
  Second, to make some of the calculations that follow more amenable, let $S(y)
  = \sqrt{\frac{\|y\|^2 - (d^\T y)^2}{1 - (\lambda^T a)^2}}$.

  The Lagrangian of \eqref{eq:phi_primal} is $L : \mathbb{R}^n \times
  \mathbb{R}_+^2 \to \mathbb{R}$,
  \[
    L(x,\mu,\theta) = \lambda^\T x - \mu(\|x\| - \|y\|) - \theta(a^\T x + d^\T
    y).
  \]
  Thus, the dual function is
  \[
    d(\mu, \theta) = \max_x L(x,\mu,\theta).
  \]
  We have that $d(\mu, \theta)$ is infinity whenever $\mu < \|\lambda
  - a \theta\|$, and $\mu \|y\| - \theta d^\T y$ otherwise.
  Hence, the dual problem, $\min_{\theta, \mu \geq 0} d(\mu, \theta)$, is $\min
  \{ \mu \|y\| - \theta d^\T y \theta \st \theta \geq 0, \mu \geq \|\lambda
  - a \theta\|\}$ which is \eqref{eq:phi_dual}.

  Let us assume that $\lambda^\T a \|y\| + d^\T y \leq 0$.
  Clearly, $x(y) = \lambda \|y\|$ is feasible for \eqref{eq:phi_primal}.
  Its objective value is $\|y\|$.
  On the other hand, $\theta(y) = 0$ is always feasible for
  \eqref{eq:phi_primal}.
  Its objective value is also $\|y\|$, therefore, $x(y)$ is the primal
  optimal solution and $\theta(y)$ the dual optimal solution.

  Now let us consider the case $\lambda^\T a \|y\| + d^\T y > 0$.
  Let us check that $\theta(y)$ is dual feasible, that is, $\theta(y) \geq 0$.
  Note that, due to the positive homogeneity of $\theta(y)$ and the condition
  $\lambda^\T a \|y\| + d^\T y > 0$ with respect to $y$, we can assume without
  loss of generality that $\|y\| = 1$.

  Let $\alpha = \lambda^\T a$ and $\beta = d^\T y$.
  Since $\theta(d) = +\infty \geq 0$ when $\|d\| =1$, we can assume that $y \neq
  d$ when $\|d\|=1$.
  Note that the same does not occur when $y = -d$ since we are assuming
  $\lambda^\T a \|y\| + d^\T y > 0$.
  Thus, $\alpha, \beta \in (-1,1)$.

  We will prove that $\theta(y) \sqrt{1 - \beta^2} = \alpha \sqrt{1 - \beta^2}
  + \beta \sqrt{1 - \alpha^2} \geq 0$, which implies that $\theta(y) \geq 0$.
  If $\alpha, \beta \geq 0$, then we are done.
  As $\alpha + \beta > 0$, at least one of them must be positive.
  Let us assume $\alpha > 0$ and $\beta < 0$, the other case is analogous.
  Then, $\alpha > - \beta \geq 0$.
  This implies that $\alpha^2 > \beta^2$.
  Subtracting $\alpha^2 \beta^2$, factorizing and taking square roots we obtain
  the desired inequality.

  Let us compute the value of the dual solution $\theta(y)$.
  First, $y = d$ and $\|d\| = 1$, $\theta(y) = +\infty$, which means that the
  optimal value is
  \[
    \lim_{\theta \to +\infty} \|\lambda - \theta a\| - \theta = - \lambda^\T a.
  \]
  One way of computing this limit is to multiply and divide the expression by
  $\tfrac{\|\lambda - \theta a\| + \theta}{\theta}$, expand, and simplify the
  numerator and denominator until one obtains something simple enough.

  Now assume $y \neq d$ if $\|d\| = 1$.
  Observe that $\|\lambda - \theta(y) a\| \|y\| - \theta(y) d^\T y
  =\sqrt{\|\lambda - \theta(y) a\|^2} \|y\| - \theta(y) d^\T y$.
  We have that
  \begin{align*}
    \|\lambda - \theta(y) a\|^2
    &= 1 + \theta(y)(\theta(y) - 2\lambda^\T a) \\
    &= 1 + (\theta(y) - \lambda^\T a + \lambda^\T a)(\theta(y) - \lambda^\T
    a - \lambda^\T a) \\
    &= 1 + (\theta(y) - \lambda^\T a)^2 - (\lambda^\T a)^2.
  \end{align*}
  Replacing $\theta(y)$, we obtain
  \begin{align*}
    \|\lambda - \theta(y) a\|^2
    &= 1 + \frac{(d^\T y)^2}{S(y)} - (\lambda^\T a)^2 \\
    &= \frac{1}{S(y)}(S^2(y) (1 - (\lambda^\T a)^2) + (d^\T y)^2) \\
    &= \frac{\|y\|^2}{S^2(y)}.
  \end{align*}
  Therefore,
  \begin{align*}
    \|\lambda - \theta(y) a\| \|y\| - \theta(y) d^\T y
    &= \frac{\|y\|^2}{S(y)} - d^\T y \lambda^\T a - \frac{(d^\T y)^2}{S(y)}\\
    &= \frac{\|y\|^2 - (d^\T y)^2}{S(y)} - d^\T y \lambda^\T a \\
    &= \sqrt{(\|y\|^2 - (d^\T y)^2)(1 - (\lambda^T a)^2)} - d^\T y \lambda^\T a.
  \end{align*}

  Let us now check the feasibility of $x(y)$.
  Let us first check that $\|x(y)\|^2 \leq \|y\|^2$.
  We have $\|x(y)\|^2 = S^2(y) - 2 S(y)(d^\T y + S(y) \lambda^\T a) \lambda^\T
  a + (d^\T y + \lambda^\T a S(y))^2$.
  Expanding and removing common terms yields $\|x(y)\|^2 = S^2(y)(1 - (\lambda^\T
  a)^2) + (d^\T y)^2 = \|y\|^2$.
  Thus, the first constraint is satisfied.\\
  To check the second constraint just notice that, as $\|a\| = 1$, $a^\T x(y)
  = -d^\T y$.

  The primal value of $x(y)$ is
  \[
    \lambda^\T x(y) = S(y)(1 - (\lambda^\T a)^2) - d^\T y \lambda^\T a  =
    \sqrt{(\|y\|^2 - (d^\T y)^2)(1 - (\lambda^T a)^2)} - d^\T y \lambda^\T a.
  \]
  As it coincides with the value of the dual solution, even when $y = d$ and
  $\|d\| =1$, we conclude that both are optimal.

  It only remains to check \eqref{eq:phi_value} for $\lambda = \pm a$.
  If $\lambda = -a$, then the linear constraint becomes $\lambda^\T x \geq d^\T
  y$ and the optimal solution is $x = \lambda\|y\|$.
  If $\lambda = a$, then the linear constraint becomes $\lambda^\T x \leq -d^\T
  y$ and $x = -d^\T y \lambda $ is then optimal. 
  In both cases \eqref{eq:phi_value} holds.
\end{proof}

\begin{lemma}
  \label{lemma:donttouch}
  Consider the set 
  \[ \snonhomo  = \{ (x,y) \mathbb{R}^{n+m} \st \|x\| \leq \|y\|,\ a^\T x + d^\T y = -1  \} \] 
  with  $a, d$ such that $\|a\| > \|d\|$. Let $\lambda, \beta\in D_1(0)$ be two vectors satisfying $\lambda^\T a + d^\T \beta \geq 0$ and consider $\maxhomobb$ defined in \eqref{eq:cphilambdadef}.
  
  Then, the face of $\maxhomobb$ defined by the valid inequality $-\lambda^\T x + \nabla \phifun(\beta)^\T y \leq 0$ does not intersect $\snonhomo$.
\end{lemma}
\begin{proof}
By contradiction, suppose that $(\xlp, \ylp) \in \maxhomobb$ is such that
\[(\xlp, \ylp) \in S \quad \wedge \quad -\lambda^\T \xlp + \nabla \phifun(\beta)^\T \ylp = 0. \]
The latter equality and the fact that $\phifun$ is sublinear implies $\phifun(\ylp) = \lambda^\T \xlp$. Moreover, $\xlp$ is a feasible solution of the optimization problem $\phifun(\ylp)$, which implies it is an optimal solution.

By Lemma~\ref{lemma:exposing_ineq} we know $(\xlp, \ylp)$ exposes the valid inequality of $\maxhomobb$ given by $-\lambda^\T x + \nabla \phifun(\ylp)^\T y \leq 0$.
By definition of exposing point this means 
\[\nabla \phifun(\ylp) = \nabla \phifun(\beta).\] 
From \eqref{eq:gradientofbeta}, since $W$ is invertible, we can see that this implies
  $\beta = \frac{\ylp}{\|\ylp\|}$.
  However, as $\lambda^\T a + d^\T \beta \geq 0$, the optimal solution of in the definition of $\phifun(\ylp)$, $x_0$, must satisfy
  $a^\T x_0 + d^\T \ylp = 0$.
  This contradicts $\phifun(\ylp) = \lambda^\T \xlp$, since $\xlp$
  is an optimal solution but $a^\T \xlp + d^\T \ylp = -1$.
\end{proof}
\end{document}